\documentclass[10pt]{article}
\usepackage{latexsym}
\usepackage{amssymb}
\usepackage{amsmath}

\usepackage{amsfonts}
\usepackage{amssymb}
\usepackage{amsthm}
\usepackage{latexsym}
\usepackage{fancyhdr}
\numberwithin{equation}{subsection}
\usepackage{color}

\newtheorem{theorem}{Theorem}[section]

\newtheorem{proposition}[theorem]{Proposition}

\newtheorem{corollary}[theorem]{Corollary}

\newcommand{\X}{{\mathfrak X}}

\renewcommand{\H}{{\mathcal H}}
\def\al{\alpha}
\def\be{\beta}
\def\DE{\Delta}

\def\de{\delta}
\def\rh{\rho}
\def\et{\eta}

\def\ga{\gamma}
\def\GA{\Gamma}

\def\OM{\Omega}
\def\ve{\varepsilon}

\def\la{\lambda}
\def\OM{\Omega}
\def\om{\omega}
\def\va{\varphi}

\def\ta{\tau}

\def\cc{\mathfrak{c}}
\def\g{\mathfrak{g}}

\def\h{\mathfrak{h}}
\def\k{\mathfrak{k}}
\def\q{\mathfrak{q}}
\def\p{\mathfrak{p}}

\def\m{\mathfrak{m}}
\def\ll{\mathfrak{l}}

\def\x{\mathfrak x}

\def\ga{\gamma}
\def\la{\lambda}
\def\ve{\varepsilon}
\def\si{\sigma}
\def\om{\omega}
\def\ep{\epsilon}

\def\ph{\phi}
\def\ch{\chi}
\def\ta{\tau}
\def\ps{\psi}

\def\Om{\Omega}

\def\PH{\Phi}
\def\ol#1{\overline{#1}}
\def\nn{\nonumber}

\def\R{{\mathbb R}}
\def\C{{\mathbb C}}
\def\N{{\mathbb N}}

\def\Z{{\mathbb Z}}

\def\Id{{\mathbb I}}
\def\A{{\mathcal A}}

\def\B{{\mathcal B}}

\def\F{{\mathcal F}}

\def\H{{\mathcal H}}

\def\K{{\mathcal K}}

\def\S{{\mathcal S}}

\def\ZZ{{\mathcal Z}}
\def\XX{{\mathcal X}}

\def\Ad{{\text Ad}}

\def\dis{\text{dis}}
\def\iy{\infty}

\def\wh{\widehat}

\def\ol#1{\overline{#1}}
\def\ul#1{\underline{#1}}

\def\hb#1{\hbox{#1}}
\def\val#1{\vert #1\vert}

\def\no#1{\Vert #1\Vert }

\def\span#1{\text{span}\{#1\}}

\def\ker#1{\hb{ker}(#1)}

\def\res#1{_{\vert #1}}
\def\inv{^{-1}}

\def\es{\emptyset}

\def\hb #1{\hbox{#1}}

\def\hb#1{\hbox{#1}}
\def\val#1{\vert #1\vert}

\def\ti{\times}

\def\ker#1{\hb{ker}(#1)}

\def\noop#1{\Vert #1\Vert_{\rm op}}

\def\ind{{\rm ind} }

\def\L1#1{L^1(#1)}

\def\L#1#2{L^{#1}(#2)}
\def\l#1#2{L^{#1}(#2)}

\def\lef({\left(}
\def\rig){\right)}

\def\ci{\circ}

\newtheorem{definition}[theorem]{Definition}

\newtheorem{remark}[theorem]{Remark}

\def\Ad{\mathrm{\, Ad \,}}

\def\k{\mathfrak k}

\begin{document}
\title{The solvable Lie group $N_{6,28}$: an example of an almost $C_0(\K) $-C*-algebra}
\author{J. Inoue, Y.-F. Lin and J. Ludwig}



\maketitle

\begin{abstract}
Motivated by the description of the C*-algebra of
the affine automorphism group $N_{6,28}$ of the Siegel upper half-plane of degree 2
as an algebra of operator fields defined over the unitary dual $\wh{N_{6,28}}$ of the group, we introduce a family of C*-algebras, which we call almost $C_0(\K)$, and we show that the C*-algebra of the group $N_{6,28}$ belongs to this class.
\end{abstract}

\section{Introduction}\label{intro}

In order to analyze  a C*-algebra $ \A $, one can use
the Fourier transform $ \F $, which allows us to decompose $ \A $ over its
unitary spectrum $ \widehat \A $. To be able to define this transform, consider the
algebra $ l^\iy(\widehat \A) $ of all bounded operator fields over $ \widehat \A $ defined by
\begin{eqnarray}\label{}
 \nn  l^\iy(\widehat \A):=\{A=(A(\pi)\in\B(\H_\pi))_{\pi\in\widehat \A}; \no A_\iy
:=\sup_{\pi\in \widehat \A}\noop {A(\pi)}<\iy\},
\end{eqnarray}
where $ \H_\pi $ is the Hilbert space of $\pi$.
The space $ l^\iy(\widehat \A) $ is a (huge) C*-algebra itself. The Fourier
transform $ \F $ defined by
\begin{eqnarray}\label{}
\nn \F(a)=\hat a:=(\pi(a))_{\pi\in\widehat \A} \quad \text{for} \quad a\in \A
\end{eqnarray}
is then an injective, hence isometric, homomorphism from $ \A $ into $ l^\iy(\widehat
\A) $. The problem is now to recognize the elements of $ \F(\A) $ inside this big
algebra $ l^\iy(\widehat \A) $.

Recall that a Lie group $ G $ is called \textit{exponential} if it is a connected, simply connected solvable Lie group for which the
exponential mapping $\text{exp}:\g \to G$ from the Lie algebra $ \g $ to its
Lie group $ G $ is a diffeomorphism.  The
Kirillov-Bernat-Vergne-Pukanszky-Ludwig-Leptin theory shows that there is a canonical
homeomorphism $ K:\g^*/G \to \widehat G $ from the space of coadjoint orbits of $ G
$ in the linear dual space $ \g^* $ onto the unitary dual space $ \widehat G $ of $ G $
(see \cite{Lep-Lud} for details and references).
Then the unitary spectrum  $ \widehat{C^*(G)} $ of the C*-algebra $ C^*(G) $ of the locally compact group $G$ can
be identified with the unitary dual $ \widehat G $ of $ G $.

Since connected Lie groups are second countable, the algebra $ C^*(G) $ and its
dual space $ \widehat G $ are separable topological spaces. This allows us to
work in $ \widehat G $ with sequences instead of nets.
Moreover, if $ G $ is amenable, then the left regular representation $ (\la,L^2(G)) $ of $
C^*(G) $, defined by $ \la(F)\xi:=F\ast \xi$ for $F\in L^1(G)$ and $\xi\in L^2(G) $, is
injective. So we can also identify $ C^*(G) $ with an algebra of convolution
operators on the Hilbert space $ L^2(G) $.

The method of describing group C*-algebras as algebras of operator fields defined on the dual spaces of the groups has been studied in \cite{Fe} and \cite{Lee}. In \cite{Lin-Lud1}, the C*-algebra of $ax+b$-like groups has been characterized as the algebra of operator fields, while the C*-algebra of the Heisenberg groups and of the threadlike groups are described in \cite{Lu-Tu}. In both cases, the topologies of the orbit spaces $\g^*/G$ of the groups are well understood (see for instance \cite{Ar-Lu-Sc}). The isomorphism problem for C*-algebras of $ax+b$-like groups was solved in \cite{Lin-Lud2}.

In this paper, we consider the Lie group $G_6$, whose Lie algebra is the 6-dimensional normal $j$-algebra $ \g_6 $ (which has been classified as $N_{6,28} $ by P. Turkowski in \cite{Tu}); a prototype of this class of Lie algebras can be found, for example,
 in \cite{Ino}. As a transformation group, our $G_6$ is the Iwasawa AN-group of $Sp(2,\R)$.
For a geometric background of this class of Lie groups as affine automorphisms of
Siegel domains, we refer the reader, for example, to the textbook \cite{Kaneyuki}.
Thanks to the orbit structure of exponential groups, we can write down the dual space $\wh{G_6}$ of our group $G_6$ and determine its topology. The main difference here from the previous cases, the $ax+b$-like groups,  is that there are now coadjoint orbits of dimension 0, 2, 4 and 6, respectively, which decompose the orbit space into the union of a sequence $ (S_i)_{i=0}^6 $ of seven increasing closed subsets. On each of the sets $ \GA_i=S_i\setminus S_{i-1}, i=1,\ldots, 6 $, the orbit space topology is Hausdorff and the main difficulty is to understand for a given $ a\in C^*(G) $ the behaviour of the operators $ \hat a(\ga)$ for $ \ga\in \GA_i $, when $ \ga $ approaches elements in $ S_{i-1} $.  For each of these  sets $ \GA_i $, we obtain different conditions for the C*-algebra, conditions which we
shall call the \emph{continuity}, the \emph{infinity} and the \emph{boundary} conditions.

Our example motivates the introduction of a special class of C*-algebras which we call \textit{almost $C_0( \K)$,} where $\K$ is the algebra of all compact operators on a certain Hilbert space (Section \ref{spc}). In Section \ref{gsix}, we describe the  exponential Lie group $G_6$, its Lie algebra $\g_6$ and its coadjoint orbits in $\g_6^*$.  Each of them needs a special treatment which we describe in the following sections. In Section \ref{topwhG}, we present  the topology of the dual space $\widehat G_6$ of $G_6$, i.e. we determine the boundaries of each orbit and we compute the limit sets of properly converging sequences of coadjoint orbits. In Section \ref{contcon}, we discover the continuity and infinity conditions and in Section \ref{bound},  the most intricate one,  we introduce the 6 different regions $ \GA_i $ of the dual space of $ G_6 $ according to the dimensions of the coadjoint orbits and obtain the boundary conditions for each of these regions. In the last section (Section \ref{cstarGdef}),
we describe the actual C*-algebra of $G_6$ as an algebra of operator fields and we see that this C*-algebra has the structure of an almost $ C_0(\K)$-C*-algebra.

\section{A special class of C*-algebras}\label{spc}

\begin{definition}\label{cbga}
\rm   Let $\GA$ be a topological Hausdorff space, let $\H$ be a Hilbert space and denote by $\K$ or $\K(\H)$ the algebra of compact operators on $\H$.
Let $C_0(\GA,\K)$ be the space of all continuous mappings $\va:\GA\mapsto \K(\H)$ vanishing at infinity. Equipped with the norm
\begin{eqnarray*}
 \no \va_\iy:=\sup_{\ga\in\GA}\noop{\va(\ga)},\quad \va\in C_0(\GA,\K),
\end{eqnarray*}
the space $C_0(\GA,\K)$ is a C*-algebra, whose spectrum is homeomorphic to the topological space $\GA$.
 \end{definition}

Let now $ A $ be a separable C*-algebra and $\widehat A$ be the unitary dual of $A$.

\begin{definition}\label{}
\rm
We suppose that there exists a
finite increasing family $ S_0\subset S_1\subset\ldots\subset S_d=\wh A  $ of
closed
subsets of the spectrum $ \wh A $ of $ A $ such that for $ i=1,\cdots, d, $ the
subsets $\GA_0=S_0$ and  $ \GA_i:=S_i\setminus S_{i-1} $ are Hausdorff in
their relative topologies. Furthermore we assume that
for every $ i\in \{0,\cdots, d\} $ there exists a Hilbert space $ \H_i $ and
a concrete realization $ (\pi_\ga,\H_i) $ of $ \ga $ on
the Hilbert space $ \H_i $ for every $ \ga\in \GA_i $.
We also assume that the set $ S_0 $ is the collection $ \X $ of all characters of $ A $.
\end{definition}

\begin{definition}\label{fourtrans}
\rm   We define the Fourier transform $ \F:A\to l^\iy(\wh A) $ as to be the
mapping:
\begin{eqnarray*}
\F(a)(\ga)=\wh a(\ga):=\pi_\ga(a) \quad \text{for} \quad \ga\in \wh A, a\in A.
\end{eqnarray*}
 \end{definition}
For a subset $ S\subset\wh A $, denote by $ CB(S) $ the $ * $-algebra of
all uniformly bounded operator fields  $ (\ps(\ga)\in \B(\H_i))_{\ga\in S\cap
\GA_i, i=1,\cdots, d}$, which are operator norm continuous on the subsets $
\GA_i \cap S$ for every $  i\in\{1,\cdots, d\} $ for which  $ \GA_i\cap S\ne\es$.
We provide the algebra $ CB(S) $ with the infinity-norm:
\begin{eqnarray*}
\no{\ps}_{S}:=\sup_{\ga\in S}\noop{\ps(\ga)}.
\end{eqnarray*}

\begin{definition}\label{norcontspec}
\rm We say that a C*-algebra $ A $ is ``\textit{almost  $ C_0(\K) $}''
if for every $ a\in A $:
\begin{enumerate}\label{}
\item  The mappings $ \ga \mapsto \F (a)(\ga) $ are norm-continuous on the different
sets $ \GA_i $.
(We remark that for every closed subset $ S $ of $ \wh A $ and every $ a\in A $,
the restriction $ \F(a)\res S $ of the operator field $ \F (a) $ to $ S $ is then contained in $
CB(S) $).

\item For any $ i=1,\cdots,d$, we have a sequence  $
(\si_{i,k}:CB(S_{i-1})\to CB(S_i))_k $ of linear mappings which are
uniformly bounded in $ k $ such that
$$
\lim_{k\to\iy}\dis\Big((\si_{i,k}(\F (a)\res{S_{i-1}})-\F
(a)\res{\GA_i}),C_0(\GA_i,\K(\H_i))\Big)=0
$$
and such that
$$
\lim_{k\to\iy}\dis\Big((\si_{i,k}(\F (a)\res{S_{i-1}})^*-\F
(a^*)\res{\GA_i}),C_0(\GA_i,\K(\H_i))\Big)=0.
$$
(This condition justifies the name of ``almost $C_0(\K)$'').
\end{enumerate}
\end{definition}

\begin{definition}\label{dnormconspec}
\rm
 Let $ D^*(A) $ be the set of all operator fields $ \va $ defined over $ \wh A $
such that
\begin{enumerate}\label{}

\item The field $ \va $ is uniformly bounded, i.e. $
\no\va:=\sup_{\ga\in\wh A}\noop{\va(\ga)}<\iy $.
\item $ \va\res{\GA_i}\in CB(\GA_i) $ for every $ i=0,1, \ldots, d $.

\item For any sequence $ (\ga_k)_{k\in\N} $ going to
infinity in $\wh A$, we have  $ \lim_{k\to\iy}\noop{\va(\ga_k)}=0 $.

\item We have
\begin{eqnarray*}
\lim_{k\to\iy}\dis\Big((\si_{i,k}(\va\res{S_{i-1}})-\va\res{\GA_i}),C_0(\GA_i,\K(\H_i))\Big)=0
\end{eqnarray*}
and
\begin{eqnarray*}
\lim_{k\to\iy}\dis\Big((\si_{i,k}(\va\res{S_{i-1}})^*-(\va\res{\GA_i})^*),C_0(\GA_i,\K(\H_i))\Big)=0.
\end{eqnarray*}
 \end{enumerate}
\end{definition}

We see immediately that for every $ a\in A $, the operator field
$ \F(a) $
is contained in the set $ D^*(A) $. In fact it turns out that $ D^*(A) $ is  a
C*-subalgebra of $ l^\iy(\wh A) $ and that $ A $ is isomorphic to $ D^*(A) $.

\begin{theorem}\label{aisdsta}
Let $ A $ be a separable C*-algebra which is almost $ C_0(\K) $. Then
the subset $ D^*(A) $ of the C*-algebra $l^\iy(\wh A)  $ is a C*-subalgebra
which is isomorphic to $ A $ under the Fourier transform.
\end{theorem}

\begin{proof}
We remark that the conditions (1) to (4) imply that $ D^*(A) $ is a closed
involution-invariant subspace of $ l^\iy(\wh A) $.

For $ i=0,\cdots, d,  $ let  $ D^*_i $ be the set of all operator fields defined
over $ S_i $ satisfying conditions (1) to (4) on the sets  $ S_j$ for $j= 1,\cdots, i $.
Then the sets $ D^*_i $ are closed subspaces of the C*-algebra $l^\iy(S_i) $.

Let $ I_C $ be the closed two-sided ideal in $ A $ generated by the elements of
the form $ ab-ba, a,b\in A $.
Then the space of characters $ S_0=\X $ of $ A $ is the spectrum of $ A/I_C $
and  $ D^*_0$ equals the algebra  $ C_0(S_0) $ of continuous functions on $ S_0
$ vanishing at infinity by the conditions (1) and (2) Since $
\F(A)\res{S_0}=C_0(S_0) $  it follows that $D_0^*=\F(A)\res{S_0}$.

Let us assume now that for some $ 1\leq i <d$, we have that
$ D^*_j=\F(A)\res{S_j}$ for $j=0,\ldots, i-1 $.  We shall prove then that $
D^*_i=\F(A)\res{S_i} $. We know already that $ \F(A)\res{S_i} $ is a subalgebra
of the closed subspace  $ D^*_i \subset CB(\S_i)$ and it follows from its
definition that  the restriction of $ D_i^* $ to $ S_{i-1} $ is contained in $
D^*_{i-1} $. Hence
\begin{eqnarray*}
\F(A)\res{S_{i-1}}\subset D^*(A)\res{S_{i-1}}\subset D^*_{i-1}= \F(A)\res{S_{i-1}}.
\end{eqnarray*}
Therefore
$ D^*(A)\res{S_{i-1}}=\F(A)\res{S_{i-1}}.$
Let $ \ph_1,\ph_2\in D^*_i $.  By our assumption, there exists $ a_1,a_2\in A $ such
that $ \ph_j{\res{S_{i-1}}}=\hat a_j{\res{S_{i-1}}}, j=1,2 $.
The product $ \ph_1\circ \ph_2$ satisfies also the conditions (1), (2) and (3).
We shall show that it also satisfies the condition (4). We now have that
\begin{eqnarray*}
 &&\ph_1\circ \ph_2{\res{\GA_i}}-\si_{i,k}(\ph_1\circ \ph_2{\res{S_{i-1}}})\\
&=& \ph_1{\res{\GA_i}} \circ\ph_2{\res{\GA_i}}-\si_{i,k}(\hat a_1{\res{S_{i-1}}}\circ \hat a_2{\res{S_{i-1}}})\\
&=&\ph_1{\res{\GA_i}} \circ\ph_2{\res{\GA_i}}-\hat a_1{\res{\GA_{i}}}\circ \hat a_2{\res{\GA_{i}}}\\
& &+ (\hat a_1{\res{\GA_{i}}}\circ \hat a_2{\res{\GA_{i}}})-\si_{i,k}(\hat a_1{\res{S_{i-1}}}\circ \hat a_2{\res{S_{i-1}}}).
\end{eqnarray*}
Now
\begin{eqnarray*}
&& \lim_{\de\to 0}\dis\big( (\hat a_1{\res{\GA_{i}}}\circ \hat a_2{\res{\GA_{i}}}-\si_{i,k}(\hat a_1{\res{S_{i-1}}}\circ \hat a_2{\res{\S_{i-1}}})),C_0(\GA_i,\K(\H_i))\big)\\
&=& \lim_{\de\to 0}\dis\big( (\widehat {a_1a_2}{\res{\GA_{i}}}-\si_{i,k}(\widehat{ a_1a_2}{\res{S_{i-1}}})),C_0(\GA_i,\K(\H_i))\big)\\
&=& 0
\end{eqnarray*}
%
%
by condition (2) in Definition \ref{norcontspec}.
Furthermore,
\begin{eqnarray*}
 &&\ph_1{\res{\GA_i}} \circ\ph_2{\res{\GA_i}}-\hat a_1{\res{\GA_{i}}}\circ \hat a_2{\res{\GA_{i}}}\\
&=&\ph_1{\res{\GA_i}} \circ(\ph_2{\res{\GA_i}}-\hat a_2{\res{\GA_{i}}})
+(\ph_1{\res{\GA_i}}-\hat a_1{\res {\GA_{i}}}) \circ\hat a_2{\res{\GA_{i}}}\\
&=&\ph_1{\res{\GA_i}} \circ(\ph_2{\res{\GA_i}}-\si_{i,k}(\ph_2{\res {S_{i-1}}}))
+\ph_1{\res{\GA_i}} \circ(\si_{i,k}(\hat a_2{\res {S_{i-1}}})-\hat a_2{\res{\GA_i}})\\
&&+(\ph_1{\res{\GA_i}}-\si_{i,k}(\ph_1{\res {S_{i-1}}})) \circ\hat a_2{\res{\GA_{i}}}
+(\si_{i,k}(\hat a_1{\res {S_{i-1}}})-\hat a_1{\res {\GA_{i}}}) \circ\hat a_2{\res{\GA_{i}}}.
\end{eqnarray*}
%
%
Hence,
\begin{eqnarray*}
 \lim_{\de\to 0}\dis(\ph_1{\res{\GA_i}}\circ \ph_2{\res{\GA_i}}-\si_{i,k}(\ph_1\circ \ph_2{\res{S_{i-1}}}),C_0(\GA_i,\K(\H_i)))= 0.
\end{eqnarray*}
%
%
Therefore $ D^*_i $ is an algebra, i.e. a C*-subalgebra of $ CB(S_i) $. The conditions (1) and (2) for the C*-algebra $ A $ tell us that the ideal $ \ker{S_{i-1}}=\{a\in A: \hat a\res{S_{i-1}}=0\} $  of $ A $ is isomorphic with the algebra $ C_0(\GA_i, \K(\H_i))\subset  \F(A) $ of all continuous mappings defined on $ \GA_i $ with values in $ \K(\H_i) $ and vanishing at infinity. Indeed, for any $ \ga\ne\ga'\in S_i $, there exists an $ a\in A $ such that $ \ga(a)\ne 0 $, but $ \ga'(a)=0 $ and also $ \hat a\res{S_{i-1}}=0 $. Then $ \hat a\res{\GA_i} $ is  in $ C_0(\GA_i,\K(\H_i)) $ by the condition (1),(2) and (3) and in this way we see that $\ker{S_{i-1}}  $ separates the points in $ \GA_i $ and that $ \widehat{\ker{S_{i-1}}}\res{\GA_i}\subset C_0(\GA_i,\K(\H_i))$. The   theorem of Stone-Weierstrass says now that $ \widehat{\ker{S_{i-1}}}\res{\GA_i}=C_0(\GA_i,\K(\H_i)) $.

Let now $ \pi $ be an element of the spectrum of $D^*_i $. Let $ R_{i-1}:D^*_i \to D^*_{i-1} $ be the restriction map and
let $ I_{i-1}:= \ker{R_{i-1}}\subset D^*_i$.
Suppose that $ \pi(I_{i-1})=\{0\} $. We can then consider $ \pi $ as a representation of the algebra $D^*_i/I_{i-1}$. This algebra is isomorphic with the  image of $ R_{i-1}  $ which is itself isomorphic with $ \F(A)\res{S_{i-1}}$ and therefore $ \pi $ is an element of $ S_{i-1} $.
Suppose that  $ \pi $ is not trivial on $ I_{i-1} $.  By the condition (4) this kernel  $I_{i-1} $ is contained in  $ C_0(\GA_i,\K) $ and since it contains $ \ker{R_{i-1}}\cap \F(A) $, it is itself isomorphic to $ C_0(\GA_i) $. Therefore $ \pi $ is an evaluation at some point in $ \GA_{i} $.

We have seen that the spectrum of $ D^*_i $ coincides with the spectrum of the subalgebra  $ \F(A)\res{S_i} $.  Hence by
the theorem of Stone-Weierstrass (see \cite{Di}), the  C*-algebra $\F(A)\res{S_i}$ and $D^*_i$ coincide.
\end{proof}

\section{The group $ G_6 $}\label{gsix}

This paper aims to show that the C*-algebra of the exponential Lie group $G_6=\exp(\g_6)$ is in the class of almost $C_0(\K)$-C*-algebras.

Let $\g= \g_6 $ be the Lie algebra spanned by the vectors $ A,B,P,Q,R,S $ and equipped with the Lie brackets:

\begin{equation*}
[P, Q]= R,
\qquad
[P, R]= S,
\qquad
[A, P]= \frac12 P,
\qquad
[B, P]=-\frac12 P,
\end{equation*}
\begin{equation*}
[B, Q]=Q,
\qquad
[A, R]=\frac12 R,
\qquad
[B, R]=\frac12 R,
\qquad
[A, S]=S,
\end{equation*}
\begin{equation*}
[A, Q]=0,
\qquad
[B, S]=0.
\end{equation*}
We introduce the following closed subgroups and coordinate functions on $ G $, write
\begin{eqnarray}\label{}
 \nn \g_0:= \span{A,B,P},\ G_0:=\exp {(\g_0)},\\
 \nn \h:= \span{Q,R,S},\ H:=\exp{( \h)}.\\
\end{eqnarray}
Then the subset $ G_0 $ is a closed subgroup of $ G $, and the subset $ H $ is a closed normal subgroup. We have
\begin{eqnarray}\label{}
 \nn G &= &G_0\cdot H 
\end{eqnarray}
as a topological product.
For $(a,b,p,q,r,s)\in\mathbb R^6$, let
\begin{equation*}
g=E(a,b,p,q,r,s):=
\exp( aA)\exp( bB)\exp( pP)\exp (qQ)\exp (rR)\exp (sS).
\end{equation*}

\subsection{A parametrization of the set of
coadjoint orbits}\label{desor}

We give in this section a system of representatives  of the set of coadjoint orbits $\g^*/\Ad^*G$.

Let $\{A^*, B^*, P^*, Q^*, R^*, S^*\}$ be the dual basis of $\{A, B, P, Q, R, S \}$, and
$f={a}^*A^*+{b}^*B^*+{p}^*P^*+{q}^*Q^*+
{r}^*R^*+{s}^*S\in\g^*$. Then
\begin{eqnarray}\label{coadG}
\Ad^* (g)(f)&=&
\left(a^*+\frac{p}{2}p^*+\frac12(pq+r)r^*+\left(s+\frac{pr}{2}\right)s^*\right)A^*\\
 \nn& &+  \left(b^*-\frac{p}{2}p^*+qq^*+\frac{1}{2}(r-pq)r^*-\frac{pr}{2}s^*\right)B^*\\
 \nn& &+  e^{\frac{-a+b}{2}}(p^*+qr^*+ rs^*)P^*\\
 \nn& &+  e^{-b}\left(q^*-pr^*+\frac12 p^2s^*\right) Q^*
 +  e^{\frac{-a-b}{2}}(r^* -ps^*) R^*
 +  e^{-a}s^*S^*.
\end{eqnarray}

In the sequel, we write
$\mathbb R^+:=\{x\in\mathbb R;\ x>0\}$ and $\mathbb R^{+,0}:=\mathbb R^+\cup\{0\}$.
We note that the determinant of the matrix
$\displaystyle{(f([V,W]))_{V,W\in\{A,B,P,Q,R,S\}}}$
is $\displaystyle{\frac14 {s^*}^2(2q^*s^*-{r^*}^2)^2}$,
and there are open orbits. The following list gives a parametrization of the coadjoint
orbits $\g^*/\Ad^*G$.
\begin{description}
\item[$0$-dimensional orbits]
$a^*A^*+b^*B^*$, where $a^*, b^*\in\mathbb R$.
\item[$2$-dimensional orbits] \
\begin{itemize}
\item [{\bf (2d-1)}]
$\alpha^*(\frac{A^*+B^*}{2})+ \varepsilon P^*$, where $\alpha^*\in\mathbb R,$ $\varepsilon=\pm 1$,
\begin{eqnarray*}
  \OM_{\alpha^*(\frac{A^*+B^*}{2})+ \varepsilon P^*}=\{
(\frac{\alpha^*}{2}+\frac{p}{2}\varepsilon)A^*
 + (\frac{\alpha^*}{2}-\frac{p}{2} \varepsilon)B^*
 + e^{-\alpha}\ve P^*, \alpha,p\in\R\}.
 \end{eqnarray*}
[ see \eqref{2d-1-1}, \eqref{2d-1-2} below]
\item [{\bf (2d-2)}]
$a^*A^*+ \nu Q^*$, where $a^*\in\mathbb R$, $\nu=\pm1$,
\begin{eqnarray*}
  \OM_{a^* A^*+\nu Q^*}=\{
a^*A^* + q\nu B^* +  e^{-b}\nu Q^*, b,q\in\R\}.
 \end{eqnarray*}
[ see 
      \eqref{2d-2-2}]
\end{itemize}
\item[$4$-dimensional orbits] \
\begin{itemize}
\item [{\bf (4d-1)}]
$\varepsilon P^*+\nu Q^*$, where
$\varepsilon=\pm1, \nu =\pm 1$,
\begin{eqnarray*}
  \OM_{\ve P^*+\nu Q^*}=\{
p A^* + q B^* + e^{\frac{-a+b}{2}}(\ve)P^*
 +  e^{-b}\nu Q^*, p,q,a,b\in\R\}.
 \end{eqnarray*}
[ see 
      \eqref{4d-1-2} ]
\item [{\bf (4d-2)}]
$b^*B^*+\varepsilon R^*$, where
$b^*\in\mathbb R$, $\varepsilon=\pm1$,
\begin{eqnarray*}
 \OM_{b^* B^*+\ve R^*}=\{
rA^*
 + \left(b^*+r-pq\ve\right)B^*
 + qP^*
 + (-p e^{-a}) Q^*
 +   e^{-a} \ve R^*,
 a,p,q,r \in\R\}.
 \end{eqnarray*}
[ see \eqref{4d-2-1}, \eqref{4d-2-2} ]
\item [{\bf(4d-3)}]
$b^*B^*+\varepsilon S^*$, where
$b^*\in\mathbb R$, $\varepsilon=\pm1$,
\begin{eqnarray*}\label{}
\OM_{b^*B^*+\varepsilon S^*}&=&\{
sA^*
 + \left(b^*-\frac{pr}{2}\ve\right)B^*
 + e^{\frac{-a}{2}}(r\ve)P^*
 +  \left(\frac{1}{2} p^2\ve \right) Q^*
 +  e^{\frac{-a}{2}}( -p\ve ) R^*
 +  e^{-a}\ve S^*,\\
 &&s,p,r,a\in\R\}.
\end{eqnarray*}
[ see \eqref{4d-3-1}, \eqref{4d-3-2} ]
\end{itemize}

\item[$6$-dimensional orbits]
\begin{itemize}
\item[{\bf (6d)}]
$\nu Q^*+\varepsilon S^*$, where $\nu=\pm1$, $\varepsilon=\pm1$,
\begin{eqnarray*}
\OM_{\nu Q^*+\varepsilon S^*}&=&\{s A^*+r B^*+q P^*+  e^{-b}(\nu+\frac{p^{2}\ve}{2})Q^*+(-e^{-\frac{a+b}{2}}p\ve)R^*+(e^{-a}\ve)S^*,\\
&& s,r,q, a,b,p\in\R\}.
 \end{eqnarray*}
[ see \eqref{6d1},\eqref{6d2}]
\end{itemize}
\end{description}

\paragraph{Proof and description of each orbit.}
Let $\Omega$ be a coadjoint orbit and $f \in \Omega$.

\begin{description}
\item[Case I)]
Suppose $f(S)\neq 0$.
Then by \eqref{coadG}, there exists  $g=E(a,0,p,0,r,s)$ such  that
$g\cdot f=b^*B^*+q^*Q^*+\varepsilon S^*$,
where $\varepsilon=\pm 1$ and $b^*, q^*\in\mathbb R$.
Thus let $f=b^*B^*+q^*Q^*+\varepsilon S^*$.
\begin{itemize}
\item[I-1)]
If $q^*\neq 0$, then there exists $g=E(0,b,0,q,0,0)$
such that $\Ad^*g(f)
=(b^*+qq^*)B^*+e^{-b}q^*Q^*+\varepsilon S^*=\nu Q^*+\varepsilon S^*$,
where $\nu=\pm1$.
\begin{eqnarray}\label{6d1}
f: & =&\nu Q^*+\varepsilon S^*,\\
\nn\Ad^*E(a,b,p,q,r,s)(f)
&=&\bar aA^*+\bar bB^*+\bar pP^*+\bar qQ^*+\bar rR^*+\bar s S^*,
\end{eqnarray}
where
\begin{eqnarray*}
&\bar a&=  \left(s+\frac12 pr\right)\varepsilon ,
\quad
\bar b=  \left( q\nu -\frac12 pr \varepsilon\right),
\quad
\bar p=  e^{\frac{-a+b}{2}}r\varepsilon ,
\\
&\bar q&=  e^{-b}\left(\nu +\frac12 p^2\varepsilon\right),
\quad
\bar r=  -e^{\frac{-a-b}{2}}p\varepsilon ,
\quad
\bar s=  e^{-a}\varepsilon .
\end{eqnarray*}
Let
$\Phi(l):=2l(Q)l(S)-l(R)^2$ for $l\in\g^*$. Then $\Phi(\bar f)=
2\bar q\bar s-\bar r^2=2e^{-a-b}\nu\varepsilon\neq 0$,
and  we have
\begin{equation}\label{6d2}
\Omega_f=\{l ;\ \Phi(l)\neq 0, \ l(S)\neq 0\}.
\end{equation}

\item[I-2)]
If $q^*=0$, then $f=b^*B^*+\varepsilon S^*$ and
$\g(f)=\mathbb R\mbox{-span}\{B, Q\}$.

\begin{equation}\label{4d-3-1}
\Ad^*E(a,0,p,0,r,s)(f)=\bar aA^*+\bar bB^*+\bar pP^*
+\bar qQ^*+\bar rR^*+\bar sS^*,
\end{equation}
where
\begin{eqnarray*}
& \bar a& =  \left(s+\frac{pr}{2}\right) \varepsilon,
\quad \bar b =\left(b^*-\frac{pr}{2}\varepsilon\right),
\quad \bar p= re^{-\frac{a}{2}}\varepsilon,\\
& \bar q& =\frac12 p^2\varepsilon,
\quad \bar r = -pe^{-\frac{a}{2}}\varepsilon,
\quad \bar s = e^{-a}\varepsilon.
\end{eqnarray*}
Let $\Phi, \Phi_2$ be functions on
$\g^*$ defined by
$$\Phi(l)
=2l(Q)l(S)-l(R)^2,
\quad
\Phi_2(l)
=l(R)l(P)-2(l(B)-b^*)l(S).$$

Then we have
\begin{eqnarray}\label{4d-3-2}
\nn\Omega_f
=\Phi^{-1}(0)\cap\Phi_2^{-1}(0)\cap
(\mathbb RA^*+\mathbb RB^*+\mathbb RP^*+(\mathbb R^{+,0})\varepsilon Q^*
+\mathbb RR^*+\mathbb R^+\varepsilon S^*).\\
&&
\end{eqnarray}
In fact, let $\bar f:=(\bar a,\bar b,\bar p,\bar q, \bar r,\bar s)$ be
an element of the right hand side. We take
\begin{eqnarray*}
a:=  -\log(\varepsilon \bar s), \quad
p:=  -\bar r e^{\frac a2}\varepsilon, \quad
r:=  \bar p e^{-\frac a2}\varepsilon, \quad
s:=  \bar a-\frac{pr}{2}.
\end{eqnarray*}
Then we have
$\Ad^*E(a,0,p,0,r,s)f=\bar f$, this is, $\bar f\in\Omega_f$.
The reverse inclusion is easily verified.

\end{itemize}

\item[Case II)]
Suppose $f(S)=0$. Then $\Ad^*G(f)(S)=\{0\}$.
Then \eqref{coadG}
is reduced to
\begin{eqnarray}\label{coad-s0}
\Ad^* g(f)&= &
\left(a^*+\frac{p}{2}p^*+\frac12(pq+r)r^*\right)A^*\\
\nn& &+ \left(b^*-\frac{p}{2}p^*+qq^*+\frac{1}{2}(r-pq)r^*\right)B^*\\
\nn& &+ e^{\frac{-a+b}{2}}(p^*+qr^*)P^*
 + e^{-b}\left(q^*-pr^*\right) Q^*
 + e^{\frac{-a-b}{2}}r^* R^*.
\end{eqnarray}

\begin{itemize}
\item[II-1)]
Suppose $f(R)=r^*\neq 0$. Then by \eqref{coad-s0},
there exists
$g=E(\alpha, \alpha, p,q,r,0)$
such that $\Ad^*g(f)=b^*B^*+\varepsilon R^*$, where $\varepsilon=\pm 1$,
$b^*\in\mathbb R$.
Let $f=b^*B^*+\varepsilon R^*$. Then $\g(f)=\mathbb R\mbox{-span}\{S, A-B\}$.
Let $X=A+B$, $Y=A-B$, $X^*=\frac 12(A^*+B^*)$, $Y^*=\frac 12(A^*-B^*)$.
Then $\{X^*, Y^*, P^*, Q^*, R^*, S^*\}$ is the dual basis of
$\{X,Y, P, Q, R, S\}$.

\begin{eqnarray}\label{4d-2-1}
f: &=&b^*B^*+\varepsilon R^*,\\
\nn\Ad^*E(\alpha, \alpha, p,q,r,0)(f)
&= &\frac12(pq+r)\varepsilon A^*
+  \left(b^*+\frac{1}{2}(r-pq)\varepsilon\right)B^*\\
\nn&  & +  q\varepsilon P^*
-  e^{-\alpha}p\varepsilon Q^*
+  e^{-\alpha}\varepsilon  R^*
\\
\nn&=& (b^*+r\varepsilon)X^* +(-b^*+pq\varepsilon)Y^*
+  q\varepsilon P^*
-  e^{-\alpha}p\varepsilon Q^*
+  e^{-\alpha}\varepsilon  R^*,
\end{eqnarray}
\begin{eqnarray}\label{4d-2-2}
\Omega_f= \{\bar xX^*+\bar yY^*+\bar pP^*+\bar qQ^*+\bar rR^* ;\
\bar r\in\varepsilon\mathbb R^+, \bar y=-b^*-(\bar p\bar q/\bar r)\}.
\end{eqnarray}

\item[II-2)]
Suppose $f(R)=0$. Then $G\cdot f(R)=\{0\}$, and we have
\begin{eqnarray*}\label{coad-r0-s0}
\nn\Ad^* g(f)=
\left(a^*+\frac{p}{2}p^*\right)A^*
+  \left(b^*-\frac{p}{2}p^*+qq^*\right)B^*
+  e^{\frac{-a+b}{2}} p^* P^*
+  e^{-b} q^* Q^*.
\end{eqnarray*}

\begin{itemize}
\item[II-2-i)]
Suppose $q^*\neq 0$ and $p^*\neq 0$. Then by
\eqref{coad-r0-s0}, there exists $g=E(a,b,p,q,0,0)$ such that
$\Ad^*g(f)=\varepsilon P^*+\nu Q^*$, where
$\varepsilon=\pm1, \nu =\pm 1$.
We have $\g(\varepsilon P^*+\nu Q^*)=\mathbb R\mbox{-span}\{S, R\}$, and
\begin{eqnarray*}\label{4d-1-1}
f &:=&  \varepsilon P^*+\nu Q^*,\\
\Ad^*E(a,b,p,q,0,0)(f) &=&
\frac{p}{2}\varepsilon A^*
+ \left(-\frac{p}{2}\varepsilon +q \nu \right)B^*
+  e^{\frac{-a+b}{2}} \varepsilon P^*
+  e^{-b} \nu Q^*,
\end{eqnarray*}
\begin{equation}\label{4d-1-2}
\Omega_f=  \mathbb RA^*+\mathbb RB^*+ \mathbb R^+\varepsilon P^*
+\mathbb R^+\nu Q^*.
\end{equation}

\item[II-2-ii)]
Suppose $q^*\neq 0$ and $p^*=0$.
Then there exists $g=E(0,b,0,q,0,0)$ such that
$g\cdot f = a^*A^*+ \nu Q^*$. We have
$\g(a^*A^*+\nu Q^*)=\mathbb R\mbox{-span}\{A, P, R, S\}$,
\begin{eqnarray*}\label{2d-2-1}
f & :=& a^*A^*+\nu Q^*,\\
\Ad^*E(0,b,0,q,0,0)(f)
 & = & a^*A^*+ q\nu B^* + e^{-b}\nu Q^*,
\end{eqnarray*}
\begin{equation}\label{2d-2-2}
\Omega_f  =a^*A^*+\mathbb R B^*+ \mathbb R^+\nu Q^*.
\end{equation}

\item[II-2-iii)]
Suppose $q^*=0$ and $p^*\neq 0$.
Then there exists $g=\exp(\alpha(A-B))\exp (pP)$ such that
$g\cdot f= \alpha^*(\frac{A^*+B^*}{2})+ \varepsilon P^*$,
where $\alpha ^*\in\mathbb R$, $\varepsilon=\pm1$.
Let $f=\alpha ^*(\frac{A^*+B^*}{2})+ \varepsilon P^*$. Then
we have
$\g(f)=\mathbb R\mbox{-span}\{ A+B, R, S, Q\}$.

Letting $X=A+B$, $Y=A-B$, $X^*=\frac{A^*+B^*}{2}$, $Y^*=\frac{A^*-B^*}{2}$,
we have
\begin{eqnarray}\label{2d-1-1}
f :=\alpha^*(\frac{A^*+B^*}{2})+ \varepsilon P^*=\alpha^*X^*+\varepsilon P^*,
\end{eqnarray}
\begin{eqnarray*}
\Ad^*E(\alpha,-\alpha, p,0, 0, 0)(f)
\nn&=& \left(\frac{\alpha^*}{2}+\frac p2 \varepsilon \right)A^*+
\left(\frac{\alpha^*}{2}-\frac p2 \varepsilon \right)B^*
+e^{-\alpha}\varepsilon P^*\\
\nn&=& \alpha^*X^*+ p\varepsilon Y^*+e^{-\alpha}\varepsilon P^*,
\end{eqnarray*}
\begin{equation}\label{2d-1-2}
\Omega_f  =\alpha^*X^*+\mathbb RY^*+\mathbb R^+\varepsilon P^*.
\end{equation}
\item[II-2-iv)]
Suppose $q^*=p^*=0$.
Then $G\cdot f=\{a^*A^*+b^*B^*\}$ which is a point.
\qed

\end{itemize}
\end{itemize}

\end{description}

\section{The topology of $ \wh G $}\label{topwhG}
\subsection{The closure of the different orbits and their corresponding irreducible representations}

We shall need in our description of $ C^*(G) $ the following criterion for the compactness of an operator $ \pi(F), F\in A, $ where $ \pi $ is an irreducible representation of a type I C*-algebra $ A $.

\begin{theorem}\label{comdfourth}
The operator $ \pi(F), F\in A $, is compact if and only if $ \pi'(F)=0 $ for every $ \pi' $ in the boundary of $ \pi $.
\end{theorem}

Since the topology of the spectrum  $ \wh {C^*(G)} $ of an exponential Lie group $ G=\exp (\g )$
is that of the topology of the space of coadjoint orbits $ \g^*/G $ according to \cite{Lep-Lud}, we have the following.

\begin{theorem}\label{expcompcon}
If $ G=\exp(\g) $ is an exponential Lie group, $\ell $ is an element of $ \g^* $ and $ \pi_\ell $ is the corresponding irreducible unitary representation, then for $ F\in  C^*(G)  $ the operator $ \pi_\ell(F) $ is compact if and only $ \pi_q(F)=0 $ for any $ q $ in the boundary of the orbit $ \OM_\ell $ of $ \ell $.
\end{theorem}

We are forced therefore to determine the boundary of every coadjoint orbit of our group $ G $.

\subsubsection{The open orbits $ \OM_{\ve S^*\pm\ve Q^*}, \ve=\pm1 $}\label{openor}

Let $f=\varepsilon S^*+\nu Q^*$, where
$\nu=\pm \ve$.
Then $\Ad^*G(f)$ is an open orbit, and
we have
$$\Omega_f=\{l ;\ \Phi(l)\neq 0, \ l(S)\neq 0\},$$
where
$\Phi(l):=2l(Q)l(S)-l(R)^2$ for $l\in\g^*$.
Noting that $\Phi(l)\leq 0$ if $l(S)=0,$ we have
\begin{eqnarray*}
\overline{\Omega_f}\setminus\Omega_f
=
\begin{cases}
\{l; \Phi(l)=0, l(S)\geq 0\}
&
\varepsilon=1, \nu=1,\\
\{l; \Phi(l)\leq 0, l(S)=0\}
\cup\{l; \Phi(l)=0, l(S)\geq 0\}
&
\varepsilon=1, \nu=-1,\\
\{l; \Phi(l)\leq 0, l(S)=0\}\cup
\{l; \Phi(l)=0, l(S)\leq 0\}
&
\varepsilon=-1, \nu=1,\\
\{l; \Phi(l)=0, l(S)\leq 0\}
&
\varepsilon =-1, \nu=-1.\\
\end{cases}
\end{eqnarray*}

\noindent
Considering the natural projection $\g^*\to \h^*$ defined by restriction and denoting
$G\cdot l:=\Ad^*G(l)\res{\h}$ for $l\in \g^*$ , we have
\begin{eqnarray*}
(\overline{\Omega_f}\res\h)\setminus({\Omega_f}\res\h)
=
\begin{cases}
\{G\cdot S^*,G\cdot Q^*, 0\}&
\varepsilon=1, \nu=1,\\
\{G\cdot S^* , \pm G\cdot Q^*, \pm G\cdot R^*, 0 \}
&
\varepsilon=1, \nu=-1,\\
\{-G\cdot S^*,\pm  G\cdot Q^*, \pm G\cdot R^*, 0 \}&
\varepsilon=-1, \nu =1,\\
\{-G\cdot S^*,-G\cdot Q^*, 0 \}&
\varepsilon=-1, \nu=-1.
\end{cases}
\end{eqnarray*}

{\bf Realization of $\boldsymbol{\pi_f} \in \widehat G$:}
We realize $\pi_f$ by taking the Pukanszky polarization
$\h=\mathbb R\mbox{-span}\{Q, R, S\}$
on $L^2(\exp(\mathbb RA+\mathbb RB)\exp(\mathbb RP))$.
We note that $\g_0:=\mathbb R\mbox{-span}\{A, B, P\}$ is
a subalgebra, and $G=G_0H$, where $G_0:=\exp\g_0$.
Writing $g_0=g_0(a,b,p):=\exp(aA+bB) \exp( pP)$, $g_1=g_1(q,r,s)=\exp(qQ+rR+sS)$,
$h=h(t,u,v)=\exp(tA+uB)\exp vP$, we have
$$
\pi_f(g_0g_1)\xi(h)=\chi_f(h^{-1}g_0g_1g_0^{-1}h)\xi(g_0^{-1}h)
=\chi_{\Ad^*(g_0^{-1}h)(f)}(g_1)\xi(g_0^{-1}h).
$$
We take the left Haar measure $dg_0$, $dg_1$, and $dg$ on $G_0$, $H$,
and $G$, respectively, which are defined by transferring Lebesgue measures by
$(a,b,p)\in \mathbb R^3 \mapsto g_0(a,b,p)\in G_0$,
$(q,r,s)\in \mathbb R^3 \mapsto g_1(q,r,s)\in H$, and $dg=dg_0dg_1$.
Let $F\in L^1(G)$. For $l\in\h^*$ we denote
\begin{equation*}
\widehat{F}^\h(a,b,p,l):=\widehat{F}^\h(g_0(a,b,p))(l):=
\int_\h F(g_0g_1)\chi_l(g_1)dg_1,
\quad
a,b,p\in\mathbb R.
\end{equation*}
Then we have
\begin{eqnarray*}
\pi_f(F)\xi(h)&=&\int_{G_0H}\pi_f(g_0g_1)\xi(h)F(g_0g_1)dg_0dg_1\\
&:=&\int_{\mathbb R^6}F(g_0(a,b,p)g_1(q,r,s))
\pi_f(g_0(a,b,p)g_1(q,r,s))\xi(h)
da db dp dq dr ds\\
&=&\int_{G_0H}F(g_0g_1)
\chi_{\Ad^*(g_0^{-1}h)(f)}(g_1)\xi(g_0^{-1}h)dg_0dg_1\\
&=&\int_{G_0H}F(hg_0g_1)
\chi_{\Ad^*(g_0^{-1})(f)}(g_1)\xi(g_0^{-1})dg_0dg_1\\
&=&\int_{G_0H} F(hg_0^{-1}g_1)
\chi_{\Ad^* g_0(f)}(g_1)\xi(g_0)\Delta_{G_0}^{-1}(g_0)dg_0dg_1\\
&=&\int_{G_0}\widehat{F^{\h}}(hg_0^{-1})(\Ad^*g_0(f|_\h))
\Delta_{G_0}^{-1}(g_0)\xi(g_0)dg_0\\
&=:&\int_{G_0}\mathcal K_F(h,g_0)\xi(g_0)dg_0,
\end{eqnarray*}
where $\Delta_{G_0}$ is the modular function of $G_0$ and
\begin{eqnarray}\label{kerS}
\mathcal K_F(h, g_0) & =&\mathcal K_F(h(t,u,v), g_0(a,b,p))\\
\nn&=&\widehat{F^{\h}}(hg_0^{-1})(\Ad^*g_0(f|_\h))
\Delta_{G_0}^{-1}(g_0)\notag \\
\nn& =&\widehat{F^{\h}}(g_0(t-a,u-b, e^{\frac{a-b}{2}}(v-p))
(\Ad^*g_0(a,b,p)(f|_\h))e^{\frac{a-b}{2}}\\
\nn&=&\wh F^{\h}(t-a,u-b, e^{\frac{a-b}{2}}(v-p),\ve( (e^{-b}(\pm 1+\frac{p^{2}}{2})Q^*+(-e^{-\frac{a+b}{2}}p)R^*+e^{-a}S^*))
e^{\frac{1}{2}(a-b)}.
\end{eqnarray}

\subsubsection{The orbits $ \OM_{b^* B^*+\ve S^*},\ve=\pm1,b^*\in\R $}\label{vesstrclo}

Let $f=b^*B^*+\varepsilon S^*$. Then $\g(f)=\mathbb R\mbox{-span}\{B,Q\}$.
We have
\begin{eqnarray*}
&\Omega_f=
\Phi^{-1}(0)\cap\Phi_2^{-1}(0)\cap
(\mathbb RA^*+\mathbb RB^*+\mathbb RP^*+(\varepsilon\mathbb R^{+,0})Q^*
+\mathbb RR^*+\varepsilon\mathbb R^+S^*),
\end{eqnarray*}
where
$$\Phi(l)
=2l(Q)l(S)-l(R)^2,
\quad
\Phi_2(l)
=l(R)l(P)-2(l(B)-b^*)l(S).$$
Then
\begin{eqnarray}\label{bdry-4d-3}
\nn\overline{\Omega_f}\setminus\Omega_f =&
(\mathbb RA^*+\mathbb RB^*+\varepsilon\mathbb R^+Q^*)\cup
(\mathbb RA^*+\mathbb RB^*+\mathbb RP^*)\\
 =& \bigcup_{a^*\in\mathbb R}\Ad^*G(a^*A^*+\varepsilon Q^*)
\cup \bigcup_{x^*\in\mathbb R, \nu=\pm1}\Ad^*G(x^*X^*+\nu P^*)
\notag\\
& \bigcup_{a^*, \lambda\in\mathbb R}(a^*A^*+\lambda B^*). \label{bdry-4d-3'}
\end{eqnarray}

\medskip\noindent
{\bf Proof of \eqref{bdry-4d-3}:}
Since $$\overline{\Omega_f}\subset
\Phi^{-1}(0)\cap\Phi_2^{-1}(0)\cap
(\mathbb RA^*+\mathbb RB^*+\mathbb RP^*+\mathbb R^{+,0}\varepsilon Q^*
+\mathbb RR^*+\mathbb R^{+,0}\varepsilon S^*),$$
we first note that
$$
\overline{\Omega_f}\setminus\Omega_f
\subset\Phi^{-1}(0)\cap\Phi_2^{-1}(0)\cap
(\mathbb RA^*+\mathbb RB^*+\mathbb RP^*+\mathbb R^{+,0}\varepsilon Q^*
+\mathbb RR^*).
$$
In fact, let $l\in\overline{\Omega_f}$,  and
suppose $l(S)\neq 0$. Then there exists $g\in G$ such that
$\Ad^*g(l)=b'B^*+q'Q^*+\varepsilon S^*$. Since
$\overline{\Omega_f}$ is $\Ad^*G$-invariant, we have
$\Phi(\Ad^*g(l))=2q'\varepsilon=0$ and
$\Phi_2(\Ad^*g(l))=-2(b'-b^*)\varepsilon=0$,
thus we have $\Ad^*g(l)=f$, this is, $l\in\Omega_f$.

Let
$f_\infty=\bar aA^*+\bar bB^*+\bar pP^*+\bar qQ^*+\bar rR^*
\in\overline{\Omega_f}\setminus\Omega_f$. Then $\bar r=0$, since
$\Phi(f_\iy)=0$. Suppose
$\bar f=\lim_{k\to\infty}\Ad^*E(a_k,0,p_k,0,r_k, s_k)(f)$.
Since $\bar f(S)=\lim_{k}e^{-a_k}\varepsilon=0$, we have
$\lim_ka_k=\infty$.

Suppose $\bar q=\lim_k\frac 12 p_k^2\varepsilon\neq 0$, since
$\bar b =\lim_{k}\left(b^*-\frac{p_kr_k}{2}\varepsilon\right)$, we have that
the sequence $\{r_k\}$ is bounded and
$\bar p=\lim_k r_ke^{-\frac{a_k}{2}}\varepsilon =0$.
It follows that
$\overline{\Omega_f}\setminus\Omega_f\subset
(\mathbb RA^*+\mathbb RB^*+\varepsilon\mathbb R^+Q^*)\cup
(\mathbb RA^*+\mathbb RB^*+\mathbb RP^*)$.

Conversely, let $\bar f=\bar aA^*+\bar bB^*+\bar qQ^*$ with
$\bar q\in\varepsilon\mathbb R^+$.
When $a\to\infty$, $p\to\sqrt{2\bar q\varepsilon}$,
$r\to 2\varepsilon(b^*-\bar b)\sqrt{2\bar q\varepsilon}^{-1}$ and
$s\to \varepsilon(\bar a+\bar b-b^*)$,
we have $\Ad^*E(a,0,p,0,r,s)f\to\bar f$.
Suppose $\bar f=\bar aA^*+\bar bB^*+\bar pP^*$ with $\bar p\neq 0$,
and let
$p(a):=2(b^*-\bar b)\bar p^{-1}e^{-\frac a2}$,
$r(a):=\varepsilon\bar pe^{\frac a2}$, and
$s(a):=\varepsilon\bar a-\frac{p(a)r(a)}{2}$.
Then $\Ad^*E(a, 0, p(a), 0,r(a), s(a))f\to\bar f$ as $a\to\infty$.
Suppose $\bar f=\bar aA^*+\bar bB^*$. Let
$p(a):=2(b^*-\bar b)\varepsilon e^{-\frac a4}$,
$r(a):=e^{\frac a4}$,
$s(a):=\varepsilon\bar a-\frac{p(a)r(a)}{2}$.
Then $\Ad^*E(a, 0, p(a), 0,r(a), s(a))f\to\bar f$ as $a\to\infty$.
\qed

\medskip\noindent
{\bf Realization of $\pi_f \in \widehat G$:}
 Taking the Pukanszky polarization $\mathfrak b=\mathbb R\mbox{-span}\{B, Q, R, S\}$ and $F\in L^1(G)$,
we realize $\pi_f$ and $\pi_f(F)$ on $L^2(\exp(\mathbb RX+\mathbb RP))$,
where $X=A+B$,
as follows: Let
\begin{equation*}
E_0(x,p):=\exp(xX+pP),
h=h(b,q,r,s):=\exp(bB)\exp(qQ+rR+sS)
\in \exp(\mathfrak b),
\end{equation*}
\begin{equation*}
E_0(x,p)_{b}:=\exp(bB)E_0(x,p)\exp(-bB)\in\exp(\mathbb RX+\mathbb RP).
\end{equation*}
For $\xi\in L^2(\exp(\mathbb RX+\mathbb RP))$, we have
\begin{eqnarray*}
 && \pi_f(E_0(x,p)h(b,q,r,s))\xi(E_0(t,u))\\
&=&\chi(F)(E_0(t-x, u-p)^{-1}_{(-b)}h(0,q,r,s)E_0(t-x, u-p)_{(-b)})
\chi_f(\exp(bB))\\
&&
\left(\frac{\Delta_{\exp(\mathfrak b)}}{\Delta_G}(\exp(-bB))\right)^{\frac12}
\xi(E_0(t-x,u-p)_{(-b)})\\
&=&\chi_{\Ad^*E_0(t-x, u-p)_{(-b)}(f)}(h(0,q,r,s))
\chi_f(\exp(bB))
\left(\frac{\Delta_{\exp(\mathfrak b)}}{\Delta_G}(\exp(-bB))\right)^{\frac12}
\xi(E_0(t-x,u-p)_{(-b)})\\
& =& \exp\bigg({i\varepsilon e^{-t+x}\big(\frac12q(u-p)^2e^{b}-r(u-p)e^{\frac b2}+s\big)}
+ib^*b+\frac{b}{4}\bigg)
\xi(t-x, e^{\frac b2}(u-p)).
\end{eqnarray*}
We denote $\widehat{F}^\h(g)(l)=\int_{\mathfrak h}\chi_l(\exp(qQ+rR+sS))
F(g\exp(qQ+rR+sS))dqdrds$ for $l\in\mathfrak h^*$ satisfying
$l([\mathfrak h,\mathfrak h])=\{0\}$, and for integrable
functions $F$ as before.
We take the left Haar measure $dg$ on $G$ defined by
$$
\int_GF(g)dg:=\int_{\mathbb R^6} F(E_0(x,p)h(b,q,r,s))e^{\frac b2}dxdpdbdqdrds.
$$
Then we have
\begin{eqnarray*}
&&\pi_f(F)\xi(E_0(t,u))
=\int_{\mathbb R^6}F(E_0(x,p)h(b,q,r,s))\pi_f(E_0(x,p)h(b,q,r,s))
\xi(E_0(t,u))e^{\frac b2}dxdpdbdqdrds\\
&=&\int_{\mathbb R^6}F(E_0(x,p)h(b,q,r,s))
\chi_{\Ad^*E_0(t-x, u-p)_{(-b)}(f)}(h(0,q,r,s))
\chi_f(\exp(bB))\\
&&
\left(\frac{\Delta_{\exp(\mathfrak b)}}{\Delta_G}(\exp(-bB))\right)^{\frac12}
\xi(E_0(t-x,u-p)_{(-b)})e^{\frac b2}dxdpdbdqdrds\\
&=&\int_{\mathbb R^3}\widehat{F^{\h}}
(E_0(t-x, u-p)\exp(bB))\left(\Ad^*E_0(x,p)_{(-b)}(f|_{\h})\right)
\chi_f(\exp(bB))\\
& &
\left(\frac{\Delta_{\exp(\mathfrak b)}}{\Delta_G}(\exp(-bB))\right)^{\frac12}
\xi(E_0(x,p)_{(-b)})e^{\frac b2}dxdpdb\\
&=&\int_{\mathbb R^3}\widehat{F^{\h}}
(E_0(t-x, u-e^{-\frac{b}{2}}p)\exp(bB))\left(\Ad^*E_0(x,p)(f|_{\mathfrak h})\right)
\chi_f(\exp(bB))
\xi(E_0(x,p))e^{\frac b4}dxdpdb\\
&=&\int_{\mathbb R^3} \widehat{F^{\h}}
(E_0(t-x,u-e^{-\frac b2}p)\exp(bB))
\left(\varepsilon e^{-x}\big(\frac12 p^2 Q^*- pR^*+S^*\big)\right)
e^{ib^*b}e^{\frac b4}\xi(E_0(x,p))dxdpdb\\
&=: &\int_{\mathbb R^2}\mathcal K_F(t,u; x,p)\xi(E_0(x,p))dxdp,
\end{eqnarray*}
where
\begin{eqnarray}\label{fdthreeker}
 &&\mathcal{K}_F(t,u ; x,p)\\
\nn  &:=&
\int_{\mathbb R}
 \widehat{F^{\h}}
(E_0(t-x,u-e^{-\frac b2}p)\exp (bB))
\left(\varepsilon e^{-x}\big(\frac12 p^2 Q^*- pR^*+S^*\big)\right)
e^{ib^*b}e^{\frac b4}
db.
\end{eqnarray}

\subsubsection{The orbits $ b^* B^*+\ve R^*, b^*\in \R, \ve=\pm 1 $}

Let
$f=b^*B^*+\varepsilon R^*$. Then $\g(f)=\mathbb R\mbox{-span}\{S, A-B\}$.
Writing $X=A+B$, $Y=A-B$, $X^*=\frac 12(A^*+B^*)$, $Y^*=\frac 12(A^*-B^*)$,
we have
\begin{eqnarray}
\Omega_f=& \{\bar xX^*+\bar yY^*+\bar pP^*+\bar qQ^*+\bar rR^* ;\
\bar r\in\varepsilon\mathbb R^+, \bar y=-b^*-(\bar p\bar q/\bar r)\}
\notag,\\
\overline{\Omega_f}\setminus\Omega_f
\nn= &\{\bar xX^*+\bar yY^*+\bar pP^* ;\ \bar x,\bar y,\bar p\in\mathbb R\}
\cup\{\bar xX^*+\bar yY^*+\bar qQ^* ;\ \bar x,\bar y,\bar q\in\mathbb R\}
\label{bdry-4d-2}\\
\nn= &\bigcup_{\bar x\in\mathbb R, \varepsilon=\pm 1}\Ad^*G(\bar xX^*+\varepsilon P^*)
\bigcup_{a^*\in\mathbb R, \nu =\pm 1}\Ad^*G(a^*A^*+\nu Q^*)\notag\\
& \bigcup_{a^*, \lambda \in\mathbb R}\{a^*A^*+\lambda B^*\}.
\end{eqnarray}

{\bf Proof of \eqref{bdry-4d-2}:}
Recall that
\begin{eqnarray*}
\Ad^*E(\alpha, \alpha, p,q,r,0)\cdot f
= (b^*+r\varepsilon)X^* +(-b^*+pq\varepsilon)Y^*
+  q\varepsilon P^*
-  e^{-\alpha}p\varepsilon Q^*
+  e^{-\alpha}\varepsilon  R^*.
\end{eqnarray*}
Let $\Psi$ be a function on
$\mathbb R\mbox{-span}\{X^*, Y^*, P^*, Q^*, R^*\}$
defined by
$$
\Psi(x,y,p,q,r):=r(y+b^*)+pq \quad \text{for }\,
(x,y,p,q,r)\in \mathbb R\mbox{-span}\{X^*, Y^*, P^*, Q^*, R^*\}.
$$
Then $\overline{\Omega_f}\subset
\Psi^{-1}(0)\cap(\mathbb R^4\times (\varepsilon\mathbb R^{+,0}))$
and $\overline{\Omega_f}\setminus\Omega_f\subset
\Psi^{-1}(0)\cap(\mathbb R^4\times\{0\})
=\{(x,y,p,0,0); x,y,p\in\mathbb R\}\cup\{(x,y,0,q,0); x,y,q\in\mathbb R\}$.

Conversely, let $\bar x, \bar y,\bar p\in\mathbb R$ and suppose $\bar
p\neq 0$, then
\[\begin{array}{rl}
\Ad^*E(\alpha,\alpha, \frac{\bar y+b^*}{q},q\varepsilon, r\varepsilon,0)f=
& (b^*+r)X^*+\bar yY^*+qP^*-e^{-\alpha}\frac{\varepsilon(\bar y+b^*)}{q}Q^*+
e^{-\alpha}\varepsilon R^*\\
\to & \bar xX^*+\bar yY^*+\bar p P^*
\, \text{ as } \, \alpha\to\infty, q\to\bar p, r\to -b^*+\bar x.
\end{array}\]
If $\bar p=0$, we take
\[\begin{array}{rl}
\Ad^* E(\alpha,\alpha, \varepsilon(\bar y+b^*)e^{\frac{\alpha}{2}},
e^{-\frac{\alpha}{2}},r\varepsilon,0)f =
&(b^*+r)X^*+\bar yY^*+ e^{-\frac{\alpha}{2}}\varepsilon P^*
-(\bar y+b^*)e^{-\frac{\alpha}{2}}Q^*+e^{-\alpha}\varepsilon R^*\\
\to &
\bar xX^*+\bar yY^*+0 P^*
\, \text{ as }\, \alpha\to\infty, r\to -b^*+\bar x.
\end{array}\]
Thus we have $\bar xX^*+\bar yY^*+\bar pP^*\in\overline{\Omega_f}$.
We can similarly verify
$\bar xX^*+\bar yY^*+\bar qQ^*\in\overline{\Omega_f}$.
\qed

{\bf Realization of $\boldsymbol{\pi_f} \in \widehat G$:}
We take the Pukanszky polarization $\q=\mathbb R\mbox{-span}\{Y, Q, R, S\}$
at $f$ to  realize $\pi_f$.
Noting that $[X,P]=0$, let
\begin{equation*}
E_0(t,u):=\exp(tX+uP) \, \text{ and } \,
h=h(y,q,r,s):=\exp yY\exp(qQ+rR+sS).
\end{equation*}
We realize $\pi_f=\ind_{\exp(\mathfrak q)}^G\chi_f$
 on $L^2(\exp(\mathbb RX+\mathbb RP))$ as follows:
Noting that $[Y,\mathbb RX+\mathbb RP]\subset\mathbb RP$,
we write $E_0(t,u)_y:=\exp(yY)E_0(t, u)\exp(-yY)$.
Let $\xi\in L^2(\exp(\mathbb RX+\mathbb RP))$. Then
\begin{eqnarray*}
&&\pi_f(E_0(x,p)h(y,q,r,s))\xi(E_0(t,u))\\
&&=\chi_f(E_0(t-x,u-p)^{-1}_{(-y)}h(0,q,r,s)E_0(t-x, u-p)_{(-y)})\chi_f(\exp yY)\\
 && \qquad \left(\frac{\Delta_{\exp(\mathfrak q)}}{\Delta_{G}}(\exp(-yY))\right)^{\frac12}
\xi(E_0(t-x,u-p)_{-y})
\\
&&
=\chi_{\Ad^*E_0(t-x,u-p)_{(-y)}(f)}(h(0,q,r,s))
\chi_f(\exp yY)
\left(\frac{\Delta_{\exp(\mathfrak q)}}{\Delta_{G}}(\exp(-yY))\right)^{\frac12}
\xi(E_0(t-x,u-p)_{(-y)})\\
&&=\exp\left({-i\varepsilon(e^{-y}q(u-p)-r)e^{x-t}+ib^*y-\frac{y}{2}}\right)
\xi(t-x, e^{-y}(u-p)).
\end{eqnarray*}
We take the left Haar measure $dg$ on $G$ defined by
\begin{equation*}
\int_GF(g)dg:=\int_{\mathbb R^6}F(E_0(x,p)h(y,q,r,s))e^{-y}dxdpdydqdrds
\end{equation*}
for integrable functions $F\in L^1(G)$, and define
\begin{equation*}
\widehat{F^\mathfrak h}(g)(l):=\int_{\mathfrak h}
F(g\exp(qQ+rR+sS))\chi_l(\exp(qQ+rR+sS))dqdrds,
\end{equation*}
for $l\in\mathfrak h^*$ satisfying $l([\mathfrak h,\mathfrak h])=\{0\}$. Then
\begin{eqnarray*}
&&\pi_f(F)\xi(E_0(t,u))
=\int_{\mathbb R^6} F(E_0(x,p)h)\pi_f(E_0(x,p)h))\xi(E_0(t,u))e^{-y}dxdpdydqdrds\\
&=&\int_{\mathbb R^6} F(E_0(x,p)h)\chi_{\Ad^*E_0(t-x,u-p)_{(-y)}(f)}(h(0,q,r,s))
\chi_f(\exp yY)\\
&& {}
\left(\frac{\Delta_{\exp(\mathfrak q)}}{\Delta_{G}}(\exp(-yY))\right)^{\frac12}
\xi(E_0(t-x,u-p)_{(-y)})e^{-y}dxdpdydqdrds\\
&=&\int_{\mathbb R^6} F(E_0(t-x,u-p)h)\chi_{\Ad^*E_0(x,p)_{(-y)}(f)}(h(0,q,r,s))
\chi_f(\exp yY)\\
&& {}
\left(\frac{\Delta_{\exp(\mathfrak q)}}{\Delta_{G}}(\exp(-yY))\right)^{\frac12}
\xi(E_0(x,p)_{(-y)})e^{-y}dxdpdydqdrds\\
&=&\int_{\mathbb R^6} F(E_0(t-x, u-e^yp)h(y,q,r,s))\chi_{\Ad^*E_0(x,p)(f)}(h(0,q,r,s))
\chi_f(\exp yY)e^{-\frac y2}\xi(E_0(x,p))dxdpdydqdrds\\
&=&\int_{\mathbb R^3} \widehat{F^\h}
(E_0(t-x,u-e^yp)\exp(yY))(\Ad^*E_0(x,p)(f|_\h )
e^{ib^*y}e^{-\frac y2}\xi(E_0(x,p))dxdpdy\\
&=&\int_{\mathbb R^3} \widehat{F^\h}
(E_0(t-x,u-e^yp)\exp(yY))(-\varepsilon p e^{-x}Q^*+\varepsilon e^{-x}R^*)
e^{ib^*y}e^{-\frac y2}\xi(E_0(x,p))dxdpdy\\
&=:&\int_{\mathbb R^2}\mathcal K_f(t,u; x,p)\xi(E_0(x,p))dxdp,
\end{eqnarray*}
where
\begin{eqnarray}\label{kerndftwo}
\mathcal{K}_f(t,u ; x,p):=
\int_{\mathbb R} e^{ib^*y}e^{-\frac y2}\widehat{F^\h}
(E(t-x,u-e^yp)\exp yY)(-\varepsilon p e^{-x}Q^*+\varepsilon e^{-x}R^*)dy.
\end{eqnarray}

\subsubsection{The orbits $ \nu Q^*+\ve P^*, \nu,\ve=\pm1 $}

Let $f=\varepsilon P^*+\nu Q^*$ for $\varepsilon, \nu =\pm 1$.
Then we have $\g(f)=\mathbb R\mbox{-span}\{S, R\}$.
Let $X^*=\frac 12(A^*+B^*)$, $Y^*=\frac12(A^*-B^*)$.
\begin{eqnarray}\label{pstqst}
\nn\Omega_f= & \mathbb RA^*+\mathbb RB^*+ \mathbb R^+\varepsilon P^*
+\mathbb R^+\nu Q^*\notag,\\
\nn\overline{\Omega_f}\setminus\Omega_f
= & \left(\mathbb RA^*+\mathbb RB^*+ \mathbb R^+\varepsilon P^*\right)
\cup\left(\mathbb RA^*+\mathbb RB^*+ \mathbb R^+\nu Q^*\right)
\cup(\mathbb RA^*+\mathbb RB^*)
\\
\nn =& \bigcup_{\alpha^*\in\mathbb R}\left(\alpha^*X^*+\mathbb RY^*+\mathbb R^+\varepsilon P^*\right)
  \bigcup_{a^*\in\mathbb R}
\left(a^*A^*+\mathbb RB^*+\mathbb R^+\nu Q^*\right)\notag\\
 \nn&\bigcup_{a^*, b^*\in\mathbb R}\{a^*A^*+b^*B^*\}\\
\nn=& \bigcup_{\alpha^*\in\mathbb R}\Ad^*(G)(\alpha^*X^*+\varepsilon P^*)
 \bigcup_{{y^*\in\mathbb R}}\Ad^*(G)(y^*Y^*+\nu Q^*)\notag\\
 & \bigcup_{a^*, b^*\in\mathbb R}\{a^*A^*+b^*B^*\}.
\end{eqnarray}

We realize $\pi_f$ in $\widehat G$ by taking a Pukanszky polarization $\mathfrak l=\mathbb R\mbox{-span}\{P, Q, R, S\}$ and $\pi_f=\ind_{\exp(\mathfrak l)}^G\chi_f$. Writing
$g_0=E_0(a,b):=\exp(aA+bB)$, $n=n(p,q,r,s):=\exp(pP)\exp(qQ+rR+sS)$ and
$t_0=E_0(t,u):=\exp(tA+uB)$, we have
\begin{eqnarray*}
\pi_f(g_0n)\xi(t_0) & =& \chi_f((g_0^{-1}t_0)^{-1}n g_0^{-1}t_0)\xi(g_0^{-1}t_0))
=\chi_{\Ad^*(g_0^{-1}t_0)f}(n)\xi(g_0^{-1}t_0))\\
&=&\chi_{\Ad^*E_0(t-a,u-b)(f)}(n)\xi(E_0(t-a,u-b))\\
&=& \exp(i({\varepsilon p e^{\frac{-t+a+u-b}{2}}+\nu q e^{-u+b}}))\xi(E_0(t-a,u-b)).
\end{eqnarray*}
Let $F\in L^1(G)$ and $l\in\mathfrak l^*$ such that $l[\mathfrak l,\mathfrak l]=\{0\}$.
We define
$$
\widehat{F^\mathfrak l}(g_0)(l):=\int_{\mathfrak l}F(g_0 n(p,q,r,s))e^{il(pP+qQ+rR+sS)}dpdqdrds.
$$
Let $dg_0dn$ be the left Haar measure transfered the
Lebesgue measure $dadbdpdqdrds$ on $\g$ by the mapping
$(a,b,p,q,r,s)\mapsto g_0(a,b)n(p,q,r,s)$.
\begin{eqnarray*}
\pi_f(F)\xi(E_0(t,u))
&=&\int_G\pi_f(g_0 n)\xi(t_0)F(g_0 n)dg_0dn\\
&=&\int_{\mathbb R^6}\chi_{\Ad^*E_0(t-a,u-b)(f)}(n)\xi(E_0(t-a,u-b))F(E_0(a,b)n)dadbdpdqdrds\\
&=&\int_{\mathbb R^6}\chi_{\Ad^*E_0(a,b)(f)}(n)\xi(E_0(a,b))F(E_0(t-a,u-b)n)dadbdpdqdrds\\
&=:&\int_{\g/\mathfrak l}\widehat{F^{\mathfrak l}}(E_0(t-a,u-b))(\Ad^*E_0(a,b)(f|_\h))\xi(E_0(a,b))dadb\\
&=&\int_{\mathbb R^2}\widehat{F^{\mathfrak l}}(E_0(t-a,u-b))
(\varepsilon e^{\frac{-a-b}{2}}P^*+\nu e^{-b}Q^*)\xi(E_0(a,b))dadb
\\
&=:&\int_{\mathbb R^2}\mathcal K_F(t,u; a,b)\xi(E_0(a,b))dadb,
\end{eqnarray*} where
\begin{eqnarray}\label{kerdfourrone}
\mathcal K_F(t,u; a,b)& := &
\widehat{F^{\mathfrak l}}(E_0(t-a,u-b))(\Ad^*E_0(a,b)(f|_{\mathfrak l}))\notag\\
& = & \widehat{F^{\mathfrak l}}(E_0(t-a,u-b))
(\varepsilon e^{\frac{-a-b}{2}}P^*+\nu e^{-b}Q^*).
\end{eqnarray}

\subsubsection{The orbits $ a^* A^*+\ve Q^*, a^*\in\R, \ve=\pm1 $}

Let $f:= a^*A^*+ \ve Q^*$. Then
$\g(a^*A^*+\ve Q^*) =\mathbb R\mbox{-span}\{A, P, R, S\}$,
\begin{eqnarray}
\Omega_f  & = a^*A^*+\mathbb R B^*+ \mathbb R^+\ve Q^*\notag,\\
\overline{\Omega_f}\setminus\Omega_f
& =a^*A^*+\mathbb R B^*+0\cdot Q^*
=\cup_{\lambda\in\mathbb R}(a^*A^*+\lambda B^*).
\end{eqnarray}
We realize $\pi_f$ in $\widehat G$ as $\pi_f=\ind_{\exp(\mathfrak l_A)}^G\chi_f$,
$\mathfrak l_A=\mathbb R\mbox{-span}\{A, P, Q, R, S\}$,
on $L^2(G/\exp(\mathfrak l_A),\chi_f)\simeq L^2(\exp\mathbb RB)$.
Write
\begin{equation*}
E_0(b):=\exp bB,
\quad
h=h(a,p,q,r,s):=\exp(aA)\exp(pP)\exp(qQ+rR+sS)
\end{equation*}
for $b\in\mathbb R$, $(a,p,q,r,s)\in\mathbb R^5$. We have
\[\begin{array}{rl}
\pi_f(E_0(b)h(a,p,q,r,s))\xi(E_0(t))= &
\chi_{\Ad^*E_0(t-b)(f)}(h)\xi(E_0(t-b))\\
= &
\exp(i(aa^*+q\ve e^{b-t}))\xi(E_0(t-b)),
\end{array}\]
\begin{eqnarray*}
\pi_f(F)\xi(E_0(t)) &=&
\int_{\mathbb R^6} F(E_0(b)h)\chi_{\Ad^*E_0(t-b)(f)}(h)\xi(E_0(t-b))dbdadpdqdrds\\
&=& \int_{\mathbb R^6} F(E_0(t-b)h)\chi_{\Ad^*E_0(b)(f)}(h)\xi(E_0(b))dbdadpdqdrds\\
&=& \int_\mathbb{R} \widehat{F^{\mathfrak l_A}}(E_0(t-b))(\Ad^*E_0(b)(f|_{\mathfrak l_A})\xi(E_0(b))db\\
&=& \int_{\mathbb{R}}\mathcal K_F(t,b)\xi(E_0(b))db,
\end{eqnarray*}
where
\begin{equation*}
\widehat{F^{\mathfrak l_A}}(g)(l):
=\int_{\mathfrak l_A}F(gh(a,p,q,r,s))\chi_l(h(a,p,q,r,s))dadpdqdrds
\end{equation*}
for $l\in\g^*$ such that $l([\mathfrak l_A,\mathfrak l_A])=\{0\}$, and
\begin{eqnarray*}
\mathcal K_F(t,b) :=\widehat{F^{\mathfrak l_A}}(E_0(t-b))(\Ad^*E_0(b)(f|_{\mathfrak l_A})
=\widehat{F^{\mathfrak l_A}}(E_0(t-b))(a^*A^*+\ve e^{-b}Q^*).
\end{eqnarray*}

\subsubsection{The orbits $ x^* X^*+\ve P^*, x^*\in\R, \ve=\pm1 $}

Let $f=\alpha^*(\frac{A^*+B^*}{2})+\varepsilon P^*
=\alpha^*X^*+\varepsilon P^*$,
$X=A+B$, $Y=A-B$, $X^*=\frac{A^*+B^*}{2}$, $Y^*=\frac{A^*-B^*}{2}$.
Then
$\g(f)=\mathbb R\mbox{-span}\{X, Q, R, S\}$.
By \eqref{2d-1-1} and \eqref{2d-1-2}, we have
\begin{eqnarray}
\Omega_f & =\alpha^*X^*+\mathbb RY^*+\mathbb R^+\varepsilon P^*\notag,\\
\overline{\Omega_f}\setminus\Omega_f
& =\alpha^*X^*+\mathbb RY^*
=\bigcup_{\lambda\in\mathbb R}\{\alpha^*X^*+\lambda Y^*\}.
\end{eqnarray}
We realize $\pi_f$ in $\widehat G$ as $\pi_f=\ind_{\exp(\mathfrak l_X)}^G\chi_f$, where
$\mathfrak l_X=\mathbb R\mbox{-span}\{X, P, Q, R, S\}$,
on $L^2(G/\exp(\mathfrak l_X), \chi_f)\simeq L^2(\exp\mathbb RY)$.
We write
\begin{equation*}
E_0(y):=\exp(yY),
\quad
h=h(x,p,q,r,s):=\exp (xX)\exp( pP)\exp(qQ+rR+sS)
\end{equation*}
for $y\in\mathbb R$ and $(x,p,q,r,s)\in\mathbb R^5$,
and  transfer a Lebesgue measure on $\mathbb R^6$ to
a left Haar measure on $G$ by
$(x,y,p,q,r,s)\mapsto E_0(y)h(x,p,q,r,s)$. For $\xi\in L^2(G/\exp(\mathfrak l_X),\chi_f)$
we have
\[\begin{array}{rl}
\pi_f(E_0(y)h(x,p,q,r,s))\xi(E_0(t))=
&\chi_{\Ad^*E_0(t-y)(f)}(h)\xi(E_0(t-y))\\
=&\exp\left(ix\alpha^*+ip\varepsilon e^{-t+y}\right)\xi(E_0(t-y)),
\end{array}\]
and for an integrable function $F$, we have
\begin{eqnarray*}
\pi_f(F)\xi(E_0(t))
&=&\int_{\mathbb R^6}
F(E_0(y)h(x,p,q,r,s))\pi_f(E_0(y)h(x,p,q,r,s))\xi(E_0(t))dydxdpdqdrds
\\
&=&\int_{\mathbb R^6} F(E_0(y)h)
\chi_{\Ad^*E_0(t-y)(f)}(h)\xi(E_0(t-y))dydxdpdqdrds
\\
&=&\int_{\mathbb R^6} F(E_0(t-y))h)
\chi_{\Ad^*E_0(y)(f)}(h)\xi(E_0(y))dydxdpdqdrds
\\
&=:&\int_{\mathbb R} \widehat{F^{\mathfrak l_X}}(E_0(t-y))(\Ad^*E_0(y)(f|_\h))
\xi(E_0(y))dy
\\
&= &\int_{\mathbb R} \widehat{F^{\mathfrak l_X}}(E_0(t-y))
(\alpha^*X^*+\varepsilon e^{-y}P^*)\xi(E_0(y))dy.
\end{eqnarray*}
For $l\in\g^*$ satisfying $l([\mathfrak l_X,\mathfrak l_X])=\{0\}$, we write
\begin{equation*}
\widehat{F^{\mathfrak l_X}}(g)(l):=\int_{\mathfrak l_X}e^{il(Z)}F(g\exp Z)dZ,
\
g\in G,
\end{equation*}
with some fixed Lebesgue measure $dZ$ on $\h$. And
$\pi_f(F)$ is described by the operator
$\pi_f(F)\xi(E_0(y))=\int_\mathbb R\mathcal K_f(t,y)\xi(E_0(y))dy$, where
\begin{equation}\label{ker2d1}
\mathcal K_f(t,y):=\widehat{F^{\mathfrak l_X}}(E_0(t-y))(\Ad^*E_0(y)(f|_{\mathfrak l_X}))
=\widehat{F^{\mathfrak l_X}}(E_0(t-y))
(\alpha^*X^*+\varepsilon e^{-y}P^*).
\end{equation}

\section{The continuity  and infinity conditions}\label{contcon}
In this short section we find the continuity and infinity conditions for $C^*(G)$.
\begin{definition}\label{conin}
\rm   We say that a net $(\ga_i)_{i\in I}$ in  a topological space $\GA$ goes to infinity, if the net contains no converging subnet.
In particular, a sequence of orbits $\ol \OM=(\OM_k)_{k\in\N}\subset \g^*$ goes to infinity, if for any compact subset $K\subset \g^*$ there exists an index $k_0$ such that $K\cap \OM_k=\es$ whenever $k\geq k_0$.
\end{definition}

\begin{proposition}[Riemann-Lebesgue Lemma]\label{infycon}
Let $A$ be a C*-algebra. If a net $(\pi_k)_k\subset \wh A$ goes to infinity, then $\lim_k \noop{\pi_k(a)}=0$ for all $a\in A$.
 \end{proposition}
\begin{proof} We know from \cite[Proposition 3.3.7]{Di} that for every $ c>0$ and $a\in A$, the subset $\{\pi\in\wh A; \noop{\pi(a)}\geq c\}$ is quasi-compact. This shows that $\lim_k \noop{\pi_k(a)}=0$, if the net $(\pi_k)_k$ goes to infinity.
\end{proof}

In the following we recall the definition of properly converging sequences and their limit sets, we also give a proposition of properly converging sequences in our group.


\begin{definition}\label{limits}
\rm   Let $ Y $ be a topological space. Let
$\ol y= (y_k)_k $ be a net in $ Y $. We denote by $ L(\ol y) $ the set of
all limit points of the net $ \ol y $.
A net $\ol y$ is called \textit{properly converging} if $
\ol y $ has limit points and if every cluster point of the net is a limit
point, i.e. the set of limit points of any subnet is always the same, indeed, it equals to $ L(\ol y) $.
\end{definition}

We know that every converging net in $ Y $ admits a properly converging subnet, hence, we can work with properly
converging nets in our space $\wh G$.

\begin{definition}\label{fourhausd}
\rm  Let $\ve=\pm 1$, and
\begin{itemize}
\item
$\R_{\ve P^*}:=\{\Ad^*(G)(x^*X^*+ \ve P^*)=x^*X^*+\Ad^*(G)(\ve P^*),  x^*\in\mathbb R \}$,
\item
$\R_{\ve Q^*}:=\{\Ad^*(G)(a^*A^*+ \ve Q^*)=a^*A^*+\Ad^*(G)(\ve Q^*),  a^*\in\mathbb R\}$,
\item
$\R_{\ve R^*}:= \{\Ad^*(G)(b^*B^*+\varepsilon R^*)=b^*B^*+\Ad^*(G)(\varepsilon R^*), b^*\in\mathbb R\}$, 
\item
$\R_{\ve S^*}:= \{\Ad^*(G)(b^*B^*+\varepsilon S^*)=b^*B^*+\Ad^*(G)(\varepsilon S^*), b^*\in\mathbb R\}$.
\end{itemize}
\end{definition}

\begin{proposition}\label{limpropcon}
Let $ \ul\OM=(\OM_k)_{k\in \N} $ be  a properly converging sequence in $ \g^*/G $. Then there exists a subsequence (also denoted by $\ul\OM= (\OM_k)_{k\in \N} $ for simplicity) such that either $ \ul\OM $ is  a constant sequence, or if $ \ul\OM $ is not constant, it is contained in one of the subsets in Definition \ref{fourhausd} and its limit set is the closure in $ \g^*/G $ of one of the points of this subset or all the elements  of the subsequence are  characters and the limit set is a single character.
\end{proposition}

\begin{proof}
The orbit space consists of four open orbits of the eight sets given in Definition \ref{fourhausd}, of the four orbits $\nu P^*+\ve Q^*, \ve,\nu\in\{1,-1\},$ and of the set of characters $ \X $. Hence, if the sequence $ \ul\OM $ contains a constant subsequence, then the limit set of the sequence itself is the closure of this constant element, since it is properly converging. If the sequence is not constant, then we can suppose that a subsequence is contained in one of the 8 sets or is made out of characters. If we deal with non-characters, then  we can write $ \ul\OM={(\OM+\la_k U^*)}_{k\in\N} $, where $ \la_k\in\R $ and $ U^* $ is one of the characters $ X^*,A^* $ or $ B^* $, and $ \OM $ is a fixed orbit in one of sets $ \R_{\ve P^*}, \R_{\ve Q^*}, \R_{\ve R^*}, \R_{\ve S^*} $. Since the sequence $ \ul \OM $ is properly converging, for any $ v $ in a limit orbit $ \OM_\iy $, there exists a sequence $ (v_k)_k\subset \OM $ and a real sequence $ (\la_k)_k $ such that $ \la_k U^*+v_k $ converges in $ \g^* $ to $ v$. We can write $ \g=\R U+\m $, where $ \m $ is an
ideal of $ \g $ containing $ [\g,\g]  $. Then by the formulas in Subsection \ref{desor}, it follows that $ v_k(U)=\va(v_k{\res \m}) $, where $ \va:\
\m^*\to \R $ is  a continuous
function. Hence the sequence $ (\la_k+\va(v_k{\res\m}))_k $ is convergent with a limit $ \la+v(U) $ and so the sequence  $ (\la_k)_k $ converges to some $ \la $. Hence the limit orbit $ \OM_\iy $ is contained in $ \ol{\OM+\la U^*} =\la U^*+\ol \OM $, and so the limit set of the sequence $ \ul \OM  $ is the set $ \ol{\OM+\la U^*} $.
\end{proof}


We define now the subsets $S_j, j=0,\cdots, 6$, which we shall need for the definition of the algebra $D^*(G)$ in the later section.

\begin{definition}\label{defgais}
\rm Let
\begin{itemize}\label{}
\item $ \GA_0=\X =\{a^*A^*+b^*B^*, a^*,b^*\in\R\}$,
\item  $ \GA_1:=\{\R_{\ve P^*},\ve=\pm 1\} $,
\item  $ \GA_2:=\{\R_{\nu Q^*},\nu=\pm 1\} $,
\item  $ \GA_3:=\{\Omega_{\ve P^*+\nu Q^*}, \nu,\ve= \pm 1\} $,
\item  $ \GA_4:=\{\R_{\ve R^*},\ve=\pm 1\} $,
\item  $ \GA_5:=\{\R_{\ve S^*}, \ve=\pm1\}$,
\item  $ \GA_6:=\{\Omega_{\ve S^*+\nu Q^*},\ve,\nu=\pm 1\} $.
 \end{itemize}
Let $ S_i:=\bigcup_{j=0}^i \GA_j\subset  \g^*/G  $. Then $ S_i $ is closed in $ S_{i+1}$ for $i=0,\ldots,5$ with respect to the orbit topology. The set $\GA_6$ is finite and open in $\wh G$ while $\GA_3$ is finite and open in $S_3$. The sets $\GA_1,\GA_2,\GA_4$ and $\GA_5$ are homeomorphic to two disjoint copies of the real line $\R$.
\end{definition}

\begin{theorem}\label{contonR}
The mappings  $ \pi\mapsto \pi(F) $, $F\in C^*(G) $, are  norm-continuous on the subsets
 $\GA_i$ of $ \wh{G} $ for all $i=0,\ldots, 6$.
\end{theorem}

\begin{proof}
This follows from the following more general result for the C*-algebra of a second countable  locally compact group $G$.

Let $(\ch_k)_k$ be  a sequence of unitary characters of $G$, which converges pointwise to a unitary character $\ch$. Let $(\pi,\H)$ be a unitary representation of $G$. Then for every $F\in L^1(G)$, we have that
\begin{eqnarray*}
 \lim_{k\to\iy}\noop{\ch_k\otimes \pi(F)-\ch\otimes \pi(F)}=0.
 \end{eqnarray*}
Indeed, it suffices to remark that
\begin{eqnarray*}
 \noop{\ch_k\otimes \pi(F)-\ch\otimes \pi(F)}=\noop{\pi(\ch_k\cdot F)-\pi(\ch\cdot F)}\leq \no{\ch_k\cdot F-\ch\cdot F}_1
 \end{eqnarray*}
and therefore by the Lebesgue's theorem of dominated convergence, $\lim_{k\to\iy}\noop{\ch_k\otimes \pi(F)-\ch\otimes \pi(F)}=0$.
\end{proof}

\section{The boundary conditions for the C*-algebra}\label{bound}

For each one of our subsets $\GA_i$ defined in Definition \ref{defgais}, we must find the linear mappings $\si_{i,\de},\de>0, $ which give the structure of the C*-algebra of $G$ an almost $C_0(\K)$-algebra. These mappings $\si_{i,\de}$ will be built on the representations coming from the different sets $\GA_i$. Hence, this forces us to make a case by case study of the different situations.


We indicate here some definitions and methods, which will be used in the proofs of this (long) section.

\begin{remark}\label{remark_pre}
\rm
\begin{enumerate}\label{}

\item
For a measurable subset $ S $ of  a measure space $(X,\mu)$, denote
 the multiplication operator with the characteristic function $ 1_{S} $ on $ L^2(X,\mu) $ by $ M_{S} $.

\item  Let $H=\exp(\h)$ be  an abelian  normal subgroup of $G$. The subspace $L^1_c=L^1_{c,\h}$ is defined to be the set of all $F$'s in $L^1(G)$, for which the partial Fourier transform
 \begin{eqnarray*}
 \wh F^\h(s,q):=\int_H F(E(s)h)\ch_q(h)dh,s\in G, q\in \h^*
 \end{eqnarray*}
is a $C^\iy$-function with compact support on $G/H\ti \h^*$.
The vector space $L^1_c$ is of course dense in $L^1(G)$ and hence is also dense in $C^*(G)$.

Furthermore, for every $F\in L^1_c$ there exists a function $\va\in C_c(G/H)$ such that
\begin{eqnarray*}
 \val{\wh F^\h(s,q)}\leq \no{q}\val{\va(s)}\, \text{ for }\, s\in G, q\in\h^*.
 \end{eqnarray*}

\item
 Let $ F:\R^d\ti\R^d\to \C $ be a smooth function with compact support for which there exists some continuous function $ \va:\R^d\to \R^+ $ with compact support, such that
\begin{eqnarray*}
  \val {F(x,y)}\leq \va(x-y) \, \text{ for }\, x,y\in\R^d .
\end{eqnarray*}
Then by the Young's inequality, the operator norm  $ \noop{\cdot} $ of the kernel operator $ T_F $ defined on $ L^2(\R^d) $ by
\begin{eqnarray*}
T_F(\xi)(x):=\int_{\R^d}F(x,y)\xi(y)dy \text{ for } \xi\in L^2(\R^d), x\in\R^d,
\end{eqnarray*}
is bounded by the $ L^1 $-norm $ \no \va_1 $ of $ \va $.

\item
Let $(X,\mu), (Y,\nu)$ be two measure spaces. Let $Y\mapsto L^2(X,\mu): y\mapsto \xi(y)$ be  an integrable mapping.
Then we have that
\begin{eqnarray}\label{estfh2}
 \nn\no{\int_Y \xi(y)d\nu(y)\vert}_2&\leq& \int_Y \no{\xi(y)}_2d\nu(y),
\end{eqnarray}
this is,
\begin{eqnarray}
 \nn\Big(\int_X\vert{\int_Y\xi(y)(x)d\nu(y)}\vert^2d\mu(x)\Big)^{1/2}&\leq&\int_Y\Big(\int_X\vert \xi(y)(x)\vert^2d\mu(x)\Big)^{1/2}d\nu(y).
\end{eqnarray}

\end{enumerate}
\end{remark}

\begin{proposition}\label{sumsofop}
Let $(X,\mu)$ be a measure space, and $(\si_i)_{i\in I}$ be a family of bounded linear operators on the Hilbert space $\H=L^2(X,\mu)$ such that $\noop{\si_i}\leq C$ for all $i\in I$ and some $C>0$. Suppose furthermore that there exists families $(T_{i,j})_{i\in I}(j=1,\ldots, N)$ and $(S_i)_{i\in I}$ of measurable subsets of $X$ such that $T_{i,j}\cap  T_{i' , j}=\es, j=1,\ldots, N,$ and  $S_{i}\cap S_{i'}=\es$ whenever $i\ne i'$. Then the linear operator
\begin{eqnarray*}
 \si=\sum_{j=1}^N\sum_{i\in I}M_{T_{i,j}}\circ \si_i\circ M_{S_i}
 \end{eqnarray*}
is bounded by $NC$.
\end{proposition}

\begin{proof} Let us write
\begin{eqnarray*}
 \si^j:=\sum_{i\in I}M_{T_{i,j}}\circ \si_i\circ M_{S_i} \quad \text{for} \quad j=1,\cdots, N.
 \end{eqnarray*}
Then $\si=\sum_{j=1}^N\si^j$ and
for $\xi\in L^2(X,\mu),j\in \{1,\ldots, N\}$ we have
\begin{eqnarray*}
 \no{\si^j(\xi)}_2^2&=&\int_X\val{\sum_{i\in I}M_{T_{i,j}}\circ \si_i\circ M_{S_i}(\xi)(x)}^2d\mu (x)\\
 &=&\sum_{i\in I}\int_{T_{i,j}}\val{\si_i(M_{S_i}(\xi))(x)}^2d\mu (x)\\
  &\leq &\sum_{i\in I}\int_X\val{\si_i(M_{S_i}(\xi))(x)}^2d\mu (x)\\
 &\leq&\sum_{i\in I}C^2\int_X\val{ M_{S_i}(\xi)(x)}^2d\mu (x)\\
 &=&C^2(\sum_{i\in I}\int_{S_i}\val{\xi(x)}^2d\mu (x))\\
 &\leq&C^2\int_X\val{\xi(x)}^2d\mu (x)\\
 &=&C^2\no\xi_2^2.
 \end{eqnarray*}
Hence $\noop\si\leq N C$.
\end{proof}

\subsection{The open orbits}\label{open}

\subsubsection{The orbit $ \Omega_{\ve,-\ve}= \Om_{\ve S^*-\ve Q^*} $ for $\ve= \pm 1$}\label{vemve}

Let $ \ell_{\ve,-\ve}:=\ve(S^*-Q^*),\ve\in \{1,-1\} $. We have seen in (\ref{6d1}) and (\ref{6d2}) that  the $ G $-orbit of $ (\ell_{\ve,-\ve}){\res\h} $ is the subset
\begin{eqnarray*}
\Om_{\ve,-\ve}{\res\h}&=&\left\{\ve\left ((-e^{-b}+\frac{e^{-b}p^2}{2})Q^*+(-e^{-\frac{a+b}{2}}p)R^*+e^{-a}S^*\right),  a,b,p\in\R\right\}.
\end{eqnarray*}
By Subsection \ref{openor}, the  boundary $ \partial(\ve,-\ve)  $ of the orbit of $
(\ell_{\ve,-\ve})\res\h $ is given by the union of the 5 orbits
\begin{eqnarray*}
G\cdot\ve S^*\res\h&=&\{E(a,b,p)\cdot \ve S^*=
(e^{-b}\frac{\ve}{2}p^2)Q^*+  (-\ve pe^{-\frac{a+b}{2}})R^*+\ve e^{-a}S^*) ;\ a,b,p\in\R\},\\
G\cdot \ve Q^*\res\h &=& \{E(a,b,p)\cdot \ve Q^*= \ve e^{-b}Q^*;\ a,b,p\in\R\}, \\
G\cdot(-\ve Q^*)\res\h &=& \{-E(a,b,p)\cdot \ve Q^*=- \ve e^{-b}Q^*;\ a,b,p\in\R\},\\
G\cdot\sqrt 2 \ve R^*\res\h&=&\{E(a,b,p)\cdot (\sqrt 2 \ve R^*)=
\ve ((-e^{-b}\sqrt 2 p)Q^*+  (e^{\frac{-a-b}{2}}\sqrt 2)R^*);\ a,b,p\in\R \},\\
G\cdot(-\sqrt 2\ve R^*)\res\h&=&\{-E(a,b,p)\cdot (\sqrt 2 \ve R^*)=
\ve ((e^{-b}\sqrt 2 p)Q^*  -(e^{\frac{-a-b}{2}}\sqrt 2)R^*);\ a,b,p\in\R \}
\end{eqnarray*}
and $\{0\}$. Here we denote $E(a,b,c)=\exp(aA)\exp(bB)\exp(pP)$ for
$(a,b,p)\in\R$.

We must now consider the following subsets of $ \R^3 $, which give us the
points in the orbit $ {\OM_{\ve,-\ve}}\res\h $   ``close'' up to a distance $ \de $ to
the corresponding boundary orbits.
\begin{definition}\label{}

\begin{enumerate}\label{}
\item  For $ k \in\Z $, denote by
$$ J_{\de,k}:=[\log(\frac{1}{\de^{6}}),+\iy[ \ti ]-\iy, \log(\frac{1}{\de})] \ti I_{\de^2,k}$$ where $I_{c,l}:=[c l,cl+c [\subset \R$ for $l\in\Z, c>0$ and let
\begin{eqnarray*}
 L_{\de,k}&:=&\R\ti\R\ti  I_{\de^2,k}.
\end{eqnarray*}

\item Let
\[K_\de=[-\frac{1}{\de},\frac{1}{\de}] \ti [-\frac{1}{\de},\frac{1}{\de}]
\ti[-\frac{1}{\de^{1/2}},\frac{1}{\de^{1/2} }].  \]

\item
Let $ S_{\de,1}:=\{(a,b,p); e^{-a}>\de^6\} (\textrm{corresponds to } \R_{\ve S^*})$.
\item
Let $ S_{\de,2}:=\{(a,b,p); e^{-a}\leq\de^6,e^{-b}<\de\}(\textrm{corresponds to } \R_{\ve S^*})$.

\item Let
\begin{eqnarray*}
 S_{\de,3,k,+}:= J_{\de,k}\cap \{(a,b,p);
\ve e^{-b/2}(1- \frac{p^2}{2})>\de^{1/2}, \de\leq e^{-b}\leq
\frac{1}{\de^{{5/4}}}\} \text{ for } k\in\Z,
\end{eqnarray*}
 $  (\textrm{corresponds to } \R_{\ve Q^*})$ and $S_{\de,3,+}:= \bigcup_k S_{\de,3,k,+}$.

\item Let
\begin{eqnarray*}
 S_{\de,3,k,-}:=J_{\de,k}\cap\{(a,b,p);
\ve e^{-b/2}(1-\frac{ p^2}{2})<-\de^{1/2}, \de\leq e^{-b}\leq
\frac{1}{\de^{{5/4}}}\} \text{ for } k\in\Z,
\end{eqnarray*}
$(\textrm{corresponds to } \R_{-\ve Q^*})$ and $S_{\de,3,-}:= \bigcup_k S_{\de,3,k,-}$.

\item Let
\begin{eqnarray*}
S_{\de,4,\pm}:= \{(a,b,p)\in\R^3;  e^{-a}<\de^6, e^{-b}\geq\de,
e^{-b/2}\vert \frac{ p^2}{2}-1\vert\leq\de^{1/2},  \pm p\leq 0 \}.
\end{eqnarray*} $(\textrm{corresponds to } \R_{\pm R^*}) $.

\item

Define the corresponding multiplication operators on $L^2(\R^3) $ by $:
M_{\de,i}=M_{S_{\de,i}}$ for $i=1,2 $, respectively; for
$k\in \Z $, define $ M_{\de,3,k,\pm}=M_{S_{\de,3,k,\pm}}$,$ M_{\de, 3,\pm}=M_{S_{\de,3,\pm}}$ and  $
M_{\de,4,\pm}=M_{S_{\de,4,\pm}} $.

\end{enumerate}

\end{definition}

\begin{remark}\label{iplusi}
\begin{enumerate}\label{propofS}$  $
\item For any $ \de>0 $, the sets $ \{S_{\de,1}, S_{\de,2},S_{\de,3,k,\pm},S_{\de,4,\pm}; k\in
\Z\}$ are pairwise disjoint measurable subsets of  $ \R^3 $.
\item  The region
\begin{eqnarray*}
  \{(a,b,p);
e^{-b/2}\vert \frac{ p^2}{2}-1\vert> {\de^{1/2}}, e^{-b}>
\frac{1}{\de^{5/4}}\}
\end{eqnarray*}
does not need to be considered, since for  a function $ F\in L^1_c $ the kernel
function $ \wh F^{\h}(s-a,t-b,
e^{\frac{a-b}{2}}(v-p),\ve(-e^{-b}+\frac{e^{-b}p^{2}}{2})Q^*-(e^{-\frac{a+b}{2}}
p)R^*+e^{-a}S^*)  $ of the operator $ \pi_{\ve,-\ve}(F) $ is 0 for $ \de $
small enough on this region.
\item
Recall that for $ \al=(a,b,p)$ and $\be=(s,t,v)\in\R^3 $, the multiplication of $ E(s,t,v)\cdot E(a,b,p) $ is given by
\[
E(s,t,v)\cdot E(a,b,p)=E(s+a,t+b, e^{\frac{-a+b}{2}}v+p).
\]
For $\be\in K_{\de}$ and $\al\in S_{\de,k,3,\pm}$, we have that $ e^{{-a}}\leq \de^6, e^b< \frac{1}{\de}, \val v<\frac{1}{\de^{1/2}} $ and so $ e^{{\frac{-a+b}{2}}}\val v\leq \de^2$ and
\begin{eqnarray}\label{suminter}
\be\cdot \al
&=&
(s+a,t+b, e^{\frac{-a+b}{2}}v+p)\\
\nn &\in &
\R\ti\R\ti( [-\de^2,\de^2[+I_{\de^2,k})\\
\nn&\subset& \bigcup_{i=-1}^1L_{\de,k+i}=: T_{\de,3,k}.
\end{eqnarray}

We see that the set $T_{\de,3,k}  $ is  the  union of the $ 3 $
disjoint boxes $ L_{\de, k+i}$ for $i\in\{-1,0, 1\} $.

Let $ N_{\de,3,k} $ be the multiplication operator with the characteristic function of the set $ T_{\de,3,k}$ for all $k\in\Z $.

\item
The irreducible representation $  \pi_{\ve,-\ve}=\ind_ H ^G \ch_{\ve(S^*-Q^*)} $ acts on a $ \xi\in L^2(G/H,\ch_{\ve(S^*-Q^*)}) $ in the following way (see \eqref{kerS}):
\begin{eqnarray*}\label{transpiveve}
\nn &&\pi_{\ve,-\ve}(F)\xi(E(s,t,v))\\
\nn&=&
\int_{\R^3}e^{\frac{1}{2}(a-b)}\wh F^{\h}(s-a,t-b, e^{\frac{a-b}{2}}(v-p),\ve( -e^{-b}(1-\frac{p^{2}}{2})Q^*+(-e^{-\frac{a+b}{2}}p)R^*+e^{-a}S^*))\\
&&\xi(E(a,b,p))dadbdp,\ F\in L^1(G).
\end{eqnarray*}
\item  Let $ F\in L^1_c $. There exists a box $ K=[-M,M]^3 $ for some $ M>0 $ such that
\begin{eqnarray*}
\widehat F^\h(s,q)=0 \text{ for } s\not\in K \text{ and } q\in\h^*.
\end{eqnarray*}
Hence it follows from the relations (\ref{suminter}) that for $ \de>0 $ small enough,
\begin{eqnarray*}
\widehat F^\h(st\inv,t\cdot (\ve(S^*-Q^*)))=0, \text{ for } t\in S_{\de,3,k,\pm}, s\not \in  T_{\de,3, k} \text{ and } k\in\Z.
\end{eqnarray*}
Therefore,
\begin{eqnarray}\label{multrel}
N_{\de,3,k}\circ \pi_{\ve,-\ve}(F)\circ M_{\de,3,k,\pm}=\pi_{\ve,-\ve}(F)\circ M_{\de,3,k,\pm} \, \text{ for }\, k\in \Z.
\end{eqnarray}
\end{enumerate}
\end{remark}

\begin{definition}\label{defsiplmi}
Let $\sigma_l:=\ind_H^G \chi_l$ for
$l= \ve S^*$, $\ve(\pm\sqrt 2 R^*-2 Q^*)$ and $\ve (-1+\frac{\de^4 k^2}2)Q^*$, respectively.
We define the linear operator $ \sigma_{\ve,-\ve,\de}(F) $ by the rule
\rm   \begin{eqnarray*}
 &&\si_{\ve,-\ve,\de}(F) :=
 \sum_{i=1}^{2}\si_{\ve S^*}(F)\circ M_{\tilde \va_i}\circ M_{\de,i}+\si_{\de,3}(F)+\si_{\de,4,+}(F)+\si_{\de,4,-}(F),
\end{eqnarray*}
where $\si_{\de,4,\pm}(F)=\si_{\ve(\pm\sqrt{{2}}R^*-2 Q^*)}(F)\circ M_{\de,4,\pm}$ and
\begin{eqnarray*}
&&\si_{\de,3}(F):=\sum_{k\in\Z}\Big(N_{\de,3,k}\circ \si_{\ve(-1+\frac{\de^4 k^2}2)Q^*}(F)
\circ (M_{\de,3,k,+}+M_{\de,3,k,-})\Big),
\end{eqnarray*}
where the operators $ M_{\tilde \va_i}, i=1,2$, are multiplication operators
with the functions $ \tilde \va_1(a,b,p):=\va(-e^{-b}), \tilde \va_2=1 $ and $ \va:\R\to \R $ is a $
C^{\iy} $-function with support contained in $ [-1,1]  $ such that $ \va([-\frac{1}{2}, \frac{1}{2}] )=\{1\}
$ and $ \no\va_\iy=1 $.
\end{definition}

The kernel functions of the different operators appearing in the sum above are given by
\begin{enumerate}
\item $\si_{\ve S^*}(F)\circ M_{\tilde \va_i}$ gives
\begin{eqnarray}\label{defveve}
\nn F_{\de,i}((s,t,v),(a,b,p))
 =& \wh F^{\h}(s-a,t-b, e^{\frac{a-b}{2}}(v-p),\ve( (\frac{e^{-b}p^{2}}{2})Q^*+(-e^{-\frac{a+b}{2}}p)R^*+e^{-a}S^*))\\
\nn& e^{\frac{1}{2}(a-b)}\va_i(a,b,p) \, \text{ for } i= 1, 2;
\end{eqnarray}
\item $\si_{\ve(\pm \sqrt{2}R^*-2 Q^*)}(F)\circ M_{\de,4,\pm}$ gives
\begin{eqnarray}
&\nn F_{\de,4,\pm}((s,t,v),(a,b,p))
=& \wh F^{\h}(s-a,t-b, e^{\frac{a-b}{2}}(v-p), \ve( -e^{-b}(\pm\sqrt 2 p +2))Q^*\pm \ve(e^{-\frac{a+b}{2}}\sqrt2) R^*)\\
\nn&& e^{\frac{1}{2}(a-b)}1_{S_{\de,4,\pm}}(a,b,p); \text{ and }
\end{eqnarray}
\item $\si_{\ve(-1+\frac{\de^4 k^2}{2})Q^*}(F)\circ M_{\de,3,k,\pm}$ gives
\begin{eqnarray}
& \nn F_{\de,3,k,\pm}((s,t,v),(a,b,p))=& \wh F^{\h}(s-a,t-b,e^{\frac{a-b}{2}}(v-p),\ve (e^{-b}(-1+\frac{\de^4 k^2}{2})Q^*)\\
\nn&& e^{\frac{1}{2}(a-b)}1_{S_{\de,3,k,\pm}}(a,b,p) \text{ for } k \in \Z.
\end{eqnarray}
\end{enumerate}

Now the representation $ \si_{\ve\vert 1-\frac{k^2\de^4}{2}\vert Q^*}$ is
equivalent to the representation $ \si_{\ve Q^*} $, since both linear
functionals are in the same $ G $-orbit. Therefore,
for any $ k\in \Z$ and $1>\de>0 $:
\begin{eqnarray}\label{sikde}
\noop{N_{\de,3,k}\circ \si_{\ve|1-\frac{\de^4 k^2}{2}|Q^*}(F) \circ M_{\de,3,k,\pm})}&\leq&\noop{\si_{ \ve Q^*}}(F).
\end{eqnarray}

\begin{proposition}\label{se3bou}
For any $ F\in L^1_c $ and $\de>0$, the operator $ \si_{\de,3}(F) $ is bounded in norm by $
C \noop{\si_{Q^*}(F)} $, where  $ C $ is the number given by $ \max_{l\in \Z}\#\{k\in \Z, T_{\de,3,k}\cap
T_{\de,3,l}\ne\es\}$ which is less and equal $9$.
\end{proposition}

\begin{proof}
Apply Proposition \ref{sumsofop}.

\end{proof}

\begin{proposition}\label{ve-sign}
For any $ F\in L^1_c$ and $\de>0$, we have that
\begin{eqnarray*}
\noop{\sigma_{\ve,-\ve,\de} (F)}\leq 2\noop{\si_{\ve S^*}(F)}+C\noop{\si_{\ve Q^*}(F)}+ 2\noop{\si_{\ve R^*}(F)}
\end{eqnarray*}
and  we can extend the mapping $ F\mapsto \si_{\ve,-\ve,\de}(F)$ to every $ F\in C^*(G)$, then
\[ \lim_{\de\to 0}\dis(\pi_{\ve(S^*-Q^*)}(F)-\sigma_{\ve,-\ve,\de} (F),\K(L^{2}(\R^3)))=0 . \]
Here $\dis(a, \K(\H))$ means the distance of an operator $a \in B(\H)$ to the algebra $\K(\H)$ of compact operators on the Hilbert space $\H$.
\end{proposition}

\begin{proof} We have that:
 \begin{eqnarray*}
\noop{\sum_{i=1}^{2}\si_{\ve S^*}(F)\circ M_{\tilde \va_i}\circ M_{\de,i}}
\leq 2 \noop{\si_{\ep S^*}(F)}.
\end{eqnarray*}
Similarly, we have
$\noop{\sum_{j= \pm 1}\si_{j\sqrt{2}R^*-2Q^*}(F)\circ M_{\de,4,\pm}}\leq 2 \noop{\si_{\ep R^*}(F)}$, since $ \pm\sqrt 2 R^*-2Q^* $ is contained in the $ G $-orbit of $\pm  R^* $. Together with Proposition \ref{se3bou}, the inequality holds.

The kernel function $ D_{\ve,-\ve,\de,1} $ of the operator $ \pi_{\ve(S^*-Q^*)}(F)-\si_{\ve S^*}(F)\circ M_{\tilde \va_1}\circ M_{\de,1}$ is given by
\begin{eqnarray*}
&&D_{\ve,-\ve,\de,1}((s,t,v)(a,b,p))\\
\nn &=&e^{\frac{1}{2}(a-b)}\left(\wh F^{\h}(s-a,t-b, e^{\frac{a-b}{2}}(v-p),\ve(-e^{-b}+\frac{e^{-b}p^{2}}{2})Q^*-(e^{-\frac{a+b}{2}}p)R^*+e^{-a}S^*)\right.\\
\nn & &-
\wh F^{\h}(s-a,t-b, e^{\frac{a-b}{2}}(v-p),\ve( e^{-b}\frac{p^{2}}{2})Q^*-(e^{-\frac{a+b}{2}}p)R^*+e^{-a}S^*)\va(-e^{-b})1_{S_{\de,1}}(a,b,p)\Big).
\end{eqnarray*}
Hence
\begin{eqnarray*}
&&
\no{(\pi_{\ve,-\ve}(F)-\si_{\ve S^*}(F)\circ M_{\tilde \va_1}\circ M_{\de,1})\circ M_{\de,1}}_{\rm H-S}^2\\
&=&
\int_{\R^5\ti\{a; e^{-a}>\de^6\}}\vert(\wh F^{\h}(s,t,v,\ve(-e^{-b}+\frac{e^{-b}p^{2}}{2})Q^*-(e^{-\frac{a+b}{2}}p)R^*+e^{-a}S^*)\\
&&
-\wh F^{\h}(s,t,v,\ve( e^{-b}\frac{p^{2}}{2})Q^*-(e^{-\frac{a+b}{2}}p)R^*+e^{-a}S^*)\va(-e^{-b})\vert^2
e^{a-b}dsdtdvdadbdp\\
&=&
\int_{\R^5\ti\{a; e^{-a}>\de^6\}}\vert(\wh
F^{\h}(s,t,v,\ve(-e^{-b}+\frac{e^{a}p^{2}}{2})Q^*-pR^*+e^{
-a}S^*)\\
&&
-\wh F^{\h}(s,t,v,\ve(
e^{a}\frac{p^{2}}{2})Q^*-pR^*+e^{-a}S^*)\va(-e^{-b}
)\vert^2
e^{\frac{3a-b}{2}}dsdtdvdadbdp\\
&\leq&
\int_{\R^5\ti\{a; e^{-a}>\de^6\}}2\vert\Big (\wh F^{\h}(s,t,v,\ve(-e^{-b}+\frac{e^{a}p^{2}}{2})Q^*- pR^*+e^{-a}S^*)
\Big)(1-\va(-e^{{-b}})\vert^2
e^{\frac{3a-b}{2}}  
dsdtdvdadbdp\\
&&+ \int_{\R^5\ti\{a; e^{-a}>\de^6\}}2\vert\Big(\wh
F^{\h}(s,t,v,\ve(-e^{-b}+\frac{e^{a}p^{2}}{2})Q^*-pR^*+e^{
-a}S^*)\\
&&
-\wh F^{\h}(s,t,v,\ve(
e^{a}\frac{p^{2}}{2})Q^*-p R^*+e^{-a}S^*)\Big)\va(-e^{-b}
)\vert^2
e^{\frac{3a-b}{2}}
dsdtdvdadbdp\\
&\leq&
\int_{\R^4\ti\{b; e^{-b}>\frac12\}\ti\{a; e^{-a}>\de^6\}}\vert
\al(s,t,v)\be(-e^{-b}+\frac{e^{a}p^2}{2},  p,e^{-a})\vert^2
e^{\frac{3a-b}{2}}dsdtdvdadbdp
\\
&& +
\int_{\R^4\ti\{b; b>-C\}\ti\{a; e^{-a}>\de^6\}}e^{-2b}\vert
\al(s,t,v)\be(\frac{e^{a}p^2}{2},  p,e^{-a})\vert^2
e^{\frac{3a-b}{2}}
dsdtdvdadbdp\\
&<&\iy
\end{eqnarray*}
for two continuous functions $ \al,\be $ on $ \R^3 $ with compact support and some $ C>0 $.
This shows that the operator $ (\pi_{\ve(S^*-Q^*)}-\si_{\ve S^*}(F)\circ M_{\tilde \va_
1}\circ M_{\de,1})\circ M_{\de,1} $ is Hilbert-Schmidt, therefore, it is compact.

Let us estimate the operator norm of the linear endomorphism
\begin{eqnarray*}
D_{\ve,\de,3,l,+}:=(\pi_{\ve,-\ve}(F)-N_{\de,3,l}\circ\si_{\ve(-1+\frac{\de^{4}
l^2}{2} )Q^*})\circ M_{\de,3,l,+}.
\end{eqnarray*}
We have seen in relation (\ref{multrel}) that
\begin{eqnarray*}
\pi_{\ve,-\ve}(F)\circ M_{\de,3,l,+}= N_{\de,3,l}\circ \pi_{\ve,-\ve}(F)\circ M_{\de,3,l,+}.
\end{eqnarray*}
The kernel function $ F_{\ve,\de, 3,l,+} $ of $ D_{\ve,\de, 3,l,+} $ is given by
\begin{eqnarray*}
&&
F_{\ve,\de,3,l,+}(s,t,v,a,b,p)\\
&=&
\left(\wh
F^{\h}(s-a,t-b,e^{\frac{a-b}{2}}(v-p),\ve(e^{-b}(-1+\frac{p^{2}}{2}
)Q^*-(e^ { -\frac{a+b}{2}}p)R^*+e^{-a}S^*)\right.\\
&&
-\left.\wh F^{\h}(s-a,t-b,e^{\frac{a-b}{2}}(v-p),  \ve
(e^{-b}(-1+\frac{\de^4
l^2}{2})Q^*)\right)\\
&&e^{\frac{1}{2}(a-b)}1_{S_{\de,3,l,+}}(a,b,p)1_{T_{\de,3,l}}(s,t,v).
\end{eqnarray*}

On the set $ S_{\de,3,l,+} $, we have that $ \frac{1}{\de^{{3/4}}} \geq e^{-b}\geq \de $, $
e^{-b/2}\val{1-\frac{p^2}{2}}>\de^{1/2} $ and $ \val {p-l \de^2}<\de^2 $. Since $ F\in L^1_c
$, there exists $ \rh>0 $ such that $ \wh F^\h(s,t,v, q Q^*+r R^*+z S^*)=0 $
whenever $ \val q>\rh\textrm{ or }\val r>\frac{1}{2}\rh $.

If now $ e^{-b/2}>\frac{1}{\de^{(3/4)}}\geq \frac{\rh}{\de^{1/2}}$ (for $\de$ small enough), then
 \begin{eqnarray*}
 \wh
F^{\h}(s-a,t-b,e^{\frac{a-b}{2}}(v-p),\ve(e^{-b}(-1+\frac{p^{2}}{2}
)Q^*-(e^ { -\frac{a+b}{2}}p)R^*+e^{-a}S^*)=0.
 \end{eqnarray*}
If $ p $ is large, then $ e^{-b}({\frac{p^2}{2}}-1)>\rh \Rightarrow F_{\ve,\de, 3,l,+}(s,t,v,a,b,p)=0$ tells us that we can assume that $ e^{-b}\val p\leq e^{-b}(\frac{p^2}{2}-1)\leq \rh$ for every $ p $ giving  a nonzero value to the function $ F_{\de,3,k,+} $.
For $ \val p $ small, we have then $ e^{-b}(\val p+\val l\de^2)\leq \frac{ C}{\de^{1/2}} $.

Since $ e^{-a}\leq \de^6 $ in the set $ S_{\de,3,l,+} $, it follows that $ e^{-\frac{a+b}{2}}\val p<\de $ whenever $  F_{\de,3,l,+}(-,e^{-\frac{a+b}{2}} p,-)\ne 0 $.
According to Remark \ref{remark_pre} we can now write
\begin{eqnarray*}
&&
\vert \wh F^{\h}(s-a,t-b,e^{\frac{a-b}{2}}(v-p),\ve(-e^{-b}+\frac{e^{-b}p^{2}}{2})Q^*-(e^{-\frac{a+b}{2}}p)R^*+e^{-a}S^*)\\
&&- \wh F^{\h}(s-a,t-b,e^{\frac{a-b}{2}}(v-p),\ve (-e^{-b}+e^{-b}\frac{\de^4
l^2}{2})Q^*)\vert\\
&\leq&(e^{-b}\vert\frac{\de^4 l^2-p^2}{2}\vert+\vert
e^{{\frac{-a-b}{2}}}p\vert+e^{-a})\ps(s-a,t-b,e^{\frac{a-b}{2}}(v-p))
\end{eqnarray*}
for some non-negative $ C^\iy $-function $ \ps $ with compact support contained in $ [-M,M] ^3 $.
Hence for any $ (a,b,p)\in\R^3 $:
$$
 \val{F_{\ve,\de,3,l,+}(s,t,v,a,b,p)} \leq M \de \ps(s-a,t-b,e^{\frac{a-b}{2}}(v-p))e^{\frac{1}{2}(a-b)},\\
$$
for all $a,b,p,s,t,v\in\R$ and some $M>0$ independent of $l,F$, and $\de$ small enough.
According to the Young's inequality, this tells us that
\begin{eqnarray*}
 \noop{(\pi_{\ve,-\ve}(F)-N_{\de,3,l}\circ\si_{\ve(-1+\frac{\de^{4}
l^2}{2} )Q^*})\circ M_{\de,3,l,+}}\leq M \de \no \ps_1,
 \end{eqnarray*}
for $l\in\Z, F\in L^1_c$ and $\de>0$ small enough.

We have seen in \eqref{suminter} that $ F_{\ve,\de,3,l,+}(s,t,v,a,b,p)=0 $
whenever $ (s,t,v)\not\in T_{\de,3,l} $.
Consequently for any sufficiently small $1>\de> 0$ we see that
\begin{eqnarray*}
\noop{(\pi_{\ve,-\ve}(F)-\si_{\de,3}(F))\circ(M_{\de,3,+}+M_{\de,3,-})}
&\leq& C\de
\end{eqnarray*}
for some constant $ C>0 $.

For $ i=2 $, the kernel function $ D_{\de,2} $ of the operator $(\pi_{\ve, -\ve}(F)-\si_{\ve S^*}(F))\circ M_{\de,2}  $ is given by
\begin{eqnarray*}
&&D_{\de,2}(s,t,v,a,b,p)\\
&=&
e^{\frac{1}{2}(a-b)}\left(\wh F^{\h}(s-a,t-b, e^{\frac{a-b}{2}}(v-p)),\ve(-e^{-b}+\frac{e^{-b}p^{2}}{2})Q^*+(-e^{-\frac{a+b}{2}}p)R^*+e^{-a}S^*)\right.\\
\nn &&
-\left. \wh F^{\h}(s-a,t-b, e^{\frac{a-b}{2}}(v-p)),\ve e^{-b}\frac{p^{2}}{2}Q^*- e^{-\frac{a+b}{2}}pR^*+e^{-a}S^*)\right)1_{S_{\de,2}}(a,b,p).
\end{eqnarray*}
Then we have a similar  manipulation:
\begin{eqnarray*}
&&\vert D_{\de,2}(s,t,v,a,b,p)\vert\\
&\leq&e^{\frac{1}{2}(a-b)}\left\vert\big( \wh F^{\h}(s-a,t-b, e^{\frac{a-b}{2}}(v-p),\ve (-e^{-b}+\frac{e^{-b}p^{2}}{2})Q^*+(-e^{-\frac{a+b}{2}}p)R^*+e^{-a}S^*)\right. \\
&&-\left.
 \wh F^{\h}(s-a,t-b, e^{\frac{a-b}{2}}(v-p),(\ve e^{-b}\frac{p^{2}}{2})Q^*+(-e^{-\frac{a+b}{2}}p)R^*+e^{-a}S^*)\big)\right\vert 1_{S_{\de,2}}(a,b,p)\\
&\leq&e^{-b}e^{\frac{1}{2}(a-b)}\vert \ps(s-a,t-b, e^{\frac{a-b}{2}}(v-p))\vert 1_{S_{\de,2}}(a,b,p)\\
&\leq&\de e^{\frac{1}{2}(a-b)}\ps(s-a,t-b, e^{\frac{a-b}{2}}(v-p))
\end{eqnarray*}
for some $ \ps\in C_c(\R^3) $.
Again we have the estimate
\begin{eqnarray}\label{esde7}
\noop{(\pi_{\ve,-\ve}(F)-\si_{\ve S^*}(F))\circ M_{\de,2} }\leq \de \int_{\R^3}\vert \ps(a,b,p)\vert dadbdp.
\end{eqnarray}
Finally for $i=4$, the kernel function $D_{\de,4,k,\pm}$ of the operator
$(\pi_{\ve,-\ve}(F)-\si_{\ve(\pm\sqrt 2 R^*-2 Q^*)}(F))\circ M_{\de,4,\pm}$
 is given by
\begin{eqnarray*}
&&D_{\de,4,\pm}(s,t,v,a,b,p)\\
&=&
e^{\frac{1}{2}(a-b)}\big(\wh F^{\h}(s-a,t-b, e^{\frac{a-b}{2}}(v-p),\ve(-e^{-b}+\frac{e^{-b}p^{2}}{2})Q^*-\ve(e^{-\frac{a+b}{2}}p)R^*+\ve e^{-a}S^*)\\
\nn &&
 -\wh F^{\h}(s-a,t-b, e^{\frac{a-b}{2}}(v-p),-\ve( e^{-b}({\pm\sqrt
2}p+2))Q^*+\ve(e^{-\frac{a+b}{2}}(\pm \sqrt
2))R^*\big)1_{S_{\de,4,\pm}}(a,b,p).
\end{eqnarray*}
Then we have as in the preceding case on the set $ S_{\ve,\de,4,\pm} $ that
\begin{eqnarray*}
e^{-b/2}\vert \frac{p}{\sqrt2}\pm 1\vert \leq
e^{-b/2}\vert \frac{p^2}{2}-1\vert \leq \de^{1/2},
\end{eqnarray*}
and
\begin{eqnarray*}
e^{-b}(\frac{p\pm\sqrt 2}{2})^2&\leq&\frac{1}{2}  e^{-b}(\frac{p^2}{2}-1)^2.
\end{eqnarray*}

\begin{eqnarray*}
&&\vert D_{\de,4,k,+}(s,t,v,a,b,p)\vert\\
&\leq&e^{\frac{1}{2}(a-b)}\left\vert\wh F^{\h}(s-a,t-b, e^{\frac{a-b}{2}}(v-p),\ve(-e^{-b}+\frac{e^{-b}p^{2}}{2})Q^*+\ve(-e^{-\frac{a+b}{2}}p)R^*+e^{-a}S^*)\right. \\
&&-\left.
 \wh F^{\h}(s-a,t-b, e^{\frac{a-b}{2}}(v-p),-\ve e^{-b}(2{\pm\sqrt
2}p)Q^*+\ve(e^{-\frac{a+b}{2}}{(\pm\sqrt 2)})R^*))\right\vert
1_{S_{\de,4,\pm}}(a,b,p)\\
 &\leq&(\vert e^{-b}(-1+\frac{p^2}{2}\pm\sqrt 2 p+2)\vert+e^{-\frac{a+b}{2}}\vert p{\pm\sqrt 2}\vert +e^{-a}) \\
&&
e^{\frac{1}{2}(a-b)}\ps(s-a,t-b, e^{\frac{a-b}{2}}(v-p))1_{S_{\de,4,\pm}}(a,b,p)\\
&\leq&
(e^{-b}\frac{(p\pm\sqrt 2)^2}{2}+e^{-\frac{a+b}{2}}\vert p{\pm\sqrt 2}\vert +e^{-a}) e^{\frac{1}{2}(a-b)}\ps(s-a,t-b, e^{\frac{a-b}{2}}(v-p))1_{S_{\de,4}}(a,b,p)\\
&\leq&
(e^{-b}(\frac{p^2}{2}-1)^2+e^{-\frac{a+b}{2}}\sqrt 2|\frac{p^2}{2}-1|+e^{-a}) e^{\frac{1}{2}(a-b)}\ps(s-a,t-b, e^{\frac{a-b}{2}}(v-p))1_{S_{\de,4}}(a,b,p)\\
&\leq&
(\de+\sqrt 2\de^{7/2}+\de^6) e^{\frac{1}{2}(a-b)}\ps(s-a,t-b, e^{\frac{a-b}{2}}(v-p))1_{S_{\de,4,\pm}}(a,b,p).
\end{eqnarray*}
Hence as before:
\begin{eqnarray}\label{esde3}
\noop{(\pi_{\ve,-\ve}(F)-\si_{-\ve (\pm\sqrt 2R^*+2 Q^*)}(F))\circ M_{\de,4,\pm }}\leq \de \int_{\R^3}\ps(a,b,p)dadbdp.
\end{eqnarray}

We have seen above that $ (\pi_{\ve,-\ve}(F)-\si_{\ve S^*}(F)\circ M_{\tilde \va_1})\circ M_{\de,1} $ is compact for any $ \de> 0 $.
Hence it follows from the estimates above that
\[\pi_{\ve,-\ve}(F)(1-M_{\de,1})-\left(\si_{\ve S^*}(F) \circ M_{\de,2}+ \sum_{j=\pm }\si_{\ve (j\sqrt{2}R^*-2Q^*)}(F) \circ M_{\de,4,j}+ \sum_{j=\pm}\si_{\de,3,j}(F)\right) \]
converges to $0$ as $ \de $ tends to $0$. Therefore,
\begin{eqnarray*}
 \lim_{\de\to 0}\text{dis}(\pi_{\ve(S^*-Q^*)}(F)-\sigma_{\ve,-\ve,\de} (F),\K(L^{2}(\R^3)))= 0 .
\end{eqnarray*}
Since $ L^1_c $ is dense in $ C^*(G) $, the  relation $ \lim_{\de\to 0}\text{dis}(\pi_{\ve(S^*-Q^*)}(F)-\sigma_{\ve,-\ve,\de} (F),\K(L^{2}(\R^3)))= 0 $ remains true for any $ F\in C^*(G) $.
\end{proof}

\subsubsection{The orbits $ \Om_{\ve,\ve}=\Om_{\ve Q^*+\ve S^*}, \ve=\pm 1$ }\label{vempe}

The boundary orbits of $ {\Om_{\ve,\ve}}\res{\h}\subset \h^* $ are the orbits of $ \ve S^* $  and of $\ve  Q^* $ inside $ \h^* $. Therefore we must decompose $ \R^3\simeq G/H $ into three  disjoint subsets. Since this case is similar but much easier to the preceding case, we shall omit the proofs.

\begin{definition}\label{}

\begin{enumerate}\label{}
\item  For $ k \in\Z $, denote by
$$ J_{\de,k}:=[\log(\frac{1}{\de^{6}}),+\iy[\ti ]-\iy, \log(\frac{1}{\de})]\ti
 I_{\de^2,k}$$ where $I_{c,l}:=[c l,cl+c [\subset \R$ for $l\in\Z, c>0$ and let
\begin{eqnarray*}
 L_{\de,k}:=\R\ti\R\ti  I_{\de^2,k}.
\end{eqnarray*}

\item   Let
\[K_\de=[-\frac{1}{\de},\frac{1}{\de}] \ti [-\frac{1}{\de},\frac{1}{\de}]
\ti[-\frac{1}{\de^{1/2}},\frac{1}{\de^{1/2} }].  \]

\item
Let
$ S_{\de,1}:=\{(a,b,p); e^{-a}>\de^6\} (\textrm{corresponds to } \R_{\ve S^*})$.
\item
Let
$ S_{\de,2}:=\{(a,b,p); e^{-a}\leq\de^6,e^{-b}<\de\}(\textrm{corresponds to } \R_{\ve S^*})$.

\item Let
\begin{eqnarray*}
 S_{\de,3,k}:= J_{\de,k}, \  k\in\Z.
\end{eqnarray*}
$ (\textrm{corresponds to } \R_{\ve
Q^*})$ and $S_{\de,3}:= \bigcup_k S_{\de,3,k}$.

\item
Define the corresponding multiplication operators on $L^2(\R^3) $ by $
M_{\de,i}=M_{S_{\de,i}}$ for $i=1,2 $ respectively, and for
$k\in \Z $ : $ M_{\de,3,k}=M_{S_{\de,3,k}}$,$ M_{\de, 3}=M_{S_{\de,3}}$ .
\end{enumerate}
\end{definition}

\begin{remark}\label{iplusipp}
\begin{enumerate}\label{propofS}$  $
\item For any $ \de>0 $, the sets $ \{S_{\de,1}, S_{\de,2},S_{\de,3,k}; k\in
\Z\}$ form a partition of $ \R^3 $.
\item
Recall that for $ \al=(a,b,p),\be=(s,t,v)\in\R^3 $, the multiplication of $ E(s,t,v)\cdot E(a,b,p) $ is given by
\[
E(s,t,v)\cdot E(a,b,p)=E(s+a,t+b, e^{\frac{-a+b}{2}}v+p).
\]
For $\be\in K_{\de},\al\in S_{\de,3,k}$, we have that $ e^{{-a}}\leq \de^6, e^b< \frac{1}{\de}, \val v<\frac{1}{\de^{1/2}} $ and so $ e^{{\frac{-a+b}{2}}}\val v\leq \de^2$ and
\begin{eqnarray}\label{suminterpp}
\be\cdot \al
&=&
(s+a,t+b, e^{\frac{-a+b}{2}}v+p)\\
\nn &\in &
\R\ti\R \ti( [-\de^2,\de^2[+I_{\de,k})\\
\nn&\subset&(\bigcup_{i=-1}^1L_{{\de},k+i})=: T_{\de,3,k}.
\end{eqnarray}
Let $ N_{\de,3,k} $ be the multiplication operator with the characteristic function of the set $ T_{\de,3,k}$ for $k\in\Z $.
We see that the set $T_{\de,3,k}  $ is contained in the  union of $ 3 $
disjoint boxes $ L_{\de, k+i}$ for $i\in\{-1,0, 1\} $.
Hence it follows from the relations (\ref{suminterpp}) that for $ \de>0 $ small enough,
\begin{eqnarray*}
\widehat F^\h(st\inv,t\cdot (\ve(S^*+ Q^*)))=0 \text{ for } t\in S_{\de,3,k}, s\not \in  T_{\de,3, k} \text{ and } k\in\Z.
\end{eqnarray*}
Hence,
\begin{eqnarray}\label{multre2l}
N_{\de,3,k}\circ \pi_{\ve,\ve}(F)\circ M_{\de,3,k}=\pi_{\ve,\ve}(F)\circ M_{\de,3,k} \, \text{ for } k\in \Z.
\end{eqnarray}
\end{enumerate}
\end{remark}

The representation $  \pi_{\ve,\ve}=\ind_ H ^G \ch_{\ve(Q^*+S^*)} $ acts on $ \xi\in L^2(G/H,\ch_{\ve(Q^*+S^*)}) $ in the following way:
\begin{eqnarray*}\label{transpivevee}
\nn &&\pi_{\ve,\ve}(F)\xi(u)\\
\nn&=&
\int_{\R^3}e^{\frac{1}{2}(a-b)}\wh F^{\h}(s-a,t-b, e^{\frac{a-b}{2}}(v-p),\ve( e^{-b}(1+\frac{p^{2}}{2})Q^*+(-e^{-\frac{a+b}{2}}p)R^*+e^{-a}S^*))\xi(E(a,b,p))dadbdp.
\end{eqnarray*}
We define the linear operator $ \sigma_{\ve,\ve,\de}(F) $ by the rule
\begin{eqnarray}\label{vevedef}
 \si_{\ve,\ve,\de}(F):=
 \si_{\ve S^*}(F)\circ M_{\tilde \va_1}\circ M_{\de,1}+ \si_{\ve S^*}(F)\circ M_{\de,2}+\si_{\de,3}(F),
\end{eqnarray}
where $\si_{\de,3}(F):=\sum_{k\in\N} N_{\de,k,3}\circ \si_{\ve(1+\frac{\de^4 k^2}{2})Q^*}(F)\circ M_{\de,3,k}$,
the operator $ M_{\tilde \va_1}$ is the multiplication operator with the function $ \tilde \va_1(a,b,p):=\va(e^{-b}) $ where $ \va:\R \to \R $ is a $ C^{\iy} $-function with support contained in $ [-1,1] $ such that $ \va(0)=1 $ and $ \no\va_\iy=1 $. The kernel functions $ F_{\de,i}$, for $i=1,2,3 $, of these operators are given by
\begin{eqnarray*}\label{defveve}
F_{\de,1}((s,t,v),(a,b,p))
 &:=&
\wh F^{\h}(s-a,t-b, e^{\frac{a-b}{2}}(v-p),\ve( (\frac{e^{-b}p^{2}}{2})Q^*+(-e^{-\frac{a+b}{2}}p)R^*+e^{-a}S^*))\\
&&
e^{\frac{1}{2}(a-b)}\va(e^{{-b}})1_{S_{\de,1}}(a,b,p),\\
F_{\de,3}((s,t,v),(a,b,p))&=&
\sum_{k\in\N}\Big(\wh F^{\h}(s-a,t-b, e^{\frac{a-b}{2}}(v-p),\ve e^{-b}(1+\frac{\de^4 k^2}{2})Q^*)\Big)\\
&&
e^{\frac{1}{2}(a-b)}1_{S_{\de,k,3}}(a,b,p),\\
F_{\de,2}((s,t,v),(a,b,p))&:=&
\wh F^{\h}(s-a,t-b, e^{\frac{a-b}{2}}(v-p),\ve( \frac{e^{-b}p^{2}}{2})Q^*+(-e^{-\frac{a+b}{2}}p)R^*+e^{-a}S^*))\\
&&
e^{\frac{1}{2}(a-b)}1_{S_{\de,2}}(a,b,p).
\end{eqnarray*}

Now the representation $ \si_{\ve(1+\frac{k^2\de^4}{2})Q^*}$ is equivalent to the representation $ \si_{\ve Q^*} $, since both linear functionals are in the same $ G $-orbit. Therefore, for any $ k \geq 0 $ and $\de> 0$,
\begin{eqnarray}\label{sikde1}
\noop{ \si_{\ve(1+\frac{\de^4 k^2}{2})Q^*}(F)\circ M_{\de,3,k}}\leq \noop{\si_{\ve Q^*}(F)}.
\end{eqnarray}
It follows from Proposition \ref{se3bou} that there exists $C>0$ such that the operator $\si_{\de,3}(F)$ is bounded by $C$ for any $F \in L_c^1$.

The proof of the following proposition is similar but much easier than that of Proposition \ref{ve-sign} and will be left to the reader.

\begin{proposition}\label{veboco}
For any $ F\in L^1_c$ and $\de>0$, we have that
\begin{eqnarray*}
\noop{\sigma_{\ve,\ve,\de} (F)}\leq 2\noop{\si_{\ve S^*}(F)}+C\noop{\si_{\ve Q^*}(F)}
\end{eqnarray*}
and  we can extend the mapping $ F\mapsto \si_{\ve,\ve,\de}(F)$ to every $ F\in C^*(G)$, then
\[ \lim_{\de\to 0}\dis(\pi_{\ve(Q^*+S^*)}(F)-\sigma_{\ve,\ve,\de} (F),\K(L^{2}(\R^3)))=0 . \]
\end{proposition}

\subsection{The boundary condition for $ \R_{\ve S^*}$, $\ve= \pm 1$}\label{vesbs}

In this subsection, we consider the subset $\GA_5=\{ \R_{\ve S^*}, \ve = \pm 1 \}$ of $\wh G$. We use the coordinates of $G$ according to the basis $\{X=A+B,B,P,Q,R,S \}$ of $\g$.
We recall that the boundary  of the orbit of $b^*B^*+ \ve S^* $ is the orbit
of the functionals $ a A^*+ \ve Q^*, x^* X^*\pm P^*$ and $a^*A^*+b^*B^*$ for
$a^*,b^*,x^*\in\R$. This tells us that we much find the conditions from these boundary points.

\begin{definition}\label{sisstdef}
Let
\rm   \begin{eqnarray*}
\si_{\ve S^*}&:=&\ind_ H ^G \ch_{\ve S^*}.
\end{eqnarray*}
The kernel function $\K_F $ of the operator $\si_{\ve S^*}(F) $ is given by
\begin{eqnarray*}
 \K_F(s,t,v;x,b,p)=\wh F^{\h}(s-x,t-b,e^{\frac{-b}{2}}(v-p),\ve(\frac{e^{-x-b}p^2}2)
Q^*-\ve(e^{-\frac{2x+b}{2}}p)R^*+\ve e^{-x}S^*)e^{-b/2}.
 \end{eqnarray*}
This representation is equivalent to the direct integral representation
\begin{eqnarray*}
  \ta_{\ve S^*}:=\int_\R ^\oplus\pi_{b^* B^* +\ve S^*}db^*
\end{eqnarray*}
acting on the Hilbert space
\begin{eqnarray*}
\H_{\ta_{\ve S^*}}=\int_\R^\oplus  L^2(G/\exp(\mathbb RB+\h),\ch_{b^*B^*+\ve
S^*})db^*\simeq \int_\R^\oplus  L^2(\R^2)db^*
\end{eqnarray*}
with the norm
\begin{eqnarray*}
\no\xi^2_2=\int_\R  \no{\xi(b^*)}_2^2 db^* \quad \text{for} \quad \xi\in \H_{\ta_{\ve S^*}}.
\end{eqnarray*}

An intertwining operator $ U_{\ve S^*} $ for this equivalence is given by
\begin{eqnarray*}
  U_{\ve S^*}(\xi)(b^*)(g):= \int_\R \xi(g\exp(b
B))e^{\frac{1}{4}b}e^{{- 2 \pi i b^*b}}db, \, \xi\in L^2(G/H, \ch_{\ve
S^*}), g\in G, b^*\in\R.
\end{eqnarray*}
Let $  C(\R,\B)  $ be the $ C^* $-algebra of all continuous uniformly bounded  mappings
from $ \R $ into the algebra $ \B $ of bounded operators on the Hilbert
space $ L^2(\R) $, containing the ideal $
C_0(\R,\K) $ of all continuous mappings from $ \R $ into the algebra of
compact operators on  $ L^2(\R) $ vanishing at infinity.
It follows from Section \ref{contcon} that the image of $ \ta_{\ve
S^*}$ is contained in $ C(\R,\B) $.
On the other hand, we can consider $ C(\R,\B) $ as a subalgebra of $
B(\H_{\ta_{\ve S^*}}) $ as for $ \ph\in C(\R,\B) $, let
\begin{eqnarray*}
\ph(\xi)(b^*):=\ph(b^*)(\xi(b^*)) \text{ for }\xi\in\H_{\ta_{\ve
S^*}}, b^*\in\R.
\end{eqnarray*}
The unitary mapping $ U_{\ve S^*} $ induces a canonical homomorphism $ \rh_{\ve S^*} $
from the algebra $  B(L^2(G/H,\ch_{\ve S^*})) $
onto $ B(\H_{\ta_\ve S^*}) $. This homomorphism is defined on $ \si_{\ve S^*}(a)$ by
\begin{eqnarray}\label{rhyuit}
  \rh_{\ve S^*}(\si_{\ve S^*}(a)) &= & U_{\ve S^*}\circ \si_{\ve S^*}(a)\circ U_{\ve S^*}^*\\
\nn&=&\int_\R^\oplus \pi_{b^*B^*+\ve S^*}(a)db^*\\
\nn&=&\ta_{\ve S^*}(a) \quad \text{for} \quad a\in C^*(G).
\end{eqnarray}
\end{definition}

\begin{definition}\label{comonS}
\rm   Let $ \partial \R_{\ve S^*} $ be the boundary of $\R_{\ve S^*}$ in $ \g^*/G $.
\end{definition}

It is easy to see from the relations \eqref{coadG}  that $ \partial \R_{\ve S^*}=\{Q,R,S\}^{\perp} + (\R^{+,0}\ve) Q^* \subset S^\perp$.

\begin{definition}\label{}
\begin{enumerate}\label{}
\item  For $ k \in\Z $, denote by
$$ J_{\de,k}:=[\log(\frac{1}{\de^6}),+\iy[\ti \R \ti  I_{\de^2,k},$$ where $I_{c,l}:=[c l,cl+c [\subset \R$ for $l\in\Z$ and $c>0$.

\item
Let
$ S_{\de,1}:=\{(x,b,p); e^{-x}>\de^6\}$.

\item  Let $ \de\mapsto r_\de\in \R^+ $ be such that $ \lim_{\de\to 0}r_\de=+\iy
$ and $ \lim_{\de\to 0}e^{r_\de} \de^{{1/2}}=0 $.

\item For a constant $ D>0 $ and $ k=(k_1,k_2,k_3)\in \Z^3 $, let
\begin{eqnarray*}
 S_{\de,D,3,k}:= \{(x,b,p)\in\R^3: e^{-x}\leq\de^6,
 x\in I_{r_\de,k_1}, b\in I_{r_\de,k_2}, p\in
I_{{D\de^2}{e^{r_\de/2(k_1+k_2)}},k_3}\}.
\end{eqnarray*}
\end{enumerate}
\end{definition}

\begin{proposition}\label{kmulsS}
For every compact subset $ K\subseteq \R^3 $, we have that
\begin{eqnarray*}
 KS_{\de,D,3,k}&\subset&
\bigcup_{ {j\in \Z^3}\atop{ \val {j_i}\leq 1,i=1,2,3}}S_{\de,
De^{\frac{r_\de}{2}(-j_1-j_2)}, 3,k+j}=: R_{\de, D,3,k}
\end{eqnarray*}
for every $ k\in\Z^3 $ and $ \de>0 $ small enough.
\end{proposition}

\begin{proof}
Indeed, we have an $ M>0 $ such that $ K\subset  [-M, M]^3$ and then for $
r_\de> M $ and $ (s,t,v)\in K , (x,b,p)\in S_{\de,D,3,k}$, it follows that
\begin{eqnarray*}
 u:=(s,t,v)\cdot (x,b,p) = (s+x,t+b, e^{b/2}v+p),
\end{eqnarray*}
$ (k_1+j_1)r_\de\leq s+x<(k_1+j_1+1)r_\de$ and $ (k_2+j_2)r_\de\leq
t+b<(k_2+j_2+1)r_\de $ for some $ j_1,j_2\in \{-1,0,1\}$.  It follows that
\begin{eqnarray*}
  \vert e^{b/2}v\vert &\leq &\vert e^{-x/2}v\vert e^{(x+b)/2}\\
     &\leq &\de^3\vert v\vert e^{(x+b)/2}\\
   &\leq&  \de^3\vert v\vert e^{r_\de/2(-j_1-j_2)}{e^{r_\de/2(k_1+k_2+j_1+j_2)} }e^{r_\de}\\
&<&
De^{r_\de/2(-j_1-j_2)}\de^2{e^{r_\de/2(k_1+k_2+j_1+j_2)} }
\end{eqnarray*}
for $ \de$ small enough, since $\lim_{\de\to 0}e^{r_\de}\de^{1/2}=0$. Hence,
\begin{eqnarray*}
 p+e^{b}v&<&
(k_3+1)D e^{\frac{r_\de}{2}(-j_1-j_2)}
\de^2 e^{\frac{r_\de}{2}(k_1+k_2+j_1+j_2)} +e^{b}v\\
&<&(k_3+2)De^{\frac{r_\de}{2}(-j_1-j_2)}
\de^2 e^{\frac{r_\de}{2}(k_1+k_2+j_1+j_2)}
\end{eqnarray*}
and also
\begin{eqnarray*}
 p+e^{b}v&\geq&
{k_3De^{\frac{r_\de}{2}(-j_1-j_2)}\de^2}
{e^{\frac{r_\de}{2}(k_1+k_2+j_1+j_2)} }-e^{b}\vert v\vert \\
&\geq &{(k_3-1)De^{\frac{r_\de}{2}(-j_1-j_2)}\de^2}
{e^{\frac{r_\de}{2}(k_1+k_2+j_1+j_2)} }.
\end{eqnarray*}
Hence $ u$ is contained in $R_{\de, D, 3, k}$.
\end{proof}

We define the multiplication operators $M_{\de,1}=M_{S_{\de,1}}$, respectively,
$M_{\de,D,3, k}=M_{S_{\de,D,3,k}}$  and $P_{\de,D,3,k}:=M_{R_{\de,D,3,k}}$  on $L^2(\R^3) $.
Then we have the following proposition.

\begin{proposition}\label{eqnm}
\begin{eqnarray*}
  \ol M_{\de,1}\circ U_{\ve S^*}= U_{\ve S^*}\circ  M_{\de,1} \quad \text{for} \quad \de>0,
\end{eqnarray*}
where $ \ol M_{\de,1}=M_{\{(x, p); e^{-x}>\de^6\}}$ on $ L^2(\R^2) $.
\end{proposition}

\begin{proof}
Indeed, for $ \xi\in L^2(G/H,\ch_{S^*}), b\in\R$ and $g\in G$, we have that
\begin{eqnarray*}
U_{S^*}(M_{\de,1}(\xi)(E(x,p)) &= & \int_\R M_{\de,1}( \xi)(E(x,p)\exp(b
B))e^{\frac{1}{4}b}e^{{- 2 \pi i b^*b}}db\\
&=&\int_\R 1_{\{e^{-x}>\de^6\}}\xi(E(x,p)\exp(b
B))e^{\frac{1}{4}b}e^{{- 2 \pi i b^*b}}db\\
&=&1_{\{e^{-x}>\de^6\}}(E(x,p))\int_\R \xi(E(x,p)\exp(b
B))e^{\frac{1}{4}b}e^{{- 2 \pi i b^*b}}db\\
&=&\ol M_{\de,1}(U_{S^*}\xi)(E(x,p)),
\end{eqnarray*}
here $E(x,p)=\exp(x X)\exp(p P)$ for $x,p\in\R$.
\end{proof}

\begin{definition}\label{sisstardef}
\rm
For $ D>0 $, we define the  linear operator $ \sigma_{\ve,\de, D}(F) $ on $
L^2(G/H,\ch_{\ve S^*})=L^2(\R^3) $ by the rule
\begin{eqnarray*}
  &&\si_{\ve,\de,D}(F) :=\sum_{k\in\Z^3}P_{{\de,D,3,k}}\circ \si_{({k_3^2 \de^3 D^2e^{r_\de(k_1+k_2)}})\ve Q^*}(F)\circ M_{\de,D,3, k}.
\end{eqnarray*}

\end{definition}

Now the representation $\si_{({k_3^2 \de^3 D^2e^{r_\de(k_1+k_2}})\ve Q^*}$ is equivalent
to the representation $ \si_{\ve Q^*} $, since both linear functionals are
in the same $ G $-orbit. Therefore for any $ k,\de $,
\begin{eqnarray}\label{sikde1}
\noop{ \si_{({k_3^2 \de^3 D^2e^{r_\de(k_1+k_2)}})\ve Q^*}(F)\circ
M_{\de,D,3,k})}&\leq&\noop{\si_{\ve Q^*}(F)}.
\end{eqnarray}

\begin{proposition}\label{boundedsetw}
For every $ F\in C^*(G) $ and $\de>0 $, the operator $ \si_{\ve,\de,D}(F) $ is bounded in the operator norm by  $ 3^3\noop{\si_{\ve Q^*}(F)} $.
\end{proposition}
\begin{proof}
For $ k\in\Z^3 $, decompose the set $ R_{\de,D,3,k} $ into a disjoint union of measurable subsets $ R_{k,j}, j= (j_1,j_2,j_3)\in \Z^3$ with $\val{j_i}\leq 1, i=1,2,3 $, such that $ R_{k,j} \subset S_{\de,
De^{\frac{r_\de}{2}(-j_1-j_2)}, 3,k+j} $ for every $ j $. This gives us at most $ 3^3 $ such subsets $ R_{k,j} $ for fixed $k$. These sets $ R_{k,j} $ are disjoint in $ k $ for fixed $ j $, since the sets $S_{\de,
De^{\frac{r_\de}{2}(-j_1-j_2)}, 3,k+j}$ are mutually disjoint in $k$ for fixed $j$.  It suffices then to apply Proposition \ref{sumsofop}.
\end{proof}

\begin{proposition}\label{Rdemulcom}
Let $ a\in C^*(G) $. Then the element $ \ta_{\ve S^*}(a)\circ
\ol M_{\de,1}:=\int_\R^\oplus \pi_{b^*B^*+\ve S^*}(a)\circ \ol M_{\de,1}db^* $ is in
$ C_0(\R,\K) $ for any $ \de>0 $.
\end{proposition}

\begin{proof}
We must prove first  that the operators  $  \pi_{b^*B^*+\ve S^*}(a)\circ \ol  M_{\de,1}$ for all $b^*\in\R$ and $\de>0 $ are all compact.
The kernel function $ F_{b^*} $ of the operator $ \pi_{b^*B^*+ \ve S^*}(F)\circ\ol{M_{\de,1}}
$ is given by Formula (\ref{fdthreeker}):
\begin{eqnarray*}
F_{b^*}((t,u),(x,p))&=&1_{\{
e^{-x}>\de^6\}} \int_{\mathbb R} \widehat{F^{\mathfrak h}}(E(t-x,u-e^{-\frac b2}p)\exp bB)\\
&&\left(\varepsilon e^{-x}\big(\frac12 p^2 Q^*- pR^*+S^*\big)\right)
e^{ib^*b}e^{\frac b4}db.
\end{eqnarray*}
Hence for $ F\in L^1_c $, using Remark \ref{remark_pre}:
\begin{eqnarray*}
  \no{ \pi_{b^*B^*+\ve S^*}(F)\circ \ol M_{\de,1}}_{H-S}^2 &= &\int_{\R^4}1_{\{
e^{-x}>\de^6\}}
\val{ F_{b^*}(t,u,x,p)}^2dxdtdu dp\\
  &=&\int_{\R^4}1_{\{
e^{-x}>\de^6\}}
\val{ \int_{\mathbb R} \widehat{F^{\mathfrak h}}(E(t-x,u-e^{-\frac b2}p)\exp bB)\\
&&\left(\varepsilon e^{-x}\big(\frac12 p^2 Q^*- pR^*+S^*\big)\right)
e^{ib^*b}e^{\frac b4}db}^2dxdtdu dp\\
&\leq &(\int_\R(\int_{\R^4}1_{\{ e^{-x}>\de^6\}}
\vert \widehat{F^{\mathfrak h}} (E(t-x,u-e^{-\frac b2}p)\exp bB)\\
&&\left(\varepsilon e^{-x}\big(\frac12 p^2 Q^*- pR^*+S^*\big)\right)
\vert^2dxdtdu dp)^{\frac{1}{2}}e^{\frac b4}db)^2\\
&=&(\int_\R(\int_{\R^4}1_{\{ e^{-x}>\de^6\}}
\vert \widehat{F^{\mathfrak h}} (E(t,u)\exp bB)\\
&&\left( \ve \big(\frac12 e^xp^2 Q^*- pR^*+e^{-x}S^*\big)\right)\vert^2 e^{ x}
dxdtdu dp)^{\frac{1}{2}}e^{\frac b4}db)^2\\
&<&\iy,
\end{eqnarray*}
since the function
\begin{eqnarray*}
 (t,u,x,p,b)\mapsto 1_{\{e^{-x}>\de^6\}}\widehat{F^{\mathfrak h}} (E(t,u)\exp bB)
\left( \ve \big(\frac12 e^xp^2 Q^*- pR^*+e^{-x}S^*\big)\right) e^{ x}
)e^{\frac b4}
 \end{eqnarray*}
has compact support.

Now we must  show that $ \lim_{b^*\to \iy}\noop{\pi_{b^*B^*+\ve S^*}(F)\circ \ol M_{\de,1} }=0. $
The kernel function $ F_{b^*} $ of the operator $ \pi_{b^*B^*+ \ve S^*}(F)\circ\ol{M_{\de,1}}  $
can be written, using partial integration, as
\begin{eqnarray*}
F_{b^*}((t,u),(x,p))&=&
1_{\{e^{-x}>\de^6\}} \int_{\mathbb R} \widehat{F^{\mathfrak h}}(E(t-x,u-e^{-\frac b2}p)\exp bB)\\
&&
\left(\varepsilon e^{-x}\big(\frac12 p^2 Q^*- pR^*+S^*\big)\right)
e^{ib^*b+\frac{b}{4}}db\\
&=&
1_{\{e^{-x}>\de^6\}} \int_{\mathbb R}\frac{1}{ib^*+\frac{1}{4}} (e^{-b/2}p\partial_2+\partial _B)\widehat{F^{\mathfrak h}}
 (E(t-x,u-e^{-\frac b2}p)\exp bB)\\
&&
\left(\varepsilon e^{-x}\big(\frac12 p^2 Q^*- pR^*+S^*\big)\right)
e^{ib^*b}e^{\frac b4}db\\
&=&
1_{\{e^{-x}>\de^6\}} \int_{\mathbb R} \frac{1}{ib^*+\frac{1}{4}}(e^{-b/2} p\partial_2+\partial _B)\widehat{F^{\mathfrak h}}
 (E(t-x,u-e^{-\frac b2}p)\exp bB)\\
&&
\left(\varepsilon \big(\frac12 e^{-x}p^2 Q^*- e^{-x}pR^*+e^{-x}S^*\big)\right)
e^{ib^*b}e^{\frac b4}db.
\end{eqnarray*}
Therefore,
\begin{eqnarray*}
&&\sup_{x,p}\int_{\R^2}\vert{F_{b^*}}((t,u), (x,p))\vert dt du \\
&\leq&\sup_{x,p}1_{\{
e^{-x}>\de^6\}} \int_{\mathbb R} \frac{1}
{\vert{ib^*+\frac{1}{4}}\vert}\vert (e^{-b/2}e^x(e^{-x} p)\partial_2+\partial _B)
 \widehat{F^{\mathfrak h}}(E(t-x,u-e^{-\frac b2}p)\exp bB)\\
&&\left(\varepsilon \big(\frac12 e^{-x}p^2 Q^*- e^{-x}pR^*+e^{-x}S^*\big) \right)\vert
e^{\frac b4}dbdt du \\
&<&\frac{C}{\vert b^*\vert },
\end{eqnarray*}
for $\val {b^*}\geq 1$ and for some constant $C>0$. Furthermore,
\begin{eqnarray*}
&&\sup_{t,u}\int_{\R^2}\vert{F_{b^*}}((u,v), (x,p))\vert dx dp
\\&\leq&
 \int_{\mathbb R} {\vert
\frac{1}{ib^*+\frac{1}{4}}\vert}1_{\{
e^{-x}>\de^6\}}(x)\vert (e^{-b/2}e^x(e^{-x} p)\partial_2+\partial _B)\widehat{F^{\mathfrak h}}
(E(t-x,u-e^{-\frac b2}p)\exp bB)\\
&&\left(\varepsilon \big(\frac12 e^{-x}p^2 Q^*- e^{-x}pR^*+e^{-x}S^*\big) \right)\vert
e^{\frac b4}dbdxdp\\
&<&\frac{C}{\vert b^*\vert },
\end{eqnarray*}
again because  the function
\begin{eqnarray*}
 (t,u,x,p,b)\mapsto 1_{\{e^{-x}>\de^6\}}\widehat{F^{\mathfrak h}} (E(t-x,u-e^{-\frac{b}{2}}p)
 \exp (bB))
\left( \ve \big(\frac12 p^2 Q^*- pR^*+S^*\big)e^{-x} \right)
e^{\frac b4}
 \end{eqnarray*}
has compact support and is $C^\iy$ in the variables $t,u,p,b$.
Hence, by the Young's inequality, for $ F\in L^1_c $, there exists a constant $ C>0 $
such that $ \noop{\si_{b^*B^*+\ve S^*}(F)}\leq \frac{C}{\val { b^*}}$ for large enough $b^*\in\R^*$.
\end{proof}

\begin{proposition}\label{mde2issmq}
For any $ F\in L^1_c $, we have that
\begin{eqnarray*}
  \lim_{\de\to 0}\noop{\si_{\ve S^*}(F)(1-M_{\de,1})-\si_{\ve,\de,D}(F)}=0.
\end{eqnarray*}
\end{proposition}

\begin{proof}
Let $ F\in L^1_c $. Then there exists $ \rh>0 $ such that
\begin{eqnarray*}
\hat F^{\h}(s-x,t-b,e^{\frac{-b}{2}}(v-p),(\frac{e^{-x-b}p^2}2)
\ve Q^*+(e^{-\frac{2x+b}{2}}p)\ve R^*+e^{-x}\ve S^*)=0
\end{eqnarray*}
if $ e^{-x-b}p^2>\rh $. For $ \de $ small enough, from Proposition \ref{kmulsS} we have that  $ P_{\de,D, 3,k}\circ
\si_{\ve S^*}(F)\circ M_{\de,D,3,k} =\si_{\ve S^*}(F)\circ M_{\de,D, 3,k}, k\in\Z.$ Therefore, we must show that
\begin{eqnarray*}
  \noop{P_{\de,D,3,k}\circ \si_{\ve S^*}(F)\circ M_{\de,D,3,k}- P_{\de,D,3,k}\circ \si
_{{k_3^2\de^4D^2e^{r_\de(k_1+k_2)}}\ve Q^*}(F)\circ M_{\de,D, 3,k}}
\leq E\de,k\in\Z,
\end{eqnarray*}
for some constant $ E>0 $ which is independent of $ \de $ and $ k $.

The kernel function $ F_{\de,D,3,k} $ of the operator $ P_{\de,D,3,k}\circ
\si_{\ve S^*}(F)\circ M_{\de,D,3, k}- P_{\de, D,3,k}\circ \si_{(\frac{k_3^2 \de^4 D^2e^{r_\de(k_1+k_2})}{2})\ve Q^*}(F)\circ M_{\de,D,3,k}$ is given by
\begin{eqnarray*}
  F_{\de,D,3,k}((s,t,v),(x,b,p)) &= &
\left(\wh F^{\h}(s-x,t-b,e^{\frac{-b}{2}}(v-p),(\frac{e^{-x-b}p^2}2)
\ve Q^*+(e^{-\frac{2x+b}{2}}p)\ve R^*+e^{-x}\ve S^*)\right.\\
&&-\left.\wh F^{\h}(s-x,t-b,e^{\frac{-b}{2}}(v-p),e^{-x-b}\frac{k_3^2 \de^4 D^2e^{ r_\de(k_1+k_2)}}{2}\ve Q^*)\right)\\
&& \, e^{-\frac{b}{2}}1_{S_{\de,D,3,k}}(x,b,p)1_{R_{\de,D,3,k}}(s,t,v).
\end{eqnarray*}
Again this gives us the estimate
\begin{eqnarray*}
  \vert F_{\de,D,3,k}((s,t,v),(x,b,p)) \vert &\leq&
\va(s-x,t-b,e^{\frac{-b}{2}}(v-p))e^{-\frac{b}{2}} \\
&&(\vert
e^{\frac{-x-b}{2}}(p-\de^2k_3De^{{r_\de(k_1+k_2)/2}})\vert \vert
e^{\frac{-x-b}{2}}(p+\de^2k_3De^{{r_\de(k_1+k_2)/2}}) \vert\\
& & \quad + e^{\frac{-2x-b}{2}}\val p+{e^{-x}}),
\end{eqnarray*}
where $\va $ is as in Remark \ref{remark_pre}.
On the sets $ S_{\de,D,3,k} $, if  $ \wh F^\h(-)\ne0 $ for $\de$ small enough, we have the relations:
\begin{eqnarray*}
e^{-x}&\leq &\de^6, \\
e^{-x-b}p^2 &\leq& \rh, \\
e^{1/2(-x-b)}\vert k_3\vert \de^2De^{1/2 r_\de(k_1+k_2)}&\leq&e^{1/2(-x-b)}(\vert p\vert+D\de^2 e^{1/2r_\de (k_1+k_2)})\\
&\leq &\rh^{1/2}+e^{1/2(-x-b)}D\de^2 e^{1/2r_\de (k_1+k_2)})\\
&<&\rh^{1/2}+\de\leq 2\rh^{1/2},\\
\vert e^{1/2(-x-b)}(p-k_3 D\de^2e^{1/2 r_\de(k_1+k_2)}\vert
& \leq &D\de^2e^{(-x-b)/2}e^{1/2 r_\de(k_1+k_2)}\\
&<&D\de^{3/2}.
\end{eqnarray*}
Therefore,
\begin{eqnarray*}
  \vert F_{\de,D,3,k}((s,t,v),(x,b,p)) \vert &\leq&
\va(s-x,t-b,e^{\frac{-b}{2}}(v-p))e^{-\frac{b}{2}}(2\rh^{1/2}D \de^{3/2}+
\rh^{{1/2}}\de^3+\de^6)\\
&\leq&\de\va(s-x,t-b,e^{\frac{-b}{2}}(v-p))e^{-\frac{b}{2}}.
\end{eqnarray*}
\end{proof}

Therefore we can conclude as in Subsection \ref{vempe}.

\begin{corollary}\label{zerobos}
Let $ a\in C^*(G) $.Then
\begin{eqnarray*}
  \lim_{\de\to 0}\dis(\rh_{\ve S^*}(\si_{\ve S^*}(a)-\si_{\ve,\de,D}(a)),C_0(\R,\K))= 0.
\end{eqnarray*}
\end{corollary}

\begin{proof}
Indeed, by Proposition \ref{mde2issmq}, we know that $ \lim_{\de\to
0}\noop{ \si_{\ve S^*}(a)\circ (1-M_{\de,1})-\si_{\ve,\de,D}(a)}=0 $ for any $ a\in
C^*(G) $. By Proposition  \ref{Rdemulcom}, we know that $ \ta_{\ve S^*}(a)\circ \ol M_{\de,1}\in C_0(\R,\K) $. Hence, by Proposition \ref{eqnm} and relation \eqref{rhyuit}:
\begin{eqnarray*}
&&\dis(\rh_{\ve S^*}(\si_{S^*}(a)-\si_{\ve,\de,D}(a))),C_0(\R,\K))\\
&\leq&\noop{\rh_{\ve S^*}(\si_{S^*}(a)-\si_{\ve,\de,D}(a))- \ta_{\ve S^*}(a)\circ \ol M_{\de,1}}
\\
&=&\noop{\rh_{\ve S^*}\Big(\si_{S^*}(a)\circ M_{\de,1}+\si_{S^*}(a)\circ(1- M_{\de,1})-\si_{\ve,\de,D}(a)\Big)- \ta_{\ve S^*}(a)\circ \ol M_{\de,1})}
\\
&=&\noop{\si_{\ve S^*}(a)\circ(1- M_{\de,1})-\si_{\ve,\de,D}(a)}\\
&\to& 0 \quad \text{ as } \quad \de \to 0.
\end{eqnarray*}
\end{proof}

\subsection{The boundary condition for $ \R_{\ve R^*} $}\label{cora}

This case is similar  to the preceding one (and easier). We shall omit most of the proofs.

We take the coordinates on $ G $ coming from the basis $ \{X:=A+B,Y:=A-B,P,Q,R,S\} $.
Then
 \begin{eqnarray*}
[X,P] =0, [Y,P] =P, [X,Q] =Q, [Y,Q] =-Q, [X,R] =R, [Y,R] =0.
\end{eqnarray*}
Hence for $ (x,y,p):=\exp(x X)\exp(y Y)\exp(pP)\in G_0 $, we have
\begin{eqnarray*}
(x,y,p)\cdot \ve R^*\res{\h}=-{e^{-x+y}p}\ve Q^*+e^{-x}\ve R^*.
\end{eqnarray*}
The irreducible representation $ \pi_{\ve R^*+b^* B^*} $ is realized as $ \pi_{\ve R^*+b^*B^*}:=\ind_{L}^G \ch_{\ve R^*+b^* B^*}$, where $ L:=\exp(\text{span}\{Y,\h\}) $.
Let $ \partial \R_{\ve R^*} $ be the boundary of $\R_{\ve R^*}$ in $ \g^*/G $.
It follows from the description of the coadjoint orbits (see (\ref{bdry-4d-2})) that
\begin{eqnarray*}
 \partial \R_{\ve R^*} = (\{S,R\}^\perp)/G.
\end{eqnarray*}

\begin{definition}\label{sisstdef}
Let
\rm   \begin{eqnarray*}
\si_{\ve R^*}&:=&\ind_ H ^G \ch_{\ve R^*}.
\end{eqnarray*}
The kernel function $\K_F $ of the operator $\si_{\ve R^*}(F) $ is then given by
\begin{eqnarray*}
 \K_F(s,t,v;x,y,p)=\wh F^{\h}(s-x,t-y,e^{y}(v-
p),-{e^{-x+y}p} \ve Q^*+e^{-x}\ve R^*)e^y.
 \end{eqnarray*}

This representation is equivalent to the direct integral representation
\begin{eqnarray*}
  \ta_{\ve R^*}:=\int_\R ^\oplus\pi_{b^* B^* +\ve R^*}db^*
\end{eqnarray*}
acting on the Hilbert space
\begin{eqnarray*}
\H_{\ta_{\ve R^*}}=\int_\R^\oplus  L^2(G/L,\ch_{b^*B^*+\ve
R^*})db^*\simeq \int_\R^\oplus  L^2(\R^2)db^*
\end{eqnarray*}
with the norm
\begin{eqnarray*}
\no\xi^2_2=\int_\R  \no{\xi(b^*)}_2^2 db^* \quad \text{for} \quad \xi\in \H_{\ta_{\ve R^*}}.
\end{eqnarray*}
An intertwining operator $ U_{\ve R^*} $ for this equivalence is given by
\begin{eqnarray*}
  U_{\ve R^*}(\xi)(b^*)(g):= \int_\R \xi(g\exp(y
Y))e^{-\frac{1}{2}b}e^{{-2 \pi i b^*y}}dy, \, \xi\in L^2(G/H, \ch_{\ve
R^*}), g\in G, b^*\in\R.
\end{eqnarray*}
Similar to Subsection \ref{vesbs}, we can again consider the C*-algebra $ C(\R,\B) $ and see that the image of $ \ta_{\ve
R^*}$ is contained in $ C(\R,\B) $. Similarly, we can also consider $ C(\R,\B) $ as a subalgebra of $
B(\H_{\ta_{\ve R^*}}) $: for $ \ph\in C(\R,\B) $, let
\begin{eqnarray*}
\ph(\xi)(b^*):=\ph(b^*)(\xi(b^*)) \text{ for }\xi\in\H_{\ta_{\ve
R^*}}, b^*\in\R.
\end{eqnarray*}
The unitary mapping $ U_{\ve R^*} $ induces a canonical homomorphism $ \rh_{\ve R^*} $
from the algebra $  B(L^2(G/H,\ch_{\ve R^*})) $
onto $ B(\H_{\ta_\ve R^*}) $. This homomorphism is defined on $ \si_{\ve R^*}(a)$ by
\begin{eqnarray}\label{rhyuit}
  \rh_{\ve R^*}(\si_{\ve R^*}(a)) &= & U_{\ve R^*}\circ \si_{\ve R^*}(a)\circ U_{\ve R^*}^*\\
\nn&=&\int_\R^\oplus \pi_{b^*B^*+\ve R^*}(a)db^*\\
\nn&=&\ta_{\ve R^*}(a) \quad \text{for} \quad a\in C^*(G).
\end{eqnarray}
\end{definition}

\begin{definition}\label{sisdefr}
\rm
\begin{enumerate}\label{}
\item
Let
$ S_{\de,1}:=
\{(x,y,p); e^{-x}>\de^6\} $.

\item  Let $ \de \mapsto r_\de\in \R^+ $ be such that $ \lim_{\de\to 0}r_\de=+\iy
$ and $ \lim_{\de\to 0}e^{r_\de} \de^{{1/2}}=0 $.
\item For a constant $ D>0 $ and $ k=(k_1,k_2,k_3)\in \Z^3 $, let
\begin{eqnarray*}
&& S_{\de,D,2,k,}:= \{(x,y,p)\in\R^3; e^{-x}\leq\de^6,
 x\in I_{r_\de,k_1}, y\in I_{r_\de,k_2}, p\in
I_{{D\de^2}{e^{r_\de(k_1-k_2)}},k_3}\}.
\end{eqnarray*}

\end{enumerate}

\begin{proposition}\label{kmulsR}
For every compact subset $ K\subseteq \R^3 $, we have that
\begin{eqnarray*}
  KS_{\de,D,2,k}\subset
\bigcup_{ j\in \Z^3\atop  \val {j_i}\leq 1,i=1,2,3}S_{\de,
De^{r_\de(-j_1-j_2)}, 2,k+j}=:  R_{\de, D,2,k}.
 \end{eqnarray*}
for every $ k\in\Z^3 $ and $ \de>0 $ small enough.
 \end{proposition}
\begin{proof}
Indeed, we have an $ M>0 $ such that $ K\subset [-M, M]^3$ and then for $
r_\de> M $, $ (s,t,v)\in K$ and $(x,y,p)\in S_{\de,D,2,k}$. It follows that
\begin{eqnarray*}
 u:=(s,t,v)\cdot (x,y,p) = (s+x,t+y, e^{-y}v+p),
\end{eqnarray*}
$ (k_1+j_1)r_\de\leq s+x<(k_1+j_1+1)r_\de$ and $ (k_1+j_1)r_\de\leq
t+y<(k_2+j_2+1)r_\de $ for some $k=(k_1, k_2)\in \Z^2$ and $ j_1,j_2\in \{-1,0,1\}$. It follows that
\begin{eqnarray*}
  \vert e^{-y}v\vert &\leq &\vert e^{-x}v\vert e^{x-y}\\
&<&
De^{r_\de(-j_1+j_2)}\de^2{e^{r_\de(k_1-k_2+j_1-j_2)} },
\end{eqnarray*}
since for $\de $ small enough  $ M\de^6 e^{2r_\de}< D\de^2$. Hence:
\begin{eqnarray*}
 p+e^{-y}v&<&
{(k_3+1)De^{r_\de(-j_1+j_2)}\de^2}
{e^{r_\de(k_1-k_2+j_1-j_2)} }+e^{-y}v\\
&<&{(k_3+2)De^{r_\de(-j_1+j_2)}\de^2}
{e^{r_\de(k_1-k_2+j_1-j_2)} }
\end{eqnarray*}
and also
\begin{eqnarray*}
 p+e^{-y}v&\geq&
{k_3De^{r_\de(-j_1+j_2)}\de^2}
{e^{r_\de(k_1-k_2+j_1-j_2)} }-e^{-y}\vert v\vert \\
&\geq &{r_\de(k_3-1)De^{(-j_1+j_2)}\de^2}
{e^{r_\de(k_1-k_2+j_1-j_2)} }.
\end{eqnarray*}
Hence $ u$ is contained in the set $R_{\de, D, 2, k}$.
\end{proof}

Define the corresponding multiplication operators on $L^2(\R^2)$ by $M_{\de,1}:=M_{S_{\de,1}}$, $M_{\de,D,k,2}:=M_{S_{\de,D, k,2}}$ and $P_{\de,D,2,k}:=M_{R_{\de,D,2,k}}$ for all $k\in \Z^3$.

\begin{proposition}\label{eqnmr}
We have that
\begin{eqnarray*}
 \ol M_{\de,1}\circ U_{\ve R^*} = U_{\ve R^*}\circ  M_{\de,1} \quad \text{for} \quad \de>0,
\end{eqnarray*}
where $\ol M_{\de,1}:=M_{\{(x,p);e^{-x}>\de^6\}} $.
\end{proposition}

\begin{proof}
Indeed, for $ \xi\in L^2(G/H,\ch_{ \ve R^*}), y\in \R$ and $b\in \R$, we have that
\begin{eqnarray*}
U_{\ve R^*}(M_{\de,1}(\xi))(E(x,p)) &= & \int_\R M_{\de,1}( \xi)(E(x,p)\exp(y Y))e^{-\frac{1}{2}y}e^{{-2\pi i b^*y}}dy\\
&=&\int_\R 1_{\{e^{-x}>\de^6\}}\xi (E(x,p)\exp(y Y))e^{-\frac{1}{2}y}e^{{-2\pi i b^*y}}dy\\
&=&1_{\{e^{-x}>\de^6\}}(E(x,p))\int_\R \xi(E(x,p)\exp(y Y))e^{-\frac{1}{2}y}e^{{-2\pi i b^*y}}dy\\
&=&\ol M_{\de,1}(U_{\ve R^*}\xi)(E(x,p)).
\end{eqnarray*}
\end{proof}

\begin{definition}\label{sidefr}
\rm
For $ D>0 $ and $ F\in L^1_c $, let
\begin{eqnarray*}
\si_{\ve,\de,D,2}(F):=\sum_{k\in\Z^3} P_{\de,D,2,k}\circ \si_{\ve \de^2 k_3
De^{r_\de(k_1-k_2)}Q^*}(F)\circ M_{\de,D,2,k}.
\end{eqnarray*}
\end{definition}

The kernel functions $ F_{\ve,\de,D,2}$ of this operator is given by :
\begin{eqnarray*}\label{defveve}
F_{\ve,\de,D,2}((s,t,v),(x,y,p))
 &=&\sum_{k\in\Z^3}F_{\ve,\de,D,2,k}((s,t,v),(x,y,p))\\
 &:=&
\sum_{k\in\Z^3}\wh F^{\h}(s-x,t-y, e^{-y}(v-p), e^{-x+y}{\de^2k_3 D
e^{r_\de(k_1-k_2)}}\ve Q^*)\\
&&e^{-y}1_{R_{\de,D,2,k}}(s,t,v)1_{S_{\de,D,2,k}}(x,y,p).
\end{eqnarray*}

\begin{proposition}\label{estas usalR}
For every $ F\in L^1_c $ and every small enough $ \de>0$, there exists a constant $ E>0 $ such that
\begin{eqnarray*}
\noop{\si_{\ve,\de,D,2,k}(F)}\leq E \noop{\si_{\ve Q^*}(F)}.
\end{eqnarray*}
\end{proposition}

\begin{proof}
The proof is similar to the proof of Proposition \ref{boundedsetw}.
\end{proof}

\end{definition}

\begin{proposition}\label{Rdemulcomrver}
Let $ a\in C^*(G) $. Then the element $ \ta_{\ve R^*}(a)\circ \ol M_{\de,1}:=\int_\R^\oplus \pi_{b^*B^*+\ve R^*}(a)\circ \ol M_{\de,1}db^* $ is in $ C_0(\R,\K) $ for every $ \de>0 $.
 \end{proposition}
\begin{proof}
Let us prove  that the operators  $  \pi_{b^*B^*+\ve R^*}(a)\circ \ol M_{\de,1},\de>0 $, are all compact.
However, it is easy to see, using Remark \ref{remark_pre}, that for $ F\in L^1_c $:
\begin{eqnarray*}
 && \no{ \pi_{b^*B^*+\ve R^*}(F)\circ \ol M_{\de,1}}_{H-S}^2\\
 &=& \int_{\R^4}1_{\{e^{-x}>\de^6\}}
\val{ \int_{\mathbb R} e^{ib^*y}e^{-\frac y2}\widehat{F}^{\mathfrak h}
(E(t-x,u-e^yp)\exp yY)(-\varepsilon p e^{-x+y}Q^*+\varepsilon e^{-x}R^*)dy}^2dxdtdu dp\\
&\leq &(\int_\R(\int_{\R^4}1_{\{ e^{-x}>\de^6\}}
\vert \widehat{F}^{\mathfrak h}
(E(t-x,u-e^yp)\exp yY)(-\varepsilon p e^{-x+y}Q^*+\varepsilon e^{-x}R^*)
\vert^2dxdtdu dp)^{\frac{1}{2}}e^{-y/2}dy)^2\\
&=&(\int_\R(\int_{\R^4}1_{\{ e^{-x}>\de^6\}}
\vert  e^x\widehat{F}^{\mathfrak h}
(E(t,u)\exp yY)(-\varepsilon p Q^*+\varepsilon e^{-x}R^*)\vert^2
dxdtdu dp)^{\frac{1}{2}}e^{-y}dy)^2\\
&<&\iy.
\end{eqnarray*}
The proof that $ \lim_{b^*\to\iy}\pi_{b^*B^*+\ve R^*}(a)\circ \ol M_{\de,1}db^* =0 $ is similar to that of the corresponding proof of Proposition \ref{Rdemulcom}
\end{proof}

\begin{proposition}\label{mde2issmr}
For any $ F\in L^1_c $, there exists  a constant $ K=K(D,F)>0 $ such that
\begin{eqnarray*}
 \noop{\si_{\ve R^*}(F)\circ (1- M_{\de,1})-\si_{\ve,\de,D,2,k}(F)} \leq K\de
\quad \text{for} \quad \de>0.
\end{eqnarray*}
\end{proposition}
\begin{proof}
Let $ F\in L^1_c $. From Proposition \ref{kmulsR}, we know that for $ \de $ small
enough, $ P_{\de,D,2, k}\circ \si_{\ve R^*}(F)\circ
M_{\de,D,2,k}=\si_{\ve R^*}(F)\circ M_{\de,D,2,k}$ for $k\in \Z^3 $. Hence, it suffices to show that
\begin{eqnarray*}
 \noop{
P_{\de,D,2, k}\circ (\si_{\ve R^*}(F)\circ M_{\de,D,2,k}-\si_{\ve \de^2 k_3
De^{r_\de(k_1-k_2)}Q^*}(F)\circ M_{\de,D,2,k})}\leq L\de,
\end{eqnarray*}
for some constant $ L>0 $ independent of $ \de $.

The kernel function $ F_{\ve,\de,D,2,k} $ of the operator $P_{\de,D,2,
k}\circ (\si_{\ve R^*}(F)\circ M_{\de,D,2,k}-\si_{\ve \de^2 k_3
De^{r_\de(k_1-k_2)}Q^*}(F)\circ
M_{\de,D,2,k}) $ is given by:
\begin{eqnarray*}
 F_{\ve,\de,D,2,k}((s,t,v),(x,y,p)) &= & \left(\wh F^{\h}(s-x,t-y,e^{y}(v-
p),-{e^{-x+y}p\ve }Q^*+e^{-x}\ve R^*)\right.\\
&&- \left.\wh F^{\h}(s-x,t-y, e^{y}(v- p),e^{-x+y}\de^2 k_3
De^{r_\de(k_1-k_2)}\ve Q^*)\right)\\
&&\, 1_{S_{\de,D,2,k}}(x,y,p)1_{P_{\de,D,2,k}}(s,t,v)e^y.
\end{eqnarray*}
Hence,
\begin{eqnarray*}
&&\vert  F_{\ve,\de,D,2,k}((s,t,v),(x,y,p))\vert\\
&\leq& \vert \wh F^{\h}(s-x,t-y,v-e^y p,-{e^{-x+y}p}\ve Q^*+e^{-x}\ve R^*\\
&&- \wh F^{\h}(s-x,t-y, v-e^y p,e^{-x+y}\de^2 k_3
De^{r_\de(k_1-k_2)}\ve Q^*)\vert \\
&&\, 1_{S_{\de,D,2,k}}(x,y,p)1_{P_{\de,D,2,k}}(s,t,v)e^y\\
&\leq&\va(s-x,t-y,v-e^y p)\\
&&(e^{-x+y}e^{r_\de(k_1-k_2)}D\de^2+e^{-x})e^y1_{S_{\de,D,2,k}}(x,y,p)1_{P_{
\de,D ,2,k}}(s,t,v)\\
&\leq&\de \va(s-x,t-y,e^{y}(v-
p))e^{y}1_{S_{\de,D,2,k}}(x,y,p)1_{P_{\de,D,2,k}}(s,t,v),
\end{eqnarray*}
for $ \de $ small enough, where $\va$ is as in Remark \ref{remark_pre}. We conclude as in the preceding cases.

\end{proof}

\begin{corollary}\label{zerobosRstar}
Let $ a\in C^*(G) $. Then
\begin{eqnarray*}
 \lim_{\de\to 0}\dis\big(\rh_{\ve R^*}\big(\si_{\ve
R^*}(a)-\si_{\ve,\de,D,2}(a)\big), C_0(\R,\K)\big)=0.
\end{eqnarray*}
\end{corollary}
\begin{proof}
The proof is similar to that of Corollary  \ref{zerobos}.
 \end{proof}

\subsection{The boundary condition for $ \Om_{\ve P^*+\nu Q^*} $, $ \ve=\pm 1,\nu=\pm 1 $}\label{dfoo}

The boundary of the orbit $\OM_{\ve P^*+\nu Q^*}$ is the union of the orbits
$x^*X^*+\nu Q^*,y^* Y^*+\ve P^*$ for all $y^*,x^*\in\R$, and of the set of characters  (see (\ref{pstqst})).
We take the coordinates on $ G $ coming from the basis $ \{Y:=2A,X:=A+B,P,Q,R,S\} $.
This gives us the bracket relations:
\begin{eqnarray*}
[X,P] =0,\, [X,Q]=Q,\, [Y,P] =P,\, [Y,Q] =0 .
\end{eqnarray*}
The irreducible representation $ \pi_{\ve P^*+\nu Q^*} $ is realized as $ \pi_{\ve P^*+\nu Q^*}:=\ind_{L}^G \ch_{\ve P^*+\nu Q^*}$, where $ L:=\exp(\ll):=\exp(\text{span}\{P,\h\} )$.

\begin{definition}\label{}
\begin{enumerate}\label{}
\item Let $ \m:=\textrm{span}\{Y,\ll\} $ and $ M:=\exp( \m) $ and let:
\rm   \begin{eqnarray*}
\si_{\nu Q^*}:= \ind_ L ^G \ch_{\nu Q^*}\simeq \ta_{\nu Q^*}:=\int_\R ^\oplus\pi_{x^* X^* +\nu Q^*}dx^*.
\end{eqnarray*}
An intertwining operator $ U_{\nu Q^*} $ for this equivalence is given by
\begin{eqnarray*}
( U_{\nu Q^*}(\xi)(x^*))(g):= \int_\R \xi(g\exp(x Y))e^{{-2\pi i x^*x}}dx \, \text{ for }\, \xi\in L^2(G/L, \ch_{\nu Q^*}), g\in G, x^*\in\R.
\end{eqnarray*}

The unitary mapping $ U_{\nu Q^*} $ induces an  isomorphism  $ \rh_{\nu Q^*} $ from the algebra $  B(L^2(G/L,\ch_{\nu Q^*})) $ onto $ B( \int_\R ^\oplus B(L^2(G/M,\ch_{x^*X^*+\nu Q^*}))dx^*)$. This homomorphism is defined by
\begin{eqnarray*}
 \rh_{\nu Q^*}(\si_{\nu Q^*}(a)) &= & U_{\nu Q^*}\circ \si_{\nu Q^*}(a)\circ U_{\nu Q^*}^*\\
&=&\int_\R^\oplus \pi_{x^*X^*+\nu Q^*}(a)dx^* \, \text{ for } \,  a\in C^*(G).
\end{eqnarray*}
\item  Let $ \k:=\textrm{span}\{X,\ll\} $ and $K:=\exp( \k) $ and let:
\rm   \begin{eqnarray*}
\si_{\ve P^*}:= \ind_ L ^G \ch_{\ve P^*}\simeq \ta_{\ve P^*}:=\int_\R ^\oplus\pi_{y^* Y^* +\ve P^*}dy^*.
\end{eqnarray*}
An intertwining operator $ U_{\ve P^*} $ for this equivalence is given by
\begin{eqnarray*}
( U_{\ve P^*}(\xi)(y^*))(g):= \int_\R \xi(g\exp(y X))e^{{-2\pi i y^*y}}dy \, \text{ for }\, \xi\in L^2(G/L, \ch_{\ve P^*}), g\in G, x^*\in\R.
\end{eqnarray*}

The unitary mapping $ U_{\ve P^*} $ induces an  isomorphism  $ \rh_{\ve P^*} $ from the algebra $  B(L^2(G/L,\ch_{\ve P^*})) $ onto $ B( \int_\R ^\oplus B(L^2(G/K,\ch_{y^*Y^*+\ve P^*}))dy^*)$. This homomorphism is defined by
\begin{eqnarray*}
 \rh_{\ve P^*}(\si_{\ve P^*}(a)) &= & U_{\ve P^*}\circ \si_{\ve P^*}(a)\circ U_{\ve P^*}^*\\
&=&\int_\R^\oplus \pi_{y^*Y^*+\ve P^*}(a)dy^*  \, \text{ for } \,  a\in C^*(G).
\end{eqnarray*}
\item  Let $\si_0:=\ind_L ^G \chi_0 $.
 \end{enumerate}

\end{definition}

\begin{definition}\label{PstQst}
\begin{enumerate}\label{}
\item Let $ S_{\de,1}:=\{(x,y); e^{-x}>\de^6,e^{{-y}}>\de^6\} $,
\item $ S_{\de,2}:=\{(x,y); e^{-x}>\de^6,e^{{-y}}\leq\de^6\} $,
\item $ S_{\de,3}:=\{(x,y); e^{-x}\leq\de^6,e^{{-y}}>\de^6\} $,
\item $ S_{\de,4}:=\{(x,y); e^{-x}\leq\de^6,e^{{-y}}\leq\de^6\} $,
\end{enumerate}
and as usual we denote by $ M_{\de,i} $ the multiplication operator on $ L^2(\R^2) $ with the function $1_{S_{\de,i}} $ for $1 \leq i \leq 4$.
\end{definition}

\begin{proposition}\label{PQone com}
For every $ a\in C^*(G) $, the operator $ \pi_{\ve P^*+\nu Q^*}(a)\circ M_{\de,1} $ is compact.
\end{proposition}
\begin{proof}
Let $ F\in L^1_c:=\{F\in L^1(G), \wh F^L\in C_c(G/L,\ll^*)\} $. The kernel function  $ F_{\ve,\nu} $ of the operator $\pi_{\ve P^*+\nu Q^*}(F)\circ M_{\de,1}$ is given by
\begin{eqnarray*}
F_{\ve,\nu}((s,t), (x,y))=(\wh F^L(s-x,t-y), e^{-x}P^*+e^{-y}Q^*)1_{S_{\de,1}}(x,y).
\end{eqnarray*}
Since $ F\in L^1_c $ is of compact support in all variables and therefore $ \pi_{\ve P^*+\nu Q^*}(F)\ci M_{\de,1} $ is Hilbert-Schmidt. Since $ L^1_c $ is dense in $ C^*(G) $, we have that $ \pi_{\ve P^*+\nu Q^*}(a)\circ M_{\de,1} $ is compact for every $ a\in C^*(G) $.
 \end{proof}

\begin{definition}\label{siPQ}
\rm
For $ F\in L^1_c $ and $ \de>0 $, let
\begin{eqnarray*}
\si_{\ve,\nu,\de}(F):= \si_{\nu Q^*}(F)\circ M_{\de,2}+\si_{\ve P^*}(F)\circ M_{\de,3}+\si_{0}(F)\circ M_{\de,4}\in \B(L^2(\R^2)).
\end{eqnarray*}
\end{definition}

The kernel function $ F_{\ve,\nu,\de}$ of this operator is given by
\begin{eqnarray*}\label{defkoPQ}
F_{\ve,\nu,\de}((s,t),(x,y))
 &=&\wh F^{L}((s-x,t-y), e^{-x}P^*)1_{S_{\de,2}}(x,y)\\
 && +\wh F^{L}((s-x,t-y),e^{-y}Q^*)1_{S_{\de,3}}(x,y)+\wh F^{L}((s-x,t-y),0)1_{S_{\de,4}}(x,y).
\end{eqnarray*}

Similar to previous cases, we have the following propositions and corollary.

\begin{proposition}\label{estas usalR}
For every $ F\in L^1_c $ and every $ \de>0 $,
\begin{eqnarray*}
\noop{\si_{\ve,\nu,\de}(F)}\leq  \noop{\si_{\ve P^*}(F)}+\noop{\si_{\nu Q^*}(F)}
+\noop{\si_{0}(F)}.
\end{eqnarray*}
\end{proposition}

\begin{proposition}\label{mde2issmr}
For any $ F\in L^1_c $, there exists  a constant $ K=K(F)>0 $ such that
\begin{eqnarray*}
 \noop{\pi_{\ve P^*+\nu Q^*}(F)\circ (1- M_{\de,1})-\si_{\ve,\nu,\de}(F)} \leq K\de
\quad \text{for} \quad \de>0.
\end{eqnarray*}
\end{proposition}
\begin{proof}
Let $ F\in L^1_c $. The kernel function $ F_{\ve,\nu,\de,0} $ of the operator
$ \pi_{\ve P^*+\nu Q^*}(F)\circ (1- M_{\de,1})-\si_{\ve,\nu,\de}(F) $ is given by:
\begin{eqnarray*}
&&F_{\ve,\nu,\de,0}((s,t),(x,y))\\
&=&(\wh F^L((s-x,t-y),e^{-x}P^*+e^{-y}Q^*)-\wh F^L((s-x,t-y),e^{-x}P^*))1_{S_{\de,2}}(x,y)\\
& &+ (\wh F^L((s-x,t-y),e^{-x}P^*+e^{-y}Q^*)-\wh F^L((s-x,t-y),e^{-y}Q^*))1_{S_{\de,3}}(x,y)\\
& &+ (\wh F^L((s-x,t-y),e^{-x}P^*+e^{-y}Q^*)-\wh F^L((s-x,t-y),0))1_{S_{\de,4}}(x,y).
\end{eqnarray*}
Since $ F\in L^1_c $, there exists a function $ \va\in C_c(\R^2) $ such that
\begin{eqnarray*}
\val{\wh F((s,t),q)}\leq \val{\va(s,t)}\no q \quad \text{for } \, q\in\ll^*, (s,t)\in\R^2.
\end{eqnarray*}
Therefore,
\begin{eqnarray*}
\val{F_{\ve,\nu,\de,0}((s,t),(x,y))}\leq 3\val{\va(s-x,t-y)}\de^6.
\end{eqnarray*}

\end{proof}

\begin{corollary}\label{zerobosRstar}
Let $ a\in C^*(G) $. Then
\begin{eqnarray*}
 \lim_{\de\to 0}\dis\big(\pi_{\ve P^*+\nu Q^*}(a)-\si_{\ve,\nu,\de}(a), \K\big)=0.
\end{eqnarray*}
\end{corollary}

\subsection{The boundary conditions for $ \R_{\ve P^*} $ and $\R_{\nu Q^*}$ }\label{ d2 }

We consider only the case $ \OM_{a^* A^*+\nu Q^*}$. The other cases are similar.

We take the coordinates on $ G $ coming from the basis $ \{A,B,P,Q,R,S\} $. Let
$L:=\exp(\ll)$ and $\ll:=\text{span}\{P,Q,R,S\}$.
Let $ \si_{0}:= \ind_ L^ G 1$ be the left regular representation of $ G $ on $L^2(G/L) $. This representation is equivalent to  $ \ta_{0}:=\oint_\R \pi_{a^* A^*+b^*B^*} da^* db^*$ acting on $ L^2(\R^2)=\oint_{R^2} \C da^*db^* $. Let $ U_{\nu Q^*} $ be the unitary equivalence given by
\begin{eqnarray*}
U_{\nu Q^*}(\xi)(a^*,b^*):=\int_{\R^2} \xi(\exp(a A)\exp(b B))e^{-2\pi i( a a^*+bb^*)}dadb,\ a^*,b^*\in \R.
\end{eqnarray*}
We see that the algebra $B(L^2(G/L,\ch_{\nu Q^*}))$ is mapped into $ C(\R^2)\subset \B(L^2(\R^2))  $ by the mapping
\begin{eqnarray*}
\rh_{\nu Q^*}(\xi):=U_{\nu Q^*}\circ \xi\circ U_{\nu Q^*}^*.
\end{eqnarray*}

\begin{definition}\label{sisdefp}
\rm For $ \de>0 $, let
\begin{enumerate}\label{}
\item $ S_{\de,1}:=\{(a,b)\in\R^2, e^{-b}>\de\} $,
\item $ S_{\de,2}:=\{(a,b)\in\R^2, e^{-b}\leq\de\} $,
\end{enumerate}
and $ M_{\de,i} $ denotes the multiplication operator on $ L^2(\R^2) $ with the function $1_{S_{\de,i}} $ for $i= 1, 2$.
For each $ F\in C^*(G)$, we define the linear operator $ \sigma_{\de,2}(F) $ on $ L^2(\R^2) $ by
\begin{eqnarray*}
 \si_{\de,2}(F):= \si_{0}(F)\circ M_{\de,2}.
\end{eqnarray*}
\end{definition}

\begin{proposition}\label{Rqstast}
Let $ a\in C^*(G) $. Then
\begin{eqnarray*}
 \lim_{\de\to 0}\dis((\si_{\nu Q^*}(a)\ci (1-M_{\de,1})-\si_{\de,2}(a)),C_0(\R,\K))=0 .
\end{eqnarray*}
\end{proposition}
The proof is similar to that of the corresponding proof in Section \ref{cora}.

\section{The C*-algebra of $ G $}\label{cstarGdef}

\rm   We recall that the space $\wh G$ has been identified in Definition \ref{defgais} with the disjoint union of the subset $\GA_j,j=0,\ldots, 6$, of $ \g^*$.
We let
\begin{eqnarray*}
 l^\iy(\wh G):= \{(\ph(\ga))_{\ga\in \wh G}; \ph(\ga)\in B(\H_{\pi_\ga}),
 \, \no{\ph}:=\sup_{\ga\in\wh G}\noop{\ph(\ga)}<\iy\}.
\end{eqnarray*}

\begin{definition}\label{resfi}
\rm   For  a subset $ \GA $ of $ \wh G $ and $ \ph\in l^\iy(\wh G) $, let
\begin{eqnarray*}
 \ph\res\GA := (\ph(\ga))_{\ga\in\GA}\in l^\iy(\GA).
\end{eqnarray*}
\end{definition}

\begin{definition}\label{deffourtr}
\rm  For $F\in C^*(G)$ and $\ga \in \wh G$, we define the Fourier transform $ \F_G: C^*(G)\to l^\iy(\wh G) $ by
\begin{equation}\label{fourdef}
\nonumber \F_G(F)(\ga):=
\pi_\ga(F)\in B(\H_{\pi_\ga}).
\end{equation}
Let
\begin{eqnarray*}
  E^*(G):= &\{\ph\in l^\iy(\wh G);  \text{ the mappings }\ga\mapsto \ph(\ga) \text{ are norm continuous and  }\\
&\quad \quad \textrm{vanish at infinity on the sets } \GA_i, i=1,2,4,5\}.
\end{eqnarray*}
\end{definition}

Then the  subset  $ E^*(G) $ is a C*-subalgebra of $ l^\iy(\wh G) $ containing $ \F_G(C^*(G)) $ by Section \ref{contcon}.
We can define the following representations of $ E^*(G) $ for $\ve= \pm 1$:
\begin{enumerate}\label{}
\item
$ \ta_{\ve S^*}(\ph):=\int_\R^\oplus \ph(b^* B^*+\ve S^*)db^*$ on the Hilbert space  $\int_\R^\oplus L^2(\R^2)db^* $. We have seen in Definition \ref{sisstdef} that the restriction of this representation to $ \F_G(C^*(G)) $ is equivalent to the representations $ \si_{\ve s^* S^*}=\ind_H^G \ch_{ \ve s^*S^*}$ for $s^*\in \R^+ $. We use now these equivalences to extend the representations $ \si_{\ve s^* S^*} $ to $ E^*(G) $.
\item
$\ta_{\ve R^*}(\ph):=\int_\R^\oplus \ph(b^*B^*+\ve R^*)db^*$ on the Hilbert space  $\int_\R^\oplus L^2(\R^2)db^* $. This gives us representations $ \si_{\ve R^*} $ of $ E^*(G) $ (which coincides with $\ind_{H}^G \ch_{\ve R^*} \textrm{ on }C^*(G)$).

\item
$\ta_{\ve P^*}(\ph):=\int_\R^\oplus \ph(x^*(\frac{A^*+B^*}{2})+ \varepsilon P^*)dx^*$ on the Hilbert space  $\int_\R^\oplus L^2(\R)db^* $. This gives us representations $ \si_{\ve P^*}$ of $ E^*(G) $ (which coincides with $\ind_{[G,G] }^G \ch_{\ve P^*} \textrm{ on }C^*(G))  $.
\item
$\ta_{\ve Q^*}(\ph):=\int_\R^\oplus \ph(a^* A^*+ \ve Q^*)da^*$ on the Hilbert space  $\int_\R^\oplus L^2(\R)da^* $. This gives us representations $ \si_{\ve Q^*}$ of $ E^*(G) $, which coincides with $\ind_{[G,G] }^G \ch_{\ve Q^*} \textrm{ on }C^*(G) )$.
\end{enumerate}

This allows us  to define the following bounded linear mappings of $ E^*(G) $, where $ \ve=\pm1,\nu= \pm1$ and $\de>0 $:
\begin{enumerate}\label{}
\item (with the notations in Section \ref{vempe})
\begin{eqnarray*}
 \si_{\ve S^*+\ve Q^*,\de}(\ph):=
 \si_{\ve S^*}(\ph)\circ M_{\tilde \va_1}\circ M_{\de,1}+\si_{\ve S^*}(\ph)\circ M_{\de,2}+\si_{\de,3}(\ph),
\end{eqnarray*}
where $\si_{\de,3}(\ph):=\sum_{k\in\Z}N_{\de,k,3}\circ \si_{\ve(1+\frac{\de^4 k^2}{2})Q^*}(\ph)\circ M_{\de,3,k}$;

\item (with the notations in Section \ref{vemve})
\begin{eqnarray*}
 \si_{\ve S^*-\ve Q^*,\de}(\ph):=
 \sum_{i=1}^{2}\si_{\ve S^*}(\ph)\circ M_{\tilde \va_i}\circ M_{\de,i}+ \si_{\pm\sqrt{2}R^*-2Q^*}(\ph)\circ M_{\de,4,\pm}+ \si_{\de,3}(\ph),
\end{eqnarray*}
where $\si_{\de,3}(\ph):=\sum_{k\in\Z}N_{\de,3,k}\circ \si_{\ve(-1+\frac{\de^4 k^2}{2})Q^*}(\ph)\circ (M_{\de,3,k,+}+M_{\de,3,k,-})$;

\item  (with the notations in Section \ref{vesbs})
\begin{eqnarray*}
 \si_{\ve S^*,\de,D}(\ph) := \sum_{k\in\Z^3}P_{\de,D,3,k}\circ \si_{k_3^2 \de^3 D^2 e^{r_{\de}(k_1+k_2)}\ve Q^*}(\ph)\circ M_{\de,D,3,k};
\end{eqnarray*}
\item  (with the  notations in Section \ref{cora})
\begin{eqnarray*}
 \si_{\ve R^*,\de,D}(\ph) := \sum_{k\in\Z^3}P_{\de,D,2,k}\circ \si_{k_3 \de^2 D e^{r_{\de}(k_1-k_2)}\ve Q^*}(\ph)\circ M_{\de,D,2,k};
\end{eqnarray*}
\item  (with the notations in Section \ref{dfoo})
\begin{eqnarray*}
 \si_{\ve P^*,\nu Q^*,\de}(\ph):= \si_{\nu Q^*}(\ph)\circ M_{\de,2}+\si_{\ve P^*}(\ph)\circ M_{\de,3}+\si_{0}(\ph)\circ M_{\de,4};
\end{eqnarray*}
\item similarly  (with the notations of Section \ref{ d2 })
\begin{eqnarray*}
\si_{\ve P^*,\de}(\ph) =\si_{\ve P^*}\circ M_{\de,2},
\end{eqnarray*}
and the corresponding linear mapping
\begin{eqnarray*}
\si_{\nu Q^*,\de}(\ph) =\si_{\nu Q^*}\circ M_{\de,2}.
\end{eqnarray*}
\end{enumerate}

We define now the C*-subalgebra $ D^*( G) $ of $ l^{\iy}(\wh G) $. It will follow from Theorem \ref{aisdsta} that this subalgebra is  the Fourier transform of $ C^*(G) $.

\begin{definition}\label{}
\rm   Let $ D^*( G)$ be the subset of $ l^{\iy}(\wh G) $ consisting of all operator fields $ \ph $ contained in $ E^*(G)$ such that for $ \ve,\om =\pm1 $.
\begin{enumerate}\label{}
\item $$\lim_{\de\to 0}\dis(\ph(\ve(S^*+Q^*))-\sigma_{\ve S^*+ \ve Q^*,\de} (\ph),\K(L^{2}(\R^3)))= 0.$$

\item $$\lim_{\de\to 0}\dis((\ph(\ve(S^*-Q^*))-\si_{\ve S^*-\ve Q^*,\de}(\ph),\K(L^2(\R^3)))=0.$$

\item $$\lim_{\de\to 0}\dis(\rh_{\ve S^*}(\si_{\ve S^*}(\ph)-\si_{\ve S^*,\de,D}(\ph)), C_0(\R,\K))= 0,$$ this is,
 $$\lim_{\de\to 0}\dis(\ta_{\ve S^*}(\ph)-\rh_{\ve S^*}(\si_{\ve S^*,\de,D}(\ph)), C_0(\R,\K))= 0.$$

\item $$\lim_{\de\to 0}\dis(\rh_{\ve R^*}(\si_{\ve R^*}(\ph)-\si_{\ve R^*,\de,D}(\ph)), C_0(\R,\K))= 0,$$ this is,
 $$\lim_{\de\to 0}\dis(\ta_{\ve R^*}(\ph)-\rh_{\ve R^*}(\si_{\ve R^*,\de,D}(\ph)), C_0(\R,\K))= 0.$$

\item $$\lim_{\de\to 0}\dis((\ph({\ve P^*+\nu Q^*})-\si_{\ve P^*+\nu Q^*,\de}(\ph)),\K(L^{2}(\R^3)))=0.$$

\item $$\lim_{\de\to 0}\dis(\rh_{\ve P^*}(\si_{\ve P^*}(\ph)-\si_{\ve P^*,\de}(\ph)), C_0(\R,\K))= 0,$$ this is,
$$\lim_{\de\to 0}\dis(\ta_{\ve P^*}(\ph)-\rh_{\ve P^*}(\si_{\ve P^*,\de}(\ph)), C_0(\R,\K))= 0.$$

\item $$\lim_{\de\to 0}\dis(\rh_{\nu Q^*}(\si_{\nu Q^*}(\ph)-\si_{\nu Q^*,\de}(\ph)), C_0(\R,\K))= 0,$$ this is,
$$\lim_{\de\to 0}\dis(\ta_{\nu Q^*}(\ph)-\rh_{\nu Q^*}(\si_{\nu Q^*,\de}(\ph))), C_0(\R,\K))= 0.$$

\item The same conditions hold for the adjoint field $ \ph^* $.
\end{enumerate}
\end{definition}

We have seen in the Sections 5 and 6 that we can use Theorem \ref{aisdsta} to show the following result.

\begin{theorem}\label{disg}
The C*-algebra of the group $ G_6 $ is an almost $ C_0(\K) $-C*-algebra. In particular, the Fourier transform maps $ C^*(G_6) $ onto the subalgebra $ D^*(G_6)\subset l^\iy(\wh {G_6}) $.
\end{theorem}

{{Junko Inoue}} {{Education Center,}}
{Organization for Supporting University Education, Tottori University},
 {Tottori 680-8550,  Japan}\\
{ E-mail: inoue@uec.tottori-u.ac.jp }

{{Ying-Fen Lin}} {{ Pure Mathematics Research Centre,}}
{{Queen's University Belfast},
 Belfast, BT7 1NN, U.K.}\\
{E-mail: y.lin@qub.ac.uk}

{{Jean Ludwig}} {{Universit\'e de Lorraine, Institut Elie Cartan de Lorraine,  UMR 7502, Metz, F-57045, France}}\\
{E-mail: jean.ludwig@univ-lorraine.fr}
 \vskip 0.5cm


\begin{thebibliography}{99}

\bibitem[Ar-Lu-Sc] {Ar-Lu-Sc}R. J. Archbold, J. Ludwig, G. Schlichting, Limit
sets and strengths of convergence for sequences in the duals of thread-like Lie
groups. {\em Math. Z.}  \textbf{255}  (2007),  no. 2, 245-282.



\bibitem [Di]{Di} J. Dixmier, {\em C*-algebras}. Translated from the French
by Francis Jellett. North-Holland Mathematical Library, Vol. 15.
North-Holland Publishing Co., Amsterdam-New York-Oxford, 1977.
xiii+492 pp.

\bibitem[Fe] {Fe} J. M. G. Fell, The structure of algebras of operator
fields, {\em Acta Math.}  \textbf{106}  (1961),  233-280.

\bibitem[In]{Ino} J. Inoue, Fourier transforms for affine automorphism groups on Siegel domain,
{\em J. Math. Soc. Japan} \textbf{43} (1991), 247-259.

\bibitem[Ka]{Kaneyuki}
S. Kaneyuki, {\it Homogeneous bounded domains and Siegel domains}, Lecture Notes in Mathematics, Vol. 241, Springer, Berlin, 1971.

\bibitem[Lee]{Lee} R-Y Lee, On the C* algebras of operator fields, {\em Indiana
Univ. Math. J.} {\bf 26} No. 2 (1977), 351-372.


\bibitem[Lep-Lud] {Lep-Lud} H. Leptin, J. Ludwig, {\em Unitary representation theory
of exponential Lie groups}, De Gruyter Expositions in Mathematics 18, 1994.

\bibitem[Lin-Lud1] {Lin-Lud1} Y.-F. Lin, J. Ludwig, The C*-algebras of
ax+b-like groups, {\em J. Funct. Anal.} {\bf 259} (2010), 104-130.

\bibitem[Lin-Lud2] {Lin-Lud2} Y.-F. Lin, J. Ludwig, An isomorphism between group
C*-algebras of $ax+b$-like groups, {\em Bull.\ London Math.\ Soc.} {\bf 45(2)} (2013), 257-267.

\bibitem [Lu-Tu] {Lu-Tu} J. Ludwig, L. Turowska, The C*-algebras of the
Heisenberg Group and of thread-like Lie groups,  \textit{Math. Z.} \textbf{268} (2011), no. 3-4, 897-930.

\bibitem [Tu] {Tu} P. Turkowski, Solvable Lie algebras of dimension six,
\textit{J. Math. Phys.} \textbf{31} (1990), 1344-1350.

\end{thebibliography}
\end{document}